\documentclass[11pt,oneside,english]{amsart}
\usepackage[T1]{fontenc}
\usepackage[latin9]{inputenc}
\usepackage{geometry}
\usepackage{amstext}
\usepackage{amsthm}
\usepackage{amssymb}

\makeatletter
\numberwithin{equation}{section}
\numberwithin{figure}{section}
\theoremstyle{plain}
\newtheorem{thm}{\protect\theoremname}[section]
\theoremstyle{plain}
\newtheorem{cor}[thm]{\protect\corollaryname}
\theoremstyle{remark}
\newtheorem{rem}[thm]{\protect\remarkname}
\theoremstyle{plain}
\newtheorem{prop}[thm]{\protect\propositionname}
\theoremstyle{plain}
\newtheorem{lem}[thm]{\protect\lemmaname}
\theoremstyle{definition}
\newtheorem{example}[thm]{\protect\examplename}
\theoremstyle{definition}
\newtheorem{defn}[thm]{\protect\definitionname}

\def\makebbb#1{
    \expandafter\gdef\csname#1\endcsname{
        \ensuremath{\Bbb{#1}}}
}
\makebbb{R}
\makebbb{N}
\makebbb{Z}
\makebbb{C}
\makebbb{G}
\makebbb{Q}
\makebbb{H}
\makebbb{P}
\makebbb{B}
\makebbb{K}
\makebbb{E}

\makeatother

\usepackage{babel}
\providecommand{\corollaryname}{Corollary}
\providecommand{\definitionname}{Definition}
\providecommand{\examplename}{Example}
\providecommand{\lemmaname}{Lemma}
\providecommand{\propositionname}{Proposition}
\providecommand{\remarkname}{Remark}
\providecommand{\theoremname}{Theorem}

\begin{document}
\title{The Sinkhorn algorithm, parabolic optimal transport and geometric
Monge-Ampère equations}
\author{Robert J. Berman}
\begin{abstract}
We show that the discrete Sinkhorn algorithm - as applied in the setting
of Optimal Transport on a compact manifold - converges to the solution
of a fully non-linear parabolic PDE of Monge-Ampère type, in a large-scale
limit. The latter evolution equation has previously appeared in different
contexts (e.g. on the torus it can be be identified with the Ricci
flow). This leads to algorithmic approximations of the potential of
the Optimal Transport map, as well as the Optimal Transport distance,
with explicit bounds on the arithmetic complexity of the construction
and the approximation errors. As applications we obtain explicit schemes
of nearly linear complexity, at each iteration, for optimal transport
on the torus and the two-sphere, as well as the far-field antenna
problem. Connections to Quasi-Monte Carlo methods are exploited.
\end{abstract}

\maketitle

\section{Introduction}

The theory of Optimal Transport \cite{v1,v2} is used in a multitude
of applications ranging from economics, statistics, cosmology, geometric
optics and meteorology to more recent applications in data sciences
(including machine learning, vision, graphics and imaging; see \cite{so,m-p,s-d---}).
In the last years there has been a flurry of numerical work applying
the Sinkhorn algorithm \cite{s} (aka the Iterative Proportional Fitting
Procedure \cite{r}) as a fast and efficient way of computing approximations
to optimal transport maps, or equivalently, solutions to certain geometric
Monge-Ampère type equations. This is motivated by applications to
machine learning \cite{cu} (concerning optimal transport in Euclidean
$\R^{n})$ and computer graphics and image processing \cite{s-d---}
(where the general setting of optimal transport on Riemannian manifolds
is considered). The key advantage of the Sinkhorn algorithm in this
context is its favorable large-scale computational properties (parallelization,
linear time convergence, etc \cite{b-g-c-N}).

The main aim of the present paper is to show that, in a large-scale
limit, the Sinkhorn  algorithm converges towards the solution of a
parabolic PDE of Monge-Ampère type, which, incidentally, previously
has appeared in \cite{s-s,k-s-w,ki} and is called the \emph{parabolic
optimal transport equation }in \cite{k-s-w}. The convergence is shown
with explicit error estimates. This leads, in particular, to the first
constructive approximation of the potential of the optimal transport
map with explicit bounds on the time-complexity of the construction
\emph{and} the approximation errors introduced by the discretization.

\subsection{Background and setup}

\subsubsection{\label{subsec:The-Sinkhorn-algorithm}The Sinkhorn iteration}

Let $p$ and $q$ be two vectors in $\R_{+}^{n}$ whose entries sum
to one. Given any matrix $K\in\R_{+}^{N}\times\R_{+}^{N}$ there exists,
by Sinkhorn's theorem \cite{s}, two diagonal positive matrices $D_{a}$
and $D_{b}$ with diagonal vectors $a$ and $b$ in $\R_{+}^{N}$
such that the matrix 
\begin{equation}
B:=D_{b}KD_{a}\label{eq:def of B}
\end{equation}
 has the property that the rows sum to $p$ and the columns sum to
$q.$ The diagonal matrices $D_{a}$ and $D_{b}$ are uniquely determined
up to scaling $D_{a}$ and $D_{b}^{-1}$ by the same positive number.
Moreover, $B$ can be obtained as the limit of the algorithm defined
by alternately normalizing the rows and columns of the matrix. In
other words,
\[
B=\lim_{m\rightarrow\infty}B(m),\,\,\,B(m)=D_{b(m)}KD_{a(m)},
\]
 where the pair of positive vectors $(a(m),b(m))$ are defined by
the following recursion, formulated in terms of matrix vector multiplications
and component-wise division of vectors:
\[
b(m)=p/K\cdot a(m)
\]
\[
a(m+1)=q/K^{T}\cdot b(m)
\]
with initial data $a(0)$ taken as the vector with entries $1.$ In
fact, any initial positive vector $a(0)$ will do and the vectors
$a(m)$ and $b(m)$ are convergent as $m\rightarrow\infty$ (without
any need of scaling) as follows from Theorem \ref{thm:conv of u m in general setting}.
The corresponding iteration 
\[
a(m+1):=q/K^{T}\cdot\left(p/K\cdot a(m)\right)
\]
will be called the Sinkhorn\emph{ iteration.} Its fixed point (uniquely
determined up to scaling) is thus the vector $a$ appearing in formula
\ref{eq:def of B}.

The same algorithm has appeared in various fields (economics, traffic
planning, statistics,...; see \cite{m-p}). In its most general (infinite
dimensional) form, known as the \emph{Iterative Proportional Fitting
Procedure} in the statistics literature, the roles of $p$ and $q$
are played by probability measures on two (possible non-finite) topological
spaces. In this setting the corresponding convergence of $B(m)$ towards
a limit $B$ was established in \cite{r} using a maximum entropy
characterization of $B,$ which in the discrete setting above says
that $B$ is the unique element realizing the infimum 
\[
\inf_{\gamma\in\Pi(p,q)}\mathcal{I}(\gamma|K)
\]
where $\mathcal{I}$ denotes the Kullback-Leibler divergence of $\gamma$
relative to $K,$ when $\gamma$ and $K$ are identified with measures
on the discrete product $\{1,...,N\}^{2}$ and $\Pi(p,q)$ denotes
the set of all matrices $\gamma$ in $\R_{+}^{N}\times\R_{+}^{N}$
with row sum $p$ and column sum $q$ (i.e. the corresponding measures
on $\{1,...,N\}^{2}$ have marginals $p$ and $q,$ respectively).
Since $-\mathcal{I}(\gamma|K)$ is the ``physical'' entropy of $\gamma$
relative to $K$ this can indeed be viewed as a maximum entropy characterization
of $B.$ An alternative proof of the convergence follows from Theorem
\ref{thm:conv of u m in general setting} below, which also shows
that $a(m)$ and $b(m)$ have unique limits (determined by the initial
value $a(0)).$

\subsubsection{\label{subsec:Discrete-optimal-transport}Discrete optimal transport}

\label{subsec:Now-replace-}Now replace $K$ with a family of matrices
$K_{\epsilon}$ of the form 
\[
(K_{\epsilon})_{ij}=e^{-\epsilon^{-1}C_{ij}},
\]
for a given matrix $C_{ij}$ and parametrized by a positive number
$\epsilon.$ Then the corresponding matrix $B_{\epsilon}\in\Pi(p,q)$
furnished by Sinkhorn's theorem converges, as $\epsilon\rightarrow0,$
to a matrix $B_{0}$ realizing the infimum 
\begin{equation}
\mathcal{C}:=\inf_{\gamma\in\Pi(p,q)}\left\langle C,\gamma\right\rangle .\label{eq:def of C intro}
\end{equation}
In the terminology of discrete optimal transport theory \cite{v1,m-p}
this means that $B_{0}$ is an \emph{optimal transport plan (coupling)
between $p$ and $q,$ with respect to the cost matrix $C.$} The
convergence follows from the maximum entropy characterization of $B_{\epsilon}$
(recalled in the previous section), which reveals that $B_{\epsilon}$
realizes the perturbed minimum 
\[
\mathcal{C}_{\epsilon}:=\inf_{\gamma\in\Pi(p,q)}\left\langle c,\gamma\right\rangle +\epsilon\mathcal{I}(\gamma|I).
\]
 While the matrix $B_{0}$ is sparse (and is typically supported on
the graph of a transport map) the approximation $B_{\epsilon}$ always
has full support (by Sinkhorn's theorem) and is thus more ``regular''
than $B_{0}.$ Accordingly, the small parameter $\epsilon$ is sometimes
referred to as the\emph{ entropic regularization parameter.} This
is illustrated by the simulations in \cite{b-g-c-N} for the case
when $p$ and $q$ represent the discretization of two probability
measures on the unit-interval in $\R,$ using a large number $N$
of points and with $C_{ij}$ the cost matrix defined by to the squared
distance function on $\R.$ When $\epsilon$ is taken to be of the
order $1/N$ \cite[Fig 1]{b-g-c-N} shows how the discrete probability
measures on $\R\times\R$ appear as smoothed out versions of the graph
of the corresponding optimal transport map. 

It should also be pointed out that the entropy minimization problem
above can be traced back to the work by Schrödinger on Quantum Mechanics
in the 30s \cite{schr} (see the survey \cite{le}, where the connection
to optimal transport is emphasized). 

\subsubsection{\label{subsec:Discretization-of-Optimal}Discretization of Optimal
Transport on the torus}

Let now $X$ be a compact manifold endowed with a cost function $c(x,y).$
To keep things as simple as possible we will start by taking the manifold
$X$ to be the $n-$dimensional torus 
\[
T^{n}:=(\frac{\R}{\Z})^{n}
\]
endowed with the standard distance function $d_{T^{n}}(x,y),$ induced
from the Euclidean distance function on $\R^{n}.$ Let $\mu$ and
$\nu$ be two probability measure on $T^{n}$ (which thus correspond
to two periodic measures on $\R^{n})$ with Hölder continuous and
strictly positive densities $e^{-f}$ and $e^{-g},$ respectively
\begin{equation}
\mu=e^{-f}dV,\,\,\,\nu=e^{-g}dV\label{eq:mu and nu in terms of f and g intro}
\end{equation}
where $dV$ is the Riemannian normalized volume form on $T^{n}.$
Define the\emph{ cost function} $c(x,y)$ by 
\[
c(x,y):=d_{T^{n}}(x,y)^{2}/2.
\]
As is well-known, a continuous self-map $F$ of $T^{n}$ transporting
(pushing forward) $\mu$ to $\nu$ is \emph{optimal} with respect
to this cost function, i.e. the corresponding transport plan $\gamma_{F}:=(I\times F)_{*}\mu$
realizes the infimum 
\begin{equation}
d^{2}(\mu,\nu)/2:=\inf_{\gamma\in\Pi(\mu,\nu)}\left\langle c,\gamma\right\rangle ,\label{eq:OT distance intro}
\end{equation}
if and only if $F$ can be expressed in terms of a \emph{potential}
$u\in C^{2}(T^{n}):$
\[
F(x):=x+\nabla u(x),\,\,\,T^{n}\rightarrow T^{n},
\]
which is strictly\emph{ quasi-convex} in the sense that the symmetric
matrix $\nabla^{2}u+I$ is positive definite:
\[
\nabla^{2}u+I>0
\]
(we identify $u$ with a $\Z^{n}-$periodic function on $\R^{n}$
so that $F(x)$ descends to define a self-map of $T^{n})$ . The function
$u$ is uniquely determined, up to an additive constant, by the following
Monge-Ampère equation 
\begin{equation}
\exp(-g(x+\nabla u(x))\det(I+\nabla^{2}u(x))=\exp(-f(x)).\label{eq:MA eq intro}
\end{equation}
We recall that the distance $d(\mu,\nu)$ between $\mu$ and $\nu,$
defined by formula \ref{eq:OT distance intro}, is usually called
the \emph{Wasserstein $L^{2}-$distance} or the \emph{Optimal Transport
distance. }

\subsection*{\label{subsec:Main-results-in-the torus}Main results in the torus
setting }

For notational reasons it will be convenient to express the entropic
regularization parameter as
\[
\epsilon:=k^{-1}
\]
 for $k$ a positive integer (but the results apply to any real parameter
$k\geq k_{0}>0).$ For each $k$ we fix a positive integer $N_{k}$
and denote by $\Lambda_{k}$ the corresponding discrete torus in $T^{n}$
with $N_{k}$ points. In other words, $\Lambda_{k}$ is the grid on
$T^{n}$ with edge-length $N_{k}^{-1/n}:$
\[
\Lambda_{k}:=(\frac{(N_{k}^{-1/n}\Z)}{\Z})^{n}\subset T^{n}.
\]
We will assume that 
\[
\lim_{k\rightarrow\infty}N_{k}=\infty.
\]
Denote by $p^{(k)}$ and $q^{(k)}$ the corresponding discrete approximations
in $\R^{N_{k}}$ of $\mu$ and $\nu,$ defined by the normalized values
of the densities of $\mu$ and $\nu$ at the points in $\Lambda_{k}.$
Defining a sequence of $N_{k}\times N_{k}$ matrices $K^{(k)}$ by
\begin{equation}
K_{ij}^{(k)}:=\mathcal{K}^{(k)}(x_{i}^{(k)},x_{j}^{(k)}),\,\,\,\mathcal{K}^{(k)}(x,y):=e^{-kd(x,y)^{2}/2}\label{eq:def of A k ij intro}
\end{equation}
and applying Sinkhorn's theorem to the triple $(K^{(k)},p^{(k)},q^{(k)})$
furnishes two positive vectors $a^{(k)}$ and $b^{(k)}$ in $\R^{N_{k}},$
uniquely determined by the normalization condition that $a_{i_{k}}^{(k)}=0$
for the index $i_{k}$ corresponding to the point $x_{i_{k}}=0$ in
$\Lambda_{k}.$ 

Our first result shows that the potential $u$ for the optimal transport
problem between $\mu$ and $\nu$ can be recovered from the positive
vectors $a^{(k)}$ and $b^{(k)},$ furnished by the Sinkhorn theorem,
i.e. the fixed points of the corresponding iteration:
\begin{thm}
\label{thm:conv in static torus setting intr}(Static case). If $x_{i_{k}}^{(k)}$
is a sequence of points in the discrete torus $\Lambda_{k}$ converging
to the point $x$ in the torus $T^{n},$ as $k\rightarrow\infty,$
then 
\[
-\lim_{k\rightarrow\infty}k^{-1}\log a_{i_{k}}^{(k)}=u(x)
\]
 where $u$ is the unique optimal transport potential solving the
Monge-Ampère equation \ref{eq:MA eq intro} and normalized by $u(0)=0.$ 
\end{thm}

This convergence result should come at no surprise and it holds in
a very general setting (see Theorems \ref{Thm:weak conv of fixed points towards optimal transport plans},
\ref{thm:Static conv in Riem setting} and \ref{thm:non-cpt static}).
But the main point of the present paper is that the Sinkhorn algorithm
itself, when viewed as a discrete dynamical system for the positive
vectors $a_{i_{k}}^{(k)},$ also admits a continuous large-scale limit
$u_{t}(x),$ evolving according to  the following fully non-linear
parabolic PDE:
\begin{equation}
\frac{\partial u_{t}(x)}{\partial t}=\log\det(I+\nabla^{2}u_{t}(x))-g(x+\nabla u_{t}(x))+f(x),\,\,\,u_{0}=0\label{eq:parabolic eq intro}
\end{equation}
 The existence of a $C^{4}-$smooth solution $u_{t}$ to this PDE,
given $f,g\in C^{2,\alpha}(T^{n}),$ essentially follows from the
results in \cite{s-s,ki,k-s-w} (for completeness a proof is provided
in Appendix B). In order to formulate the convergence in the next
theorem we first observe that for $m\geq1$ the function
\begin{equation}
u_{m}^{(k)}(x_{i})=-k^{-1}\log\frac{a_{i}^{(k)}(m)}{p_{i}}\label{eq:change of variables u a}
\end{equation}
on the discrete torus $\Lambda_{k}$ admits a canonical extension,
defining a quasi-convex function on $X:$ 
\begin{equation}
u_{m}^{(k)}(x):=k^{-1}\log\sum_{i=1}^{N_{k}}e^{-kd(x,y_{i}^{(k)})^{2}/2}b_{i}^{(k)}(m-1)\label{eq:canonical ext intro}
\end{equation}
expressed in terms of a Fourier/Gauss type sum, with $k$ playing
the role of the band-width.
\begin{thm}
\label{thm:conv in dynamic torus setting intro}(Dynamic case). Assume
that $f$ and $g$ are in $C^{2,\alpha}(T^{n})$ for some $\alpha>0$
and that $N_{k}=k^{n},$ i.e. the edge-length of the grid on $T^{n}$
is equal to $k^{-1}.$ For any sequence of discrete times $m_{k}$
(iterations) such that $m_{k}/k\rightarrow t$ we have 
\[
\lim_{k\rightarrow\infty}u_{m_{k}}^{(k)}=u_{t}
\]
 uniformly on $T^{n},$ where $u_{t}$ is the smooth and strictly
quasi-convex solution of the parabolic PDE \ref{eq:parabolic eq intro}
with initial data $u_{0}=0.$ More precisely, for any positive integers
$k$ and $m$
\[
\sup_{T^{n}}\left|u_{m}^{(k)}-u_{m/k}\right|\leq C\frac{m}{k}k^{-1},
\]
for a constant $C$ independent of $t.$ More generally, if at the
initial discrete time $m=0,$ 
\[
u_{0|\Lambda_{k}}^{(k)}=u_{0|\Lambda_{k}},
\]
for a given strictly quasi-convex function $u_{0}$ in $C^{4,\alpha}(T^{n}),$
then the corresponding result still holds (with a constant $C$ depending
on $u_{0}).$
\end{thm}

An immediate consequence is that if the number $m_{k}$ of iterations
is too small, of the order $o(k),$ then $u_{m_{k}}^{(k)}\rightarrow0$
under the Sinkhorn iterations, i.e. ``nothing happens''. As discussed
in Section \ref{subsec:Comparison-with-previous sinkhorn rate} the
estimate in the previous theorem can be expected to be sharp under
the regularity assumption in the theorem. Moreover, the proof of the
theorem yields an essentially explicit control on the constant $C$
appearing in the error bound.

Using that $u_{t}$ converges exponentially to a potential $u$ for
the optimal transport problem, as $t\rightarrow\infty$ (which follows
from results in \cite{k-s-w})  we deduce the following
\begin{cor}
\label{cor:conv of explicit apprl intro} (Constructive approximation
of the potential $u$). Assume that $f$ and $g$ are in $C^{2,\alpha}(T^{n})$
for some $\alpha>0$ and that $N_{k}=k^{n},$ i.e. the ``edge length''
of the grid on $T^{n}$ is equal to $k^{-1}.$ There exists a positive
constant $A_{0}$ such that for any $A>A_{0}$ the following holds:
after $m_{k}=\left\lfloor Ak\log k\right\rfloor $ iterations the
corresponding quasi-convex functions $u_{k}(x):=u_{m_{k}}^{(k)}(x)$
satisfy the estimate
\begin{equation}
\sup_{T^{n}}\left|u_{k}-u\right|\leq Ck^{-1}\log k,\label{eq:estimate in cor torus intro}
\end{equation}
for some constant $C$ (depending on $A),$ where $u$ is a potential
for the corresponding optimal transport map. Moreover, the discrete
probability measures $\gamma_{k}$ on $T^{n}\times T^{n},$ determined
by the Sinkhorn algorithm, converge weakly towards the corresponding
optimal transport plan $(I\times(\nabla u+I))_{*}\mu,$ concentrating
exponentially on the graph $\Gamma$ of the transport map $\nabla u+I:$
\begin{equation}
\gamma_{k}\leq pk^{p}e^{-kd_{\Gamma_{u}}^{2}/p}\delta_{\Lambda_{k}},\label{eq:estimate for gamma k cor torus intro}
\end{equation}
for some positive constant $p,$ where $d_{\Gamma}$ denotes the vertical
distance to the graph $\Gamma$ in $T^{n}\times T^{n},$ i.e. $d_{\Gamma}(x,y):=d_{T^{n}}(y,\nabla u(x)+x)$
and $\delta_{\Lambda_{k}}$ denotes the discrete uniform probability
measure on the discrete torus $\Lambda_{k}.$ 
\end{cor}

More generally, we will show in Section \ref{sec:Generalizations-and-outlook},
that if $f$ and $g$ are in $C^{\infty}(T^{n}),$ then the previous
theorem and its corollary still hold as long as the number of discretization
points in the grid satisfies
\[
N_{k}\geq C_{\delta}k^{n/2(1+\delta)},\,\,\,C_{\delta}>0
\]
 for some $\delta\in]0,1/2].$ In particular, this means that one
can then take a grid with larger edge-length 
\[
h=O(1/k^{1/2+\delta})
\]
 without affecting the quality of the approximation $u_{k}$, that
is, without affecting the order $O(k^{-1}\log k$) of the corresponding
error terms. In other words, for smooth data the error terms are nearly
of order $O(h^{2})$ if the entropic regularization parameter is taken
to be close to the square of the edge-length $h$ of the grid.

Since each iteration in the Sinkhorn algorithm may be formulated in
terms of matrix-vector multiplication, which requires $O(N^{2})$
arithmetic operators, and since the direct construction of the function
$u_{k}$ uses, in general, $O\left(kN_{k}^{2}\log k\right)$ elementary
arithmetic operations. Moreover, in the present case of the torus
the matrix-vector operations in question are discrete convolutions
and can thus be performed using merely $O\left(N\log N\right)$ arithmetic
operations, by using the Fast Fourier Transform (or $O(N^{1+1/n})$
operations, using separability; see Section \ref{subsec:Optimal-transport-on linear}).
Thus the construction of $u_{k}$ in the previous corollary requires
merely $O\left(kN_{k}(\log N_{k})\log k\right)$ arithmetic operations
to obtain on error term of order $O(k^{-1}\log k).$ The same number
of arithmetic operations (modulo a negligible term $O(N_{k})$ needed
to for the summing) yields the following constructive approximation
of the squared Optimal Transport distance $d(\mu,\nu)$ (formula \ref{eq:OT distance intro})
with an additive error of the order $O(k^{-1}\log k):$ 
\begin{cor}
\label{cor:cons transport cost}(Constructive approximation of $d(\mu,\nu)^{2}$).
Assume that $f$ and $g$ are in $C^{2,\alpha}(T^{n})$ for some $\alpha>0.$
There exists a positive constant $A$ such that after $m_{k}=\left\lfloor Ak\log k\right\rfloor $
iterations 
\[
\frac{1}{2}d(\mu,\nu)^{2}=\frac{1}{k}\sum_{i=1}^{N_{k}}p_{i}^{(k)}\log a_{i}^{(k)}(m_{k})+\frac{1}{k}\sum_{i=1}^{N_{k}}q_{i}^{(k)}\log b_{i}^{(k)}(m_{k})+O(\frac{1}{k}\log k)
\]
and, as a consequence, such an equality also holds for the squared
discrete Optimal Transport distance $d(\mu^{(k)},\nu^{(k)})^{2}$
(defined by formula \ref{eq:def of C intro} with $C_{ij}=d(x_{i},x_{j})^{2}/2$). 
\end{cor}

Finally we point out that, by symmetry, Theorem \ref{thm:conv in dynamic torus setting intro}
also shows that, on the one hand, the functions 
\begin{equation}
v_{m_{k}}^{(k)}(x_{i}):=-k^{-1}\log\frac{b_{i}^{(k)}(m_{k})}{q_{i}}\label{eq:chang of variables v intro}
\end{equation}
converge, as $k\rightarrow\infty,$ towards the solution $v_{t}$
of the parabolic equation obtained by interchanging the roles of $\mu$
and $\nu.$ On the other hand, by Lemma \ref{lem:dens prop}, the
function $v_{m_{k}}^{(k)}$ is equal to the Legendre transform (in
the space variable) of $u_{m_{k}}^{(k)},$ up to a negligible $O\left(k^{-1}\log k\right)$
error term. Thus Theorem \ref{thm:conv in dynamic torus setting intro}
is consistent, as it must, with the fact that the Legendre transform
of the solution $u_{t}$ of equation \ref{eq:parabolic eq intro}
solves the parabolic equation obtained by interchanging the roles
of $\mu$ and $\nu$ (as can be checked by a direct calculation). 

\subsection{Generalizations }

\subsubsection{The static case}

The result in the static setting is shown to hold in a very general
setting of optimal transport between two probability measures $\mu$
and $\nu$ defined on compact topological spaces $X$ and $Y,$ respectively:
see Theorems \ref{Thm:weak conv of fixed points towards optimal transport plans},
\ref{thm:Static conv in Riem setting} and also Theorem \ref{thm:non-cpt static}
which, in particular, applies to the classical Euclidean setting where
$X$ and $Y$ are convex domains in $\R^{n}$ and $c(x,y):=-x\cdot y.$
In the case when $Y$ is convex the corresponding limit $\phi$ obtained
from the Sinkhorn iteration is the unique convex normalized solution
to the second boundary value problem for Monge-Ampère equation in
the interior of $X:$ 
\begin{equation}
e^{-g(\nabla\phi)}\det(\nabla^{2}\phi)dx=\mu,\,\,\,(\partial\phi)(X)\subset Y.\label{eq:sec bv pr intro}
\end{equation}
The result holds without any regularity assumptions on $\mu$ and
$g$ if Alexandrov's classical notion of a Monge-Ampère measure is
employed (see Section \ref{subsec:The-non-compact-setting}). Moreover,
the convergence holds even when the target set $Y$ is not convex,
but then $(\partial\phi)(X)$ is contained in the convex hull of $Y$
(see the discussion in Section \ref{subsec:The-non-compact-setting}).

In the general setting the roles of the positive vectors $p^{(k)}$
and $q^{(k)},$ discretizing $\mu$ and $\nu,$ are played by two
sequences $\mu^{(k)}$ and $\nu^{(k)},$ satisfying certain density
properties with respect to $\mu$ and $\nu$ (which are almost always
satisfied in practice). Moreover, the cost function $c$ is merely
assumed to be continuous and can even be replaced by any sequence
$c_{k}$ converging uniformly to $c,$ as the inverse $k$ of the
entropic regularization parameter tends to infinity (which applies,
in particular, to the convolutional Wasserstein distances introduced
in \cite{s-d---}, where the matrix in formula \ref{eq:def of A k ij intro}
is replaced by a heat kernel; see Section \ref{subsec:Application-to-convolutional}). 

\subsubsection{The dynamic case}

The corresponding result in the dynamic setting (Theorem \ref{thm:dynamic general})
requires that $X$ and $Y$ be compact Riemannian manifolds and further
regularity assumptions on $c,$ $\mu$ and $\nu$ (the case when $X$
and $Y$ have boundaries is left for the future). A ``local density
property'' on the approximations $\mu^{(k)}$ and $\nu^{(k)}$ is
required, roughly saying that the approximations of $\mu$ and $\nu$
hold up to length scales of the order $k^{-1/2}$, with an $O(1/k)-$error
term. Interestingly, the local density properties turn out to be satisfied
when the approximations $\mu^{(k)}$ and $\nu^{(k)}$ are defined
by (weighted) point clouds, generated using \emph{Quasi-Monte Carlo
}methods for numerical integration \cite{b-e-g,b-s-s-w,b-c-c-g-s-t}.
The general results are applied to the case of the round \emph{two-sphere}
endowed with the two different cost functions: (1) $d(x,y)^{2}$ and
(2) $-\log|x-y|.$ These two cases appear, for example, in applications
to (1) computer graphics (texture mapping), medical imaging \cite{d-t,z et al},
mesh adaptation for global whether and climate prediction \cite{w-b-b-c}
and (2) the reflector antenna problem in geometric optics \cite{waII,g-o}.
Nearly linear complexity of the corresponding Sinkhorn iteration is
achieved in both cases, using fast transforms and $O(N^{3/2})-$complexity
using separability.

\subsection{Relation to previous results}

To the best of the authors knowledge these are the first convergence
results concerning the Sinkhorn algorithm (and its fixed points) in
the limit when the number $N$ of points and the inverse of the regularization
parameter $\epsilon$ \emph{jointly} tend to infinity (see \cite{c-d-p-s}
and references therein for the static case when only $\epsilon^{-1}(=k)$
tends to infinity in the Euclidean $\R^{n}-$ setting and \cite{le1,le}
for a very general setting). This kind of joint limit is, in practice,
what is studied in numerical simulations in the context of geometric
optimal transport (see for example \cite{s-d---} and the in-depth
study in \cite{sc,fe} on CPU and GPU hardware, respectively. Thus
the convergence analysis in the present paper provides a theoretical
bases for these simulations and yields concrete rates, under appropriate
regularity assumptions (see Section \ref{subsec:Comparison-with-previous sinkhorn rate}
for a comparison with previous rates for the Sinkhorn algorithm).
In particular, the present results rigorously establish and quantify
the experimental observations in \cite[Fig 2]{sc}, that the Sinkhorn
algorithm converges after essentially $O(\epsilon^{-1})$ iterations.
Moreover, the experimental findings in \cite[Fig 2]{sc} that $\epsilon$
can be taken to be close to the order $O(h^{2})$ on a grid with ``edge
length'' $h$ are confirmed, if the data is assumed to be $C^{\infty}-$smooth
(note that the data in \cite[Fig 2]{sc} is even real-analytic).

It should also be pointed out that the ``change of variables'' in
formula \ref{eq:change of variables u a} which plays an important
role in the present paper, is also crucial for numerical simulations
as it ensures numerical stability, as emphasized in \cite{m-p,sc,fe}.
As stressed in \cite{fe} the corresponding log-sum-exp KeOps routines
\cite{c-f-g}, used in \cite{fe} to implement the iteration $u_{m}\rightarrow u_{m+1}$
on GPU hardware, are just as efficient as matrix-vector products with
the kernel matrix $K_{ij}.$ Moreover, the present results provide
a theoretical justification for the kernel truncations employed in
\cite{sc,fe}. To briefly explain this we recall that the starting
point of the stabilization scheme advocated in \cite{sc} (see also
\cite[Remark 4.22]{m-p}) is to write the iteration in terms of the
variables $(u_{m+1}-u_{m},v_{m+1}-v_{m}).$ In the setting of a general
cost function $c(x,y)$ this corresponds to replacing the matrix kernel
$e^{-kc(x,y)}$ with the ``stabilized'' kernel
\[
e^{-kc(x,y)}e^{-ku_{m-1}(x)}e^{-kv_{m-1}(y)}
\]
By Theorems \ref{thm:conv in dynamic torus setting intro}, \ref{thm:dynamic general}
(and the argument in the proof \ref{eq:estimate for gamma k cor torus intro})
the latter kernel is, for $k$ large, exponentially concentrated on
the graph of $\Gamma_{t}$ in $X\times Y$ of the diffeomorphism $F_{u_{t}}$
corresponding to the parabolic solution $u_{t}$ for $t=(m-1)/k.$
This means that the corresponding $N_{k}\times N_{k}$ matrices $\left(e^{-kc(x_{i},y_{j})}e^{-ku_{m-1}(x_{i})}e^{-kv_{m-1}(y_{j})}\right)$
are effectively sparse. Building on the results in the present paper
this is exploited in \cite{b-m} to modify the Sinkhorn algorithm
in order to obtain a numerically stable algorithm on the torus, which
is shown to have nearly $O(N)-$complexity at each iteration.

\subsubsection{Comparison with other numerical schemes for Optimal Transport}

The quantitative convergence in Corollary \ref{cor:conv of explicit apprl intro}
should be compared with previous results in the rapidly growing literature
on numerical approximations schemes for solutions to Optimal Transport
problems and the corresponding Monge-Ampère type equations. However,
the author is not aware of any previous results providing both complexity
bounds (in terms of $N)$ \emph{and} a quantified rate of convergence
of the error of the approximate solution, as $N\rightarrow\infty.$
Recall that a time-honored approach is to approximate the optimal
transport potential $u$ by solving the linear program which is dual
to the discretized optimal transport problem. This can be done using
combinatorial algorithms. However, they do not scale well for large
$N$ (see, for example, the exposition in \cite{Br-F-H-L...}, where
applications to cosmology are given in the periodic setting). Moreover,
it is not clear how to establish quantitative convergence rates for
the convergence towards $u.$ Another influential approach is the
Benamou-Brenier Augmented Lagrangian approach, using computational
fluid mechanics, introduced in \cite{b-b} in the periodic setting
(numerical experiments suggest that it has $O(N^{3})$ time-complexity,
as pointed out in \cite{b-d}). There is also a rapidly expanding
literature concerning other discretization schemes in the field of
numerical analysis of PDEs, mainly concerning the case of domains
in $\R^{n}$ (as in Section \ref{subsec:The-non-compact-setting})
and the periodic case of the torus (as in Section \ref{subsec:Main-results-in-the torus}).
Here we will focus on the schemes where convergence results - analogous
to Theorems \ref{thm:non-cpt static}, \ref{thm:conv in static torus setting intr}
-  have been established. The mostly studied approaches are based
on either \emph{finite differences \cite{b-f-o,b-d} }or \emph{semi-discrete
Optimal Transport \cite{k-m-t} }(which goes back to the classical
work of Alexandrov and Pogorelov). Then\emph{ }the fully non-linear
Monge-Ampère equation for the potential $u$ (and the corresponding
boundary conditions \ref{eq:sec bv pr intro} or periodicity conditions)
is replaced with a finite dimensional non-linear algebraic equation
for a ``discrete'' function $u_{h},$ where $h$ denotes the corresponding
spatial resolution. In our torus setting we thus have $h=1/k$ (defined
as the entropic regularization paremeter) when the data is $C^{2,\alpha}$
smooth and for $C^{\infty}-$data $h$ can be taken arbitrarily close
to $1/k^{1/2}.$ The equation for $u_{h}$ may be expressed as the
fixed point condition for a non-linear map $S_{h}:$
\[
S_{h}(u_{h})=u_{h},
\]
 (whose role in our setting is played by the scaled logarithm of the
Sinkhorn operator defined by formula \ref{eq:scaled log Sinkhor op}).
In practice, the map $S_{h}$ is often taken to be a Newton type iteration.
In experiments it has been computed in almost linear time-complexity
$O(N_{h})$\emph{ \cite{b-f-o,b-d}.} Moreover, merely a few dozen
iterations $S_{h}^{(m_{h})}(u_{0})$ appear to be needed in order
to obtain a good approximation of $u_{h}.$ In the semi-discrete approach
the convergence of $u_{h}$ towards $u,$ as $h\rightarrow0,$ follows
directly from basic stability results for optimal transport plans.
Moreover, in the case of finite difference approach the convergence
is established in \cite{b-d} in the setting of domains. Another discretization
of the Optimal Transport problem in domains, called the \emph{logarithmic
discrete Monge\textendash Ampère optimization problem,} is introduced
in \cite{l-r} and the corresponding solution $u_{h}$ is shown to
converge towards $u.$ Very recently, a general convergence framework
is introduced in \cite{f h}, showing how to slightly modify a range
of existing numerical schemes (including \cite{b-f-o}) to establish
the convergence of the corresponding discrete solutions $u_{h}$ towards
$u$ in the setting of domains.

However, the problem of establishing convergence \emph{rates} as $h\rightarrow0$
is still open in these schemes (see the discussion in the survey \cite{f-g-n}).
From this point of view one of the main points of the present paper
is to rigorously quantify how many iterations $m_{h}$ are needed
in the setting of the Sinkhorn iteration in order to approximate the
solution $u$ to nearly order $O(h),$ for $C^{2,\alpha}-$data (and
nearly order $O(h^{2})$ for $C^{\infty}-$data). The answer is that
nearly $O(h^{-1})$ iterations are needed, according to Corollary
\ref{cor:conv of explicit apprl intro} (see the discussion in Section
\ref{subsec:Comparison-with-previous sinkhorn rate}). This is in
line with the experimental findings in \cite{sc,fe} (see, in particular,
\cite[Figure 3.19]{fe} where $\epsilon:=k^{-1}=10^{-4}$ and \cite[Fig 2]{sc}). 
\begin{rem}
Interestingly, as discussed in detail in \cite{sc,fe}, heuristics
inspired by simulated annealing in numerics (aka $\epsilon-$scaling)
and multi-scale techniques can be used to effectively reduce the of
number of Sinkhorn iterations needed to approximate the optimal transport
potential $u$ to a few dozen. In a nutshell, the idea is to fine-tune
and gradually decrease both the parameter $\epsilon$ and the spatial
resolution $h$ during the Sinkhorn iterations (see \cite[Figure 3.26]{fe}).
One can anticipate that variants of Theorems \ref{thm:conv in dynamic torus setting intro},
\ref{thm:dynamic general} will play a role in making these heuristics
and experimental findings mathematically rigorous with quantitative
error estimates. The local density property in Definition \ref{def:local dens}
should be useful in this regard (note that the length scale $k^{-1/2}(=\epsilon^{1/2})$
appearing in Definition \ref{def:local dens} corresponds to the notion
of blurring scale in \cite{fe}). On the other hand, as discussed
in Section \ref{sec:Outlook}, it is of independent interest to be
able to approximate the solution $u_{t}$ of the parabolic optimal
transport equations at any \emph{finite} time $t$ and then $\epsilon$
must be kept fixed during the iterations (so that an $O(\epsilon)-$approximation
of $u_{t}$ is obtained after $O(\epsilon^{-1}t)$ iterations, if
logarithmic factors are ignored).
\end{rem}

In the setting of semi-discrete optimal transport the convergence
of a damped Newton iteration towards $u_{h}$ is established in \cite{k-m-t}
(under similar assumptions as in Theorem \ref{cor:constr of approx in quite general setting })
at a linear rate (in the exponential sense). Furthermore, rates of
convergence of $u_{h}$ towards $u$ are established in \cite{ber2}.
However, the damping parameter and the rate established in \cite{k-m-t}
depends on the discrete solution $u_{h}$ (and hence on $h)$ in a
rather complicated way and degenerates as $N_{h}$ is increased, i.e.
when $h$ is decreased; see \cite[Remark 1.3]{k-m-t} and compare
also with the discussion in Section \ref{subsec:Comparison-with-previous sinkhorn rate}.
This means that, for the moment, the analog of Corollary \ref{cor:conv of explicit apprl intro}
in the semi-discrete setting seems ot be out of reach.

A damped Newton approach has also been applied directly to the Monge-Ampère
equation and shown to converge to $u$ at a linear rate in the periodic
setting in \cite{lo-r} when the target density is constant and then
in \cite{s-a-k} in the general periodic setting. In these approaches
the iteration $u^{(m+1)}$ is obtained by inverting a second order
linear elliptic operator (in non-divergence form) depending on $u^{(m)}.$
Various discretizations of these schemes are studied experimentally
in \cite{lo-r,s-a-k,k-l-p,w-b-b-c}.

One advantage of the Sinkhorn framework over many other approaches,
when applied to general manifolds, is that it is meshfree. In other
words, it does not require generating a grid or a polyhedral tessellation
of the manifolds, but only a suitable point cloud, which can be efficiently
generated using Quasi-Monte Carlo methods. In the case of the round
sphere various different numerical algorithms have previously been
explored in the literature: see \cite{d-t,z et al,w-b-b-c} for experimental
work on the case of the cost function $d(x,y)^{2}$ and \cite{c-k-o,de-m-t}
for the case of the cost function $-\log|x-y|$, as applied to the
reflector antenna problem in geometric optics. 

\subsubsection{Kähler geometry}

The present results are very much inspired by an analogous setup which
appears in complex (Kähler) geometry. Briefly, the role of the Sinkhorn
algorithm is then played by Donaldson's iteration, introduced in \cite{do},
whose fixed points are called balanced metrics and $k$ appears as
the power of a given ample line bundle $L$ over $X$ with $e^{-ku}$
playing the role of a Hermitian metric on $L.$ Moreover, the role
of the function $\rho_{ku}$ (formula \ref{eq:explicit expres for pho})
is played by the (normalized) point-wise norm of the Bergman kernel
on the diagonal, induced by the pair $(u,\mu).$ From this point of
view the ``static case'' of Theorem \ref{thm:conv in static torus setting intr}
(and its generalization Theorem \ref{Thm:weak conv of fixed points towards optimal transport plans})
is the analog of \cite[Thm B]{bbgz} and the ``dynamic case'' of
Theorem \ref{thm:conv in dynamic torus setting intro} is the analog
of the result in \cite{Ber} showing that Donaldson's iteration converges
to the Kähler-Ricci flow \cite{ca}, as conjectured in \cite{do}.
In fact, identifying the real torus $T^{n}$ with a reduction of the
complex torus $X:=\C^{n}/(\Z^{n}+i\Z^{n})$ the parabolic flow \ref{eq:parabolic eq intro}
can, in the case when $g$ is constant be identified with a twisted
Kähler-Ricci flow \cite{ca} whose stationary solutions are Kähler
potentials solving the corresponding complex Monge-Ampère equation
(known as the Calabi-Yau equation in this context).\footnote{When $f$ and $g$ are constant the corresponding twisted Kähler-Ricci
flow coincides with Hamilton's Ricci flow restricted to the space
of Kähler metrics. } But is should be stressed that a new feature of the analysis in the
present paper, compared to the usual situation in Kähler geometry
(apart from allowing a non-uniform target measure $\nu,$ i.e. a non-constant
$g)$ is that the source measure $\mu$ is taken to be\emph{ discrete}
and depend on $k,$ i.e. it is given by a discrete sequence $\mu^{(k)}.$
In practice, such discretizations are used in implementations of Donaldson's
iteration, such as the experimental work \cite{d-k-l-r}, motivated
by String Theory. The discreteness of $\mu^{(k)}$ leads to various
technical complications, that do not seem to have been studied rigorously
in the Kähler geometry setting. Interestingly, the density condition
on the sequence $\mu^{(k)}$ appearing in Lemma \ref{lem:dens prop}
can be viewed as a real analog of the Bernstein-Markov property for
a sequence $\mu^{(k)},$ as studied in the complex geometric and pluripotential
theoretic setting (see the discussion on page 8 in \cite{b-b-w}).
The relations between the real and complex settings will be expanded
on elsewhere.

\subsection{Acknowledgments}

I am grateful to Klas Modin for many discussions and, in particular,
for drawing my attention to the recent numerical works on the applications
of the Sinkhorn algorithm to Optimal Transport, which was the starting
point of the present paper. Also thanks to Gabriel Peyré and Jean-David
Benamou for helpful comments and discussions, Jean Feydy for sharing
his thesis draft and the two referees for their comments that greatly
improved the exposition. Lacking background in Numerics and Optimal
Transport I apologize for any omission in accrediting results properly.
This work was supported by grants from the ERC, the KAW foundation
and the Göran Gustafsson foundation. 

\subsection{Organization}

In section \ref{sec:General-setup-and} a general setting for iterations
on $C(X),$ generalizing the Sinkhorn algorithm, is introduced. The
iteration in question, which is determined by a triple $(\mu,\nu,c),$
is essentially equivalent to the Iterative Proportional Fitting Procedure
and the results in Section \ref{sec:General-setup-and} are probably
more or less well-known (except perhaps Theorem \ref{thm:conv of u m in general setting}).
But one point of the presentation is to exploit the variational structure.
It can be viewed as a real analogue of the formalism introduced in
\cite{Ber}, in the setting of Donaldson's iteration \cite{do} and
it lends itself to various generalizations of the optimal transport
problem (such as Monge-Ampère equations with exponential non-linearities).
In section \ref{sec:Proofs-of-the} the variational structure is used
to give a general convergence result for the Sinkhorn fixed points,
which when specialized to the torus setting yields Theorem \ref{thm:conv in static torus setting intr}.
Applications to the second boundary value problem for the Monge-Ampère
operator in $\R^{n}$ and to convolutional Wasserstein distances are
also given. Then in Section \ref{sec:Convergence-of-the iter torus}
the convergence of the Sinkhorn iteration towards the parabolic Optimal
Transport equations on the torus is shown (Theorem \ref{thm:conv in dynamic torus setting intro}).
The proof leverages the regularity theory and a priori estimates of
the corresponding parabolic PDE on the torus (shown in Appendix B).
In the following Section \ref{sec:Generalizations-and-outlook} the
result is generalized to a rather general setting of optimal transport
on compact manifolds. In Section \ref{sec:Nearly-linear-complexity}
it is shown that nearly linear complexity can be achieved in the case
of optimal transport on the torus and the sphere (which applies, in
particular, to the reflector antenna problem). Section \ref{sec:Outlook}
gives an outlook on relations to singularity formation in the parabolic
optimal transport equations. In particular, the variational approach
to the Sinkhorn iteration introduced in Section \ref{sec:General-setup-and}
is exploited in order to propose a generalized notion of solution
to the corresponding parabolic PDE. In Appendix A a proof of a discrete
version of the classical stationary phase approximation is provided.

\section{\label{sec:General-setup-and}General setup and preliminaries}

If $Z$ is a compact topological space then we will denote by $C(Z)$
the space of continuous functions on $Z$ endowed with the sup-norm
and by $\mathcal{P}(Z)$ the space of all (Borel) probability measures
on $Z,$ endowed with the weak topology. Given a subset $S$ of $Z$
we will denote by $\chi{}_{S}$ the function which is equal to $0$
on $S$ and infinity otherwise.

Throughout the paper we will assume given a triple $(\mu,\nu,c)$
where $\mu$ and $\nu$ are probability measures on compact topological
spaces $X$ and $Y,$ respectively and a function $c$ on $X\times Y.$
The function $c$ will be assumed to be continuous in all sections
except in Section \ref{sec:Generalizations-and-outlook} (where we
assume that $c$ is lower semi-continuous). The supports of $\mu$
and $\nu$ will be denote by $X_{\mu}$ and $Y_{\nu},$ respectively.
Given $u\in C(X)$ and $v\in C(Y)$ we will, abusing notation slightly,
identify $u$ and $v$ with their pull-backs to $X\times Y.$ 

\subsection{\label{subsec:Recap-of-Optimal}Recap of Optimal Transport and the
$c-$Legendre transform }

Let us start by recalling the standard setup for optimal transport
(see the book \cite{v1} for further background). A probability measure
$\gamma\in\mathcal{P}(X\times Y)$ is said to be a \emph{transport
plan (or coupling) between $\mu$ and $\nu$ }if the the push forwards
of $\gamma$ to $X$ and $Y$ are equal to $\mu$ and $\nu,$ respectively.
The subspace of all such probability measures in $\mathcal{P}(X\times Y)$
will be denote by $\Pi(\mu,\nu).$ A transport plan in $\Pi(\mu,\nu)$
is said to be\emph{ optimal wrt the cost function $c,$ }if it realizes
the following infimum: 
\[
\inf_{\gamma\in\Pi(\mu,\nu)}\int_{X\times Y}c\gamma
\]
 By weak compactness such an optimal transport plan always exists.
The \emph{$c-$Legendre transform $u^{c}$ }of a function $u\in C(X)$
is defined as the following function in $C(Y)$ 
\[
u^{c}(y):=\sup_{x\in X}\left(-c(x,y)-u(x)\right).
\]
Similarly, if $v\in C(Y)$ then $v^{c}$ is the function on $C(X)$
defined by replacing $u$ in the previous formula with $v$ and taking
the sup over $Y.$ A function $u\in C(X)$ is said to be \emph{$c-$convex}
if 
\[
(u^{c})^{c}=u
\]
Equivalently, $u$ is $c-$convex iff there exists some $v\in C(Y)$
such that $u=v^{c}.$ Indeed, this follows from the observation that
$u^{ccc}=u^{c}$ for any $u\in C(X),$ which in turn follows from
$u^{cc}\leq u.$ The following functional on $C(X)$ will be called
the \emph{Kantorovich functional}: 
\begin{equation}
J(u):=\int u\mu+\int u^{c}\nu.\label{eq:Kant functional}
\end{equation}

\begin{prop}
\label{prop:optim crit}(``Optimality criterion'') A transport plan
$\gamma\in\Pi(\mu,\nu)$ is optimal iff there exists $u\in C(X)$
which is $c-$convex and such that $\gamma$ is supported in 
\begin{equation}
\Gamma_{u}:=\left\{ (x,y)\in X\times Y:\,u(x)+u^{c}(y)+c(x,y)=0\right\} \label{eq:def of Gamma u}
\end{equation}
Moreover, if this is the case then 
\[
\int_{X\times Y}c\gamma=-J(u)
\]
\end{prop}

\begin{proof}
This is standard and known as the Knott-Smith optimality criterion
(in the Euclidean setting) \cite{v1}. For completeness we provide
the simple proof of the direction that we shall use later on. Since
$u+u^{c}+c\geq0$ on $X\times Y$ the following lower bound holds
for any given $\gamma\in\Pi(\mu,\nu)$
\begin{equation}
\int_{X\times Y}c\gamma\geq-\inf_{u\in C(X)}J(u)\label{eq:lower bd in pf prop opt crit}
\end{equation}
 Now, if $\gamma$ is supported in $\Gamma_{u}$ it follows directly
that $\int_{X\times Y}c\gamma=-J(u)$ and hence $\gamma$ attains
the lower bound above, i.e. $\gamma$ is optimal. 
\end{proof}
\begin{rem}
A byproduct of Theorem \ref{Thm:weak conv of fixed points towards optimal transport plans}
below (applied to the case when $\mu_{k}=\mu$ and $\nu_{k}=\nu$
for all $k)$ is a proof that there always exists a transport plan
$\gamma_{*}$ with support in $\Gamma_{u}$ for some $c-$convex function.
Since $\gamma_{*}$ saturates the lower bound \ref{eq:lower bd in pf prop opt crit}
it follows that taking the infimum over all $\gamma$ in $\Pi(\mu,\nu)$
yields equality in \ref{eq:lower bd in pf prop opt crit} . As a consequence,
\begin{equation}
\inf_{\gamma\in\Pi(\mu,\nu)}\int_{X\times Y}c\gamma=\sup_{(u,v)\in\Phi_{c}}\int u\mu+\int v\nu,\,\,\,\Phi_{c}:=\left\{ (u,v)\in C(X)\times C(Y):\,u+v\leq c\right\} \label{eq:Kant duality}
\end{equation}
This is the content of Kantorovich duality, which is usually shown
using Rockafeller-Fenchel duality in topological vector spaces \cite{v1}.
\end{rem}

\subsubsection{\label{subsec:The-torus-setting}The torus setting}

In section we consider the case when $X=Y=T^{n}:=\R^{n}/\Z^{n}$ and
the cost function $c:=d_{T^{n}}^{2}/2$ is half the squared standard
distance function on $T^{n}.$ We will identify a function $u$ on
$T^{n}$ with a $\Z^{n}-$periodic function on $\R^{n}$ in the usual
way. Similarly, we identify the cost function $c$ on $T^{n}$ with
a function on $\R^{n}\times\R^{n},$ which is $\Z^{n}-$periodic in
each argument
\begin{equation}
c(x,y):=\frac{1}{2}d_{T^{n}}(x,y)^{2}:=\frac{1}{2}\inf_{m\in\Z^{n}}|x+m-y|^{2}\label{eq:cost function on the torus}
\end{equation}
Note that $c(x,\cdot)$ is Lipschitz with Lipschitz constant $\sqrt{n}$
on $T^{n}$(endowed with its standard metric). As a consequence any
$c-$convex function $u$ on $T^{n}$ is also Lipschitz with Lipschitz
constant $\sqrt{n}.$ In this particular case a $c-$convex function
will be called \emph{quasi-convex} and we will say that $u\in C^{2}(T^{n})$
is \emph{strictly} \emph{quasi-convex} if $\nabla^{2}u+I>0.$ 
\begin{lem}
\label{lem:quasi conv}Let $u\in C^{2}(T^{n})$ be strictly quasi-convex.
Then
\begin{itemize}
\item the map
\begin{equation}
x\mapsto y_{x}:=x+(\nabla u)(x)\label{eq:c gradient map torus}
\end{equation}
 defines a $C^{1}-$diffeomorphism of $T^{n}.$
\item $u^{c}$ is also a strictly quasi-convex $C^{2}-$function on $T^{n}$
and the corresponding map 
\begin{equation}
y\mapsto x_{y}:=y+(\nabla u^{c})(y)\label{eq:def of x y}
\end{equation}
is the inverse of the map \ref{eq:c gradient map torus} and the following
matrix relation holds
\begin{equation}
(\nabla^{2}u+I)(x_{y})^{-1}=(\nabla^{2}u^{c}+I)(y)\label{eq:matrix relation for u}
\end{equation}
\end{itemize}
Conversely, if $u\in C^{2}(T^{n})$ is quasi-convex and the map \ref{eq:c gradient map torus}
is a $C^{1}-$diffeomorphism of $T^{n},$ then $u$ is strictly quasi-convex.
Moreover, if $u\in C^{k}(T^{n})$ is strictly quasi-convex, for $k$
a positive integer and $k\geq2,$ then $u^{c}\in C^{k}(T^{n}).$ 
\end{lem}

\begin{proof}
Given a function $\phi$ on $\R^{n}$ we denote by $\phi^{*}$ its
classical Legendre transform: 
\begin{equation}
\phi^{*}(y):=\sup_{x\in\R^{n}}x\cdot y-\phi(x)\label{eq:def of classical Leg tr}
\end{equation}
(in other words, this is the $c-$Legendre transform wrt $c(x,y):=-x\cdot y).$
Given a $\Z^{n}-$invariant quasi-convex function $u$ on $\R^{n}$
we set $\phi(x):=u(x)+|x|^{2}/2.$ Then it follows directly from the
definitions that $\phi$ is convex and
\begin{equation}
\phi^{*}(y)=u^{c}(y)+|y|^{2}/2,\,\,\,x+(\nabla u)(x)=\nabla\phi(x),\,\,\nabla^{2}u+I=\nabla^{2}\phi\label{eq:rel between leg trans and c leg}
\end{equation}
Next note that since $u^{c}$ is continuous and periodic it is bounded
and hence $\phi^{*}$ is finite on all of $\R^{n}$ with quadratic
growth. As a consequence, given $y\in\R^{n},$ the function $x\mapsto x\cdot y-\phi(x)$
attains it sup at some point $x_{y}\in\R^{n}$ 
\begin{equation}
\phi^{*}(y)=x_{y}\cdot y-\phi(x_{y})\label{eq:leg transform of phi in terms of x y}
\end{equation}
 and since the point is a local maximum we have $y=\nabla\phi(x_{y}).$
This shows that $\nabla\phi$ maps $\R^{n}$ surjectively onto $\R^{n}.$
Moreover, since $\nabla^{2}\phi>0$ the $\phi$ function is strictly
convex and $x_{y}$ is uniquely determined. Thus the $C^{1}-$map
$\nabla\phi$ is a bijection and moreover its inverse $y\mapsto x_{y}$
is also a $C^{1}-$map (by the inverse function theorem), which proves
the first claim in the lemma. Moreover, since $x_{y}$ is the unique
minimizer of the sup defining $\phi^{*}$ the function $\phi^{*}$
is differentiable with gradient $x_{y}$ at $y.$ Hence, the $C^{1}-$inverse
of $\nabla\phi$ is given by $\nabla\phi^{*},$ showing that $\phi^{*}$
is in $C^{2}(\R^{n}).$ Differentiating the identity $\nabla\phi^{*}\circ\nabla\phi=I$
finally proves \ref{eq:matrix relation for u} and the last statement
follows from the implicit function theorem: $\nabla\phi$ is $C^{k-1}$
implies (since $\nabla^{2}\phi>0)$ that its inverse $\nabla\phi^{*}$
is also $C^{k-1}.$ Finally, if $u\in C^{2}(T^{n})$ is quasi-convex
and the map \ref{eq:c gradient map torus} is a diffeomorphism of
$T^{n},$ then differentiating $(\nabla\phi)^{-1}\circ\nabla\phi=I$
reveals that the non-negative matrix $\nabla^{2}\phi$ is non-degenerate,
hence strictly positive, i.e $u$ is strictly quasi-convex. 
\end{proof}
We will also have use for the following
\begin{lem}
\label{lem:smooth}Assume that $u$ is $C^{1}-$smooth and strictly
quasi-convex. Then, for any fixed $y\in T^{n},$ the unique infimum
of the function $x\mapsto c(x,y)+u(x)$ on $T^{n}$ is attained at
$x=x_{y}$ (defined by formula \ref{eq:def of x y}). Moreover, the
function $x\mapsto c(x,y)$ is smooth on some neighborhood of $x_{y}$
in $T^{n}$ and its Hessian is equal to the identity there.
\end{lem}

\begin{proof}
First observe that, given $y\in T^{n},$ the infimum in question is
attained at $x_{y}$ (defined by formula \ref{eq:def of x y}), as
follows directly from combining formulas \ref{eq:leg transform of phi in terms of x y}
and \ref{eq:rel between leg trans and c leg}. Representing $x_{y}$
and $y$ with points in $\R^{n}$ we thus have

\[
d_{T^{n}}(x_{y},y)^{2}=\inf_{m\in\Z^{n}}|x_{y}+m-y|^{2}=|x_{y}+m_{0}-y|^{2}
\]
 for some $m_{0}\in\Z^{n}.$ We claim that, under the assumptions
of the lemma, the inf above is uniquely attained at $m_{0},$ i.e.
\[
m\neq m_{0}\implies|x_{y}+m-y|^{2}>|x_{y}+m_{0}-y|^{2}.
\]
To see this we note that, since $u$ is periodic, when viewed as a
function on $\R^{n},$ we have 
\begin{equation}
\inf_{x\in\R^{n}}d_{T^{n}}(x,y)^{2}/2+u(x)=\inf_{x\in\R^{n}}|x-y|^{2}/2+u(x)\label{eq:inf u equal to inf u}
\end{equation}
Since the inf in the left hand side above is attained at $x_{y}$
so is the inf in the right hand side. Now assume, to get a contradiction,
that the claim above does not hold, i.e. there exists a non-zero $m\in\Z^{n}$
such that $|x_{y}+m-y|=|x_{y}+m_{0}-y|.$ This implies that the inf
in the right hand side in formula \ref{eq:inf u equal to inf u} is
attained both at $x_{y}$ and at $x_{y}+m$ (since $u$ is periodic).
But this contradicts the fact that the function $x\mapsto|x-y|^{2}/2+u(x)$
is strictly convex on $\R^{n}$ (by the assumed strict quasi-convexity
of $u$ on $T^{n}).$ Finally, the claim shows, since the inequality
in the claim is preserved when $\bar{x}$ is perturbed slightly, that
$d_{T^{n}}(x,y)^{2}=|x-y|^{2}$ for all $x$ sufficiently close to
$x_{y}.$ Hence, $x\mapsto d_{T^{n}}(x,y)^{2}/2$ is smooth there
and its Hessian is constant, as desired.
\end{proof}

\subsection{The log Sinkhorn iteration on $C(X)$}

In this section we will consider an iteration on $C(X),$ which can
be viewed as a reformulation of the Sinkhorn algorithm and the Iterative
Proportional Fitting Procedure, recalled in Section \ref{subsec:The-Sinkhorn-algorithm}
(see Section \ref{subsec:Discretization-and-the}). Given data $(\mu,\nu,c),$
as in Section \ref{subsec:Recap-of-Optimal}, we first introduce the
following maps
\[
T_{\mu}:\,C(X)\rightarrow C(Y),\,\,\,u\mapsto v[u]:=\log\int e^{-c(x,\cdot)-u(x)}\mu(x)
\]
and 
\[
T_{\nu}:\,C(Y)\rightarrow C(X),\,\,\,v\mapsto u[v]:=\log\int e^{-c(\cdot,y)-v(y)}\nu(y)
\]
(abusing notation slightly we will write $T_{\mu}(u)=v[u]$ etc).
This yields an iteration on $C(X)$ defined by
\begin{equation}
u_{m+1}:=S[u_{m}],\label{eq:iteration for u}
\end{equation}
 where $S$ is defined as the the composed operator $T_{\nu}\circ T_{\mu}$
on $C(X):$ 
\[
S:\,C(X)\rightarrow C(X),\,\,\,u\mapsto u[v[u]]
\]
In lack for a better name the iteration \ref{eq:iteration for u}
will be called the\emph{ log Sinkhorn iteration }and the operator
$S$ will be called the\emph{ log Sinkhorn operator. }It will be convenient
to rewrite it as the following difference equation:
\begin{equation}
u_{m+1}-u_{m}=\log(\rho_{u_{m}}),\label{eq:difference eq}
\end{equation}
 where $\rho_{u}$ is defined by
\begin{equation}
\rho_{u}:=e^{S[u]-u}\label{eq:def of pho u in terms of S}
\end{equation}
and has the property that $\rho_{u}\mu$ is a probability measure
on $X$ (as follows directly from the definitions).

In this section we will use a variational approach to study the log
Sinhorn iteration. An alternative approach will also be used in Section
\ref{sec:Convergence-of-the iter torus} which relies on the observation
that the log Sinkhorn iteration contracts the $L^{\infty}-$distance
on $C(X)$ (see Step 2 in the proof of Lemma \ref{lem:general conv towards parab}).

\subsubsection{Existence and uniqueness of fixed points}

Consider the following functional $\mathcal{F}$ on $C(X):$ 

\begin{equation}
\mathcal{F}:=I_{\mu}-\mathcal{L},\,\,\,I_{\mu}(u)=\int_{X}u\mu,\,\,\,\mathcal{L}(u):=-\int_{Y}v[u]\nu.\label{eq:def of F and L}
\end{equation}
Note that $I_{\mu}$ and $\mathcal{L}$ are equivariant under the
additive action of $\R$ and hence $\mathcal{F}$ is invariant.
\begin{rem}
This functional can be viewed as an analog of the Kantorovich functional
$J(u)$ (formula \ref{eq:Kant functional}), where the $c-$Legendre
transform $u^{c}$ is replaced by $v[u].$ From a numerical perspective
this amounts to replacing the supremum defining $u^{c}$ by a ``soft
max'' \cite{m-p}. It is well-known that $\mathcal{F}$ decreases
under the log Sinkhorn iteration, i.e. that $\mathcal{F}(Su)\leq\mathcal{F}(u).$
This is usually shown using block coordinate descent on the ``dual
functional'' to the ``primal'' minimization problem in Prop \ref{prop:The-unique-minimizer};
see \cite[Prop 4.21]{m-p} and \cite[Section 3.1]{via}. But here
we observe that, in fact,\emph{ both} functionals $I_{\mu}$ and $-\mathcal{L}$
decrease along the iteration (see Step 1 in the proof of Theorem \ref{thm:conv of u m in general setting}).
This will be important in the proof of the last step of Theorem \ref{thm:conv of u m in general setting}
and in the proof of Prop \ref{prop:generalized parabolic}. 
\end{rem}

\begin{lem}
The following is equivalent:
\begin{itemize}
\item $u$ is a critical point for the functional $\mathcal{F}$ on $C(X)$
\item $\rho_{u}=1$ a.e. with respect to $\mu$
\end{itemize}
Moreover, if $u$ is a critical point, then $u_{*}:=S(u)$ is a fixed
point for the operator $S$ on $C(X)$
\end{lem}

\begin{proof}
First observe that the differential of the functional $\mathcal{L}$
defined in formula \ref{eq:def of F and L}, at an element $u\in C(X),$
is represented by the probability measure $\rho_{u}\mu,$ where $\rho_{u}$
is defined by formula \ref{eq:def of pho u in terms of S}. This means
that for any $\dot{u}\in C(X)$
\[
\frac{d(\mathcal{L}(u+t\dot{u}))}{dt}|_{t=0}=\int\dot{u}\rho_{u}\mu.
\]
This follows readily from the definitions by differentiating $t\mapsto v[(u+t\dot{u})]$
to get an integral over $(X,\mu)$ and then switching the order of
integration. As a consequence, $u$ is a critical point of the functional
$\mathcal{F}$ on $C^{0}(X)$ iff $\rho_{u}\mu=\mu,$ i.e. iff $\rho_{u}=1$
a.e. with respect to $\mu.$ Finally, if this is the case then $S(u)=u$
a.e wrt $\mu$ and hence $S(S(u))=S(u)$ (since $S(f)$ only depends
on $f$ viewed as an element in $L^{1}(X,\mu)).$ 
\end{proof}
The following basic compactness property holds:
\begin{lem}
\label{lem:compactness of function spaces}Given a point $x_{0}\in X$
the subset $\mathcal{K}_{x_{0}}$ of $C(X)$ defined as all elements
$u$ in the image of $S$ satisfying $u(x_{0})=0$ is compact in $C(X).$ 
\end{lem}

\begin{proof}
First observe that, since $X\times Y$ is assumed compact, the continuous
function $c$ is, in fact, uniformly continuous on $X.$ Hence, it
follows from the very definition of $S$ that $S(C(X))$ is an equicontinuous
family of continuous functions on $X.$ By Arzela-Ascoli theorem it
follows that the set $\mathcal{K}_{x_{0}}$ is  compact in $C(X).$ 
\end{proof}
Using the previous two lemmas gives the following
\begin{prop}
\label{prop:exist and uniq of fix point}The operator $S$ has a fixed
point $u_{*}$ in $C(X).$ Moreover, $u_{*}$ is uniquely determined
a.e. wrt $\mu$ up to an additive constant and $u_{*}$ minimizes
the functional $\mathcal{F}$. More precisely, there exists a unique
fixed point in $S(C(X))/\R.$
\end{prop}

\begin{proof}
We start by noting that 
\[
\mathcal{F}(Su)\leq\mathcal{F}(u)
\]
 (this is shown in the first step of Theorem \ref{thm:conv of u m in general setting}
below). Since $\mathcal{F}$ is invariant under the natural $\R-$action
we conclude that 
\[
\inf_{C(X)}\mathcal{F}=\inf_{\mathcal{K}_{0}}\mathcal{F},
\]
 where $\mathcal{K}_{0}$ denotes the compact subset of $C(X)$ appearing
in Lemma \ref{lem:compactness of function spaces}. Since $\mathcal{F}$
is clearly continuous on $C(X)$ this implies the existence of a minimizer
of $\mathcal{F}$ which is moreover in $\mathcal{K}_{0}.$

Next observe that $\mathcal{F}$ is convex on $C(X).$ Indeed, for
any fixed $y\in Y,$ $u\mapsto v[u](y)$ is convex on $C(X),$ as
follows directly from Jensen's inequality. Hence, $-\mathcal{L}$
is convex and since $I_{\mu}$ is affine we conclude that $\mathcal{F}$
is convex. More precisely, Jensen's (or Hölder's) inequality implies
that $\mathcal{F}$ is strictly convex on $C(X)/\R$ viewed as a subset
of $L^{1}(\mu)/\R.$ Hence, if $u_{0}$ and $u_{1}$ are two minimizers,
then there exists a constant $C$ such that $u_{0}=u_{1}+C$ a.e.
wrt $\mu.$ In particular, if $C=0$ then $u_{*}:=S(u_{0})=S(u_{1})$
gives the same fixed point of $S.$
\end{proof}

\subsubsection{Monotonicity and convergence properties of the iteration}

We next establish the following result, which can be seen as a refinement,
in the present setting, of the convergence of the general Iterative
Proportional Fitting Procedure established in \cite{r}. The result
will be used in the proof of Proposition \ref{prop:generalized parabolic}. 
\begin{thm}
\label{thm:conv of u m in general setting}Given $u_{0}\in C(X)$
the corresponding iteration $u_{m}:=S^{m}u_{0}$ converges uniformly
to a fixed point $u_{\infty}$ of $S.$ 
\end{thm}

\begin{proof}
\emph{Step 1: $I_{\mu}$ and $-\mathcal{L}$ are decreasing along
the iteration and hence $\mathcal{F}$ is also decreasing. The functionals
are strictly decreasing at $u_{m}$ unless $S(u_{*})=u_{*}$ for $u_{*}:=S(u_{m}).$ }

Using the difference equation \ref{eq:difference eq} for $u_{m}$
and Jensen's inequality, we have 
\[
I_{\mu}(u_{m+1})-I_{\mu}(u_{m})=\int\log\rho_{u_{m}}\mu\leq\log\int\rho_{u_{m}}\mu=\log1=0
\]
Moreover, equality holds unless $\rho_{u_{m}}=1$ a.e wrt $\mu$ i.e.
$S(u_{m})=u_{m}$ and $S(u_{*})=u_{*}$ everywhere on $X.$ Similarly,
by symmetry, 

\[
\mathcal{L}(u_{m})-\mathcal{L}(u_{m+1})=\int\log\rho_{v_{m}}\nu\leq\log\int\rho_{v_{m}}\nu=\log1=0,
\]
 where now $\rho_{v},$ for $v\in C(Y),$ denotes the probability
measure on $Y$ defined as in formula \ref{eq:def of pho u in terms of S},
but with the roles of $\mu$ and $\nu$ interchanged.

\emph{Step 2: Convergence in $C(X)/\R.$}

Given the initial data $u_{0}$ we denote by $\mathcal{K}_{u_{0}}$
the closure of the orbit of $u_{0}$ in $C(X)$ under repeated application
of $S.$ By Lemma \ref{lem:compactness of function spaces} $\mathcal{K}_{u_{0}}/\R$
is compact in $C(X)/\R.$ Hence, after perhaps passing to a subsequence,
$u_{m}\rightarrow u_{\infty}$ in $C(X)/\R.$ Now, since $\mathcal{F}$
is decreasing along the orbit we have 
\[
\mathcal{F}(u_{\infty})=\inf_{\mathcal{K}_{0}}\mathcal{F}.
\]
 Hence, by the condition for strict monotonicity it must be that $Su_{\infty}=u_{\infty}$
a.e. wrt $\mu$ and hence, since $u_{\infty}$ is the image of $S,$
it follows that $Su_{\infty}=u_{\infty}$ on all of $X.$ It then
follows from Proposition \ref{prop:exist and uniq of fix point} that
$u_{\infty}$ is uniquely determined in $C(X)/\R$ (by the initial
data $u_{0}),$ i.e. the whole sequence converges in $C(X)/\R.$

\emph{Step $3:$ Convergence in $C(X)$}

Let us first show that there exists a number $\lambda\in\R$ such
that 
\begin{equation}
\lim_{m\rightarrow\infty}I_{\mu}(u_{m})=\lambda.\label{eq:limit is lambda}
\end{equation}
 By Step $1$ $I_{\mu}$ is decreasing and hence it is enough to show
that $I_{\mu}(u_{m})$ is bounded from below. But $I_{\mu}=\mathcal{F}+\mathcal{L},$
where, by Prop \ref{prop:exist and uniq of fix point} (or the previous
step) $\mathcal{F}$ is bounded from below (by $\mathcal{F}(u_{\infty}))$.
Moreover, by the first step $\mathcal{L}(u_{m})\geq\mathcal{L}(u_{0}),$
which concludes the proof of \ref{eq:limit is lambda}. Next, decompose 

\[
u_{m}=\tilde{u}_{m}+u_{m}(x_{0}),\,\,\,
\]
 By Lemma \ref{lem:compactness of function spaces} the sequence $(\tilde{u}_{m})$
is relatively compact in $C(X)$ and we claim that $|u_{m}(x_{0})|\leq C$
for some constant $C.$ Indeed, if this is not the case then there
is a subsequence $u_{m_{j}}$ such that $|u_{m_{j}}|\rightarrow\infty$
uniformly on $X.$ But this contradicts that $I_{\mu}(u_{m})$ is
uniformly bounded (by \ref{eq:limit is lambda}). It follows that
the sequence $(u_{m})$ is also relatively compact. Hence, by the
previous step the whole sequence $u_{m}$ converges to the unique
minimizer $u_{*}$ of $\mathcal{F}$ in $S(C(X))$ satisfying $I_{\mu}(u_{*})=\lambda.$
\end{proof}
\begin{rem}
\label{rem:lin convergence rate}The convergence result in \cite{r}
is, in the present setting, equivalent to convergence of the induced
iteration on $C(X)/\R.$ In fact, the latter convergence holds at
linear rate, i.e. there exists a norm $\left\Vert \cdot\right\Vert _{C(X)/\R}$
on $C(X)/\R$ and a positive number $\delta$ such that $\left\Vert u-u_{\infty}\right\Vert _{C(X)/\R}\leq e^{-\delta m}.$
Indeed, setting $\left\Vert u-u'\right\Vert _{C(X)/\R}:=\left\Vert \sup(u-u')-\inf(u-u')\right\Vert _{C(X)}$
(which corresponds, under $u\mapsto e^{-u},$ to the Hilbert metric
on the cone of positive functions in $C(X))$ this follows from Birkhoff's
theorem about positive operators on cones, precisely as in the finite
dimensional situation of the Sinkhorn iteration considered in \cite{f-l};
see also \cite[Thm 4.2]{m-p} and \cite{via}. 
\end{rem}

\subsubsection{The induced discrete evolution on $C(X)\times C(Y)$}

Fixing an initial function $u_{0}\in C(X)$ the corresponding evolution
$m\mapsto u_{m}$ induces a sequence of pairs $(u_{m},v_{m})\in C(X)\times C(Y)$
defined by the following recursion: 
\[
(u_{m+1},v_{m+1}):=(u[v_{m+1}],v[u_{m}])
\]

\subsubsection{The induced discrete evolution on $\mathcal{P}(X\times Y)$ and entropy}

Let us briefly explain the dual point of view involving the space
$\mathcal{M}(X\times Y)$ of measures on $X\times Y$ (which, however,
is not needed for the proofs of the main results). The data $(\mu,\nu,c)$
induces the following element $\gamma_{c}\in\mathcal{M}(X\times Y):$
\[
\gamma_{c}:=e^{-c}\mu\otimes\nu
\]
Given a function $u\in C(X)$ we will write 
\begin{equation}
\gamma_{u}:=e^{-(u+v[u])}\gamma_{c}\label{eq:def of gamma u}
\end{equation}

\begin{lem}
\label{lem:-u fixed iff gamma u is in Pi}$u$ satisfies $S(u)=u$
a.e. wrt $\mu$ iff $\gamma_{u}\in\Pi(\mu,\nu).$
\end{lem}

\begin{proof}
A direct computation reveals that the push-forwards of $e^{-(u+v[u])}\gamma_{c}$
to $X$ and $Y,$ respectively, are given by
\[
\int_{X}e^{-(u+v[u])}\gamma_{c}=\nu,\,\,\,\,\int_{Y}e^{-(u+v[u])}\gamma_{c}=\rho_{u}\mu
\]
Hence, $\gamma_{u}\in\Pi(\mu,\nu)$ iff $\rho_{u}\mu=\mu,$ which,
by the definition \ref{eq:def of pho u in terms of S} of $\rho_{u},$
concludes the proof.
\end{proof}
The discrete dynamical system $u_{m}$ induces a sequence 
\[
\gamma_{m}:=\gamma_{u_{m}}(=e^{-u_{m}(x)}e^{-v_{m}(x)}\gamma_{c})\in\mathcal{P}(X\times Y)
\]

\begin{prop}
\label{prop:The-unique-minimizer}The unique minimizer $\gamma_{*}$
of the functional $\mathcal{I}(\cdot|\gamma_{c})$ on $\Pi(\mu,\nu)$
is characterized by the property that it has the form 
\[
\gamma_{_{*}}=e^{-\Phi}\gamma_{c}
\]
 for some $\Phi\in C(X)+C(Y).$ Moreover, $\gamma_{_{*}}=\gamma_{u_{*}}$,
where $u_{*}$ is a fixed point for $S$ on $C(X)$ (or more generally,
on $L^{1}(X,\mu))$ and 
\begin{equation}
\inf_{\Pi(\mu,\nu)}\mathcal{I}(\cdot|\gamma_{c})=\inf_{C(X)\times C(Y)}\mathcal{F}\label{eq:inf of I is inf of F}
\end{equation}
and given any function $u_{0}\in C(X),$ the corresponding sequence
$\gamma_{m}$ converges in $L^{1}$ (i.e. in variation norm) towards
$\gamma_{*}$ (and moreover $\mathcal{I}(\gamma_{m}|\gamma_{*})\rightarrow0).$
\end{prop}

\begin{proof}
By construction $\gamma_{*}:=\gamma_{u_{*}}$ has the property that
\[
\gamma_{_{*}}=e^{-\Phi}\gamma_{c},\,\,\,\gamma_{_{*}}\in\Pi(\mu,\nu)
\]
for some $\Phi\in L^{\infty}(X)+L^{\infty}(Y).$ But a standard calculus
argument reveals that any such $\gamma_{*}$ is the unique minimizer
of the restriction of $\mathcal{I}$ to the affine subspace $\Pi(\mu,\nu)$
of $\mathcal{P}(X\times Y$) (using that $\mathcal{I}$ is strictly
convex). The last convergence statement then follows directly from
Theorem \ref{thm:conv of u m in general setting} (only the easier
convergence in $C(X)/\R$ is needed). 
\end{proof}
Rewriting
\[
k^{-1}\mathcal{I}(\gamma|\gamma_{kc})=\int c\gamma+k^{-1}\mathcal{I}(\gamma|\mu\otimes\nu),
\]
 the equality \ref{eq:inf of I is inf of F} can be viewed as an entropic
variant of Kantorovich duality \ref{eq:Kant duality} in the limit
when $c$ is replaced by $kc$ for a large positive number $k.$ In
fact, it follows from Theorem \ref{Thm:weak conv of fixed points towards optimal transport plans}
applied to $\mu_{k}=\mu$ and $\nu_{k}=\nu$ that
\[
\lim_{k\rightarrow\infty}\inf_{\gamma\in\Pi(\mu,\nu)}k^{-1}\mathcal{I}(\gamma|\gamma_{kc})=\inf_{\gamma\in\Pi(\mu,\nu)}\int c\gamma=\sup_{\Phi_{c}}\int u\mu+\int v\nu,
\]
 as in the Kantorovich duality \ref{eq:Kant duality}. In the next
section we will consider the setting where $\mu$ and $\nu$ also
change with $k.$

\subsection{\label{subsec:The-parametrized-setting}The scaled setting and discretization}

Let us next consider the following variant of the previous setting,
parametrized by a parameter $k$ (which is the parameter that will
later on tend to infinity and which corresponds to the entropic regularization
parameter $\epsilon:=k^{-1}$). This means that we replace the triple
$(\mu,\nu,c)$ with a sequence $(\mu^{(k)},\nu^{(k)},kc).$ As explained
in Section \ref{subsec:Discrete-optimal-transport} replacing $c$
with $kc$ corresponds to introducing the entropic regularization
parameter $\epsilon=k^{-1}.$ We then rescale the functions in $C(X)$
and $C(Y)$ by $k$ and consider the corresponding rescaled operators:

\[
v^{(k)}[u]:=k^{-1}\log\int e^{-kc(x,\cdot)-ku(x)}\mu^{(k)}(x)
\]
\[
u^{(k)}[v]:=k^{-1}\log\int e^{-kc(\cdot,y)-kv(y)}\nu^{(k)}(y)
\]
\begin{equation}
S^{(k)}(u):=k^{-1}S(ku)\label{eq:scaled log Sinkhor op}
\end{equation}
etc. The corresponding rescaled iteration is thus defined by the iteration
\begin{equation}
u_{m+1}^{(k)}:=S^{(k)}u_{m}^{(k)}\in C(X),\label{eq:scaled iteration}
\end{equation}
given the initial value $u_{0}^{(k)}\in C(X).$ It will be called
the \emph{scaled log Sinkhorn iteration (at level $k).$ }Equivalently,
\begin{equation}
u_{m+1}^{(k)}-u_{m}^{(k)}=k^{-1}\log(\rho_{ku_{m}^{(k)}}),\label{eq:difference equation}
\end{equation}
 where 
\begin{equation}
\rho_{ku}(x):=\frac{e^{ku^{(k)}\left[v^{(k)}[u]\right](x)}}{e^{ku}}(x\text{)},\label{eq:explicit expres for pho}
\end{equation}
which can be explicitly expressed as 
\[
\rho_{ku}(x)=\int_{Y}\frac{e^{-kc(x,y)-ku(x)}}{\int_{X}e^{-kc(x',y)-ku(x')}\mu^{(k)}(x')}\nu^{(k)}(y)
\]
We also set 
\begin{equation}
\mathcal{F}^{(k)}(u):=k^{-1}\mathcal{F}(ku)=\int u\mu+\int v^{(k)}[u]\nu.\label{eq:def of F k}
\end{equation}
By Theorem \ref{thm:conv of u m in general setting} (applied to a
fixed $k)$ as $m\rightarrow\infty$ the iteration $u_{m}^{(k)}$
converges in $C(X)$ to a fixed point $u^{(k)}$ of the operator $S^{(k)}$
(uniquely determined by the initial value $u_{0}^{(k)})).$ 

We observe that the following compactness property holds (and is proved
exactly as in Lemma \ref{lem:compactness of function spaces}): 
\begin{lem}
\label{lem:compactness of function with k}The union $\bigcup_{k\geq0}S^{(k)}$
is relatively compact in $C(X)/\R$ (identifying $C(X)/\R$ with the
set of all continuous functions vanishing at a given point $x_{0})$
\end{lem}

\subsubsection{\label{subsec:Discretization-and-the}Discretization and the Sinkhorn
algorithm}

Now assume that $\mu^{(k)}$ and $\nu^{(k)}$ are discrete probability
measures whose supports are finite sets 
\[
X^{(k)}:=\{x_{i}^{(k)}\}_{i=1}^{N_{k}},\,\,\,Y^{(k)}:=\{y_{i}^{(k)}\}_{i=1}^{N_{k}}
\]
 of the same number $N_{k}$ of points in $X$ and $Y,$ respectively.
This means that there exist vectors $p^{(k)}$ and $q^{(k)}$ in $\R^{N_{k}}$
such that
\[
\mu^{(k)}=\sum_{i=1}^{N_{k}}\delta_{x_{i}^{(k)}}p_{i}^{(k)},\,\,\,\nu^{(k)}=\sum_{i=1}^{N_{k}}\delta_{x_{i}^{(k)}}q_{i}^{(k)}.
\]
 Moreover, since $\mu^{(k)}$ and $\nu^{(k)}$ are probability measures
the vectors $p^{(k)}$ and $q^{(k)}$ are elements in the simplex
$\Sigma_{N_{k}}$ in $\R^{N_{k}}$ defined by
\begin{equation}
\Sigma_{N}:=\left\{ v\in\R^{N}:v_{i}\geq0,\,\,\,\sum_{i=1}^{N}v_{i}=1\right\} ,\label{eq:def of simplex}
\end{equation}
 which we identify with $\mathcal{P}(\{1,...,N\}).$ Similarly, we
identify the discrete measure 
\[
\gamma_{c}^{(k)}:=e^{-kc}\mu^{(k)}\otimes\nu^{(k)}
\]
on $X\times Y$ with the matrix $\tilde{K}\in\R^{N_{k}}\times\R^{N_{k}}$
defined by 

\[
\tilde{K}_{ij}:=K_{ij}^{(k)}p_{i}^{(k)}q_{j}^{(k)},\,\,\,K_{ij}^{(k)}:=\exp(-kC_{ij}),\,\,\,C_{ij}:=c(x_{i}^{(k)},y_{j}^{(k)}),
\]
 where $C_{ij}$ is viewed as a cost function on $\{1,...,N\}^{2}.$
Under the identifications 
\[
C(X^{(k)})\leftrightarrow\R_{+}^{N_{k}},\,\,\,u\mapsto a,\,\,\,a_{i}:=e^{-ku(x_{i}^{(k)})}p_{i}^{(k)}
\]
and 
\[
C(Y^{(k)})\leftrightarrow\R_{+}^{N_{k}},\,\,\,v\mapsto b,\,\,\,b_{i}:=e^{-kv(y_{j}^{(k)})}q_{i}^{(k)}
\]
the scaled iteration \ref{eq:scaled iteration} gets identified with
the recursion $a^{(k)}(m)$ defined by the Sinkhorn algorithm determined
by the matrix $K^{(k)}$ and the positive vectors $p^{(k)}$ and $q^{(k)}$
(see Section \ref{subsec:The-Sinkhorn-algorithm}). Given an initial
positive vector $a^{(k)}(0)$ Theorem \ref{thm:conv of u m in general setting}
thus shows that $(a^{(k)}(m),b^{(k)}(m))$ converges, as $m\rightarrow\infty,$
to a pair of positive vectors $(a^{(k)},b^{(k)})$ such that the scaled
matrix $D_{b}K^{(k)}D_{a}$ has the property that the rows sum to
$p^{(k)}$ and the columns sum to $q^{(k)}.$
\begin{rem}
\label{rem:Fourier}By construction, the functions $u_{m}^{(k)}(x)$
on $X$ can be expressed in terms of a Fourier type sum: 
\[
u_{m}^{(k)}(x)=k^{-1}\log\sum_{i=1}^{N_{k}}e^{-kc\left(x,y_{i}^{(k)}\right)}b_{i}^{(k)}(m-1)
\]
 where the ``Fourier coefficients'' $b_{i}^{(k)}(m-1)$ are given
by the Sinkhorn algorithm. In the case when $X$ and $Y$ are domains
in $\R^{n},$ with $c(x,y)=-x\cdot y,$ this is the analytic continuation
to $i\R^{n}$ of a bona fide Fourier sum with Fourier coefficients
in $k$ times the support of $\nu^{(k)}.$ Hence, $k$ plays the role
of the ``band-width''. 
\end{rem}

\section{\label{sec:Proofs-of-the}Convergence of the  fixed points }

In this section we will prove various generalizations of Theorem \ref{thm:conv in static torus setting intr},
stated in the introduction. Throughout the section we will consider
the parametrized setting in Section \ref{subsec:The-parametrized-setting}
and assume that the sequences $\mu^{(k)}$ and $\nu^{(k)}$ converge
to $\mu$ and $\nu$ in $\mathcal{P}(X)$ and $\mathcal{P}(Y),$ respectively
(in the standard weak topology). We will denote by $u^{(k)}$ the
fixed point of the corresponding operator $S^{(k)}$ on $C(X),$ uniquely
determined by the normalization condition $u^{(k)}(x_{0})=0,$ at
a given point $x_{0}$ in $X$ and set $v^{(k)}:=T_{\mu}u^{(k)}=:v^{(k)}[u^{(k}],$
which is a fixed point of the corresponding operator $S^{(k)}$ on
$C(Y).$ 

\subsection{A general convergence result for the fixed points}

We start by giving a density condition on $\mu^{(k)}$ ensuring that
$v^{(k)}[u]$ converges uniformly to the $c-$Legendre transform $u^{c}$
of $u,$ when $\mu$ has full support:
\begin{lem}
\label{lem:dens prop}Assume that the sequence $\mu^{(k)}$ converging
to $\mu$ in $\mathcal{P}(X)$ has the following ``density property'':
for any given open subset $U$ intersecting the support $X_{\mu}$
of $\mu$ 
\begin{equation}
\liminf_{k\rightarrow\infty}k^{-1}\log\mu^{(k)}(U)\geq0\label{eq:density property}
\end{equation}
Then, for any given $u\in C(X),$ the sequence $v^{(k)}[u]$ converges
uniformly to $(\chi_{X_{\mu}}+u)^{c}$ in $C(Y).$ 
\end{lem}

\begin{proof}
Replacing the integral over $\mu^{(k)}$ with a sup directly gives
\begin{equation}
v^{(k)}[u](y)\leq(\chi_{X_{\mu}}+u)^{c}(y)\label{eq:ineq v k and u c}
\end{equation}
for any $y\in Y.$ To prove a reversed inequality let $x_{y}$ be
a point in $X_{\mu}$ where the sup defining $(\chi_{X_{\mu}}+u)^{c}(y)$
is attained and $U_{\delta}$ a neighborhood of $x_{y}$ where the
oscillation of $c(\cdot,y)+u$ is bounded from above by $\delta$
(the existence of $U_{\delta}$ is ensured by the continuity of $c$
and the compactness of $X$ and $Y).$ Then 
\[
v^{(k)}[u](y)\geq k^{-1}\log\int_{U_{\delta}}e^{-k(c(x,y)+u(x))}\mu^{(k)}(x)\geq k^{-1}\log\mu^{(k)}(U_{\delta})+u^{c}(y)-\delta
\]
Hence, as $k\rightarrow\infty,$ $v^{(k)}[u](y)\rightarrow u^{c}(y)$
and since $v^{(k)}[u]$ is equicontinuous (by the assumed compactness
of $X\times Y$ and the continuity of $c)$ this implies the desired
uniform convergence. 
\end{proof}
\begin{example}
\label{exa:The-density-property}(Weighted point clouds). If $\mu_{k}=\mu$
for any $k$ then the density property is trivially satisfied. More
generally, the density property \ref{eq:density property} is satisfied
by any reasonable approximation $\mu^{(k)}.$ For example, in the
discrete case where $\mu^{(k)}=\sum_{i=1}^{N_{k}}w_{i}^{(k)}\delta_{x_{i}^{(k)}}$
the property in question holds if $\sup_{i}1/w_{i}^{(k)}$ and the
inverse of the number of points $x_{i}^{(k)}$ in any given open set
$U$ intersecting $X_{\mu}$ have sub-exponential growth in $k.$ 
\end{example}

\begin{thm}
\label{Thm:weak conv of fixed points towards optimal transport plans}Suppose
that $\mu^{(k)}\rightarrow\mu$ and $\nu^{(k)}\rightarrow\mu$ in
$\mathcal{P}(X)$ and $\mathcal{P}(Y),$ respectively and assume that
$\mu^{(k)}$ and $\nu^{(k)}$ satisfy the density property \ref{eq:density property}.
Let $u^{(k)}$ be the normalized fixed point for the scaled log Sinkhorn
operator $S^{(k)}$ on $C(X).$ Then, after perhaps passing to a subsequence,
the following holds:
\[
u^{(k)}\rightarrow u
\]
uniformly on $X,$ where $u$ is a $c-$convex minimizer of the Kantorovich
functional $J$ (formula \ref{eq:Kant functional}) satisfying
\begin{equation}
u=(\chi_{Y_{\nu}}+(\chi_{X_{\mu}}+u)^{c})^{c}\label{eq:u in theorem as a Legendre transform twice}
\end{equation}
As a consequence, the corresponding probability measures
\[
\gamma^{(k)}:=e^{-k(u^{(k)}+v^{(k)})}e^{-kc}\mu^{(k)}\otimes\nu^{(k)}\in\mathcal{P}(X\times Y)
\]
converge weakly to a transport plan $\gamma$ between $\mu$ and $\nu,$
which is optimal wrt the cost function $c.$ 
\end{thm}

\begin{proof}
\emph{Step 1: Convergence of a subsequence of $u^{(k)}$ }

In the following all functions will be normalized by demanding that
the values vanish at a given point. By Lemma \ref{lem:compactness of function with k}
we may, after perhaps passing to a subsequence, assume that $u^{(k)}\rightarrow u^{(\infty)}$
uniformly on $X,$ for some element $u^{(\infty)}$ in $C(X).$ By
the previous lemma, for any given $u\in C(X)$ we have
\begin{equation}
\mathcal{F}^{(k)}(u)=J(\chi_{X_{\mu}}+u)+o(1),\label{eq:F k asympt to J}
\end{equation}
 where $\mathcal{F}^{(k)}$ is defined by formula \ref{eq:def of F k}
and $J$ denotes the Kantorovich funtional, defined by formula \ref{eq:Kant functional}.
Now take a sequence $\epsilon_{k}$ of positive numbers tending to
zero such that
\begin{equation}
u^{(\infty)}-\epsilon_{k}\leq u^{(k)}\leq u^{(\infty)}+\epsilon_{k}\label{eq:u k compard to u infy}
\end{equation}
Since $u\mapsto v^{(k)}[u]$ is decreasing it follows that 

\begin{equation}
\mathcal{F}^{(k)}(u^{(\infty)})-2\epsilon_{k}\leq\mathcal{F}^{(k)}(u^{(k)})+2\epsilon_{k}\label{eq:ineq for F k}
\end{equation}
Next note that since $u^{(k)}$ minimizes the functional $\mathcal{F}^{(k)}$
(by Prop \ref{prop:exist and uniq of fix point}) we have $\mathcal{F}^{(k)}(u^{(k)})\leq\mathcal{F}^{(k)}(u)$
for any given $u$ in $C(X).$ Hence, combining \ref{eq:ineq for F k}
and \ref{eq:F k asympt to J} and letting $k\rightarrow\infty$ gives

\[
J(u^{(\infty)})\leq\inf_{u\in C(X)}J(u),
\]
 showing that $u^{(\infty)}$ minimizes $J$ on $C(X).$ To see that
$u^{(\infty)}$ is $c-$convex first recall that, by definition, $u^{(k)}$
satisfies 
\[
u^{(k)}=u^{(k)}[v^{(k)}[u^{(k)}]].
\]
Hence, combing \ref{eq:u k compard to u infy} with the previous lemma,
applied twice, gives

\[
u^{(k)}=u^{(k)}[(\chi_{X_{\mu}}+u^{(\infty)})^{c}]+o(1)=((\chi_{Y_{\mu}}+(\chi_{X_{\mu}}+u^{(\infty)})^{c})^{c}+o(1)
\]
 This shows that $u^{(\infty)}=((\chi_{Y_{\mu}}+(\chi_{X_{\mu}}+u^{(\infty)})^{c})^{c},$
proving that $u^{(\infty)}=f^{c}$ for some $f\in C(Y).$ Hence $u^{(\infty)}$
is $c-$convex. 

\emph{Step 2: Convergence of $\gamma^{(k)}$ (for the subsequence
in Step 1) towards an optimizer}

By Lemma \ref{lem:-u fixed iff gamma u is in Pi} \emph{$\gamma^{(k)}$
is in $\Pi(\mu,\nu).$ }Hence, by weak compactness, we may assume
that \emph{$\gamma^{(k)}$ }converges towards an element $\gamma^{(\infty)}$
in $\mathcal{P}(X\times Y).$ By Prop \ref{prop:optim crit} it will
thus be enough to show that $\gamma^{(\infty)}$ is supported in $\Gamma_{u^{(\infty)}}.$
To this end let $\Gamma_{\delta}$ be the closed subset of $X\times Y$
where $u+u^{c}+c\geq\delta>0$ for $u:=u^{(\infty)}.$ By the previous
lemma $\gamma^{(k)}\leq e^{-k\delta/2}\mu^{(k)}\otimes\nu^{(k)}$
on $\Gamma_{\delta},$ when $k$ is sufficiently large and hence the
limit $\gamma^{(\infty)}$ is indeed supported on $\Gamma_{u^{(\infty)}}.$ 
\end{proof}
In order to ensure that the whole sequence $u^{(k)}$ is convergent
some conditions on the cost function $c$ and the measures $\mu$
and $\nu$ need to be imposed. Exploiting well-known uniqueness result
for optimal transport plans/maps this can, in particular, be achieved
in the following Riemannian setting. 
\begin{thm}
\label{thm:Static conv in Riem setting}Let $M$ be a Riemannian manifold
and denote by $d$ the Riemannian distance function. Let $X$ and
$Y$ be compact subsets of $M$ such that $Y$ is a topological domain,
i.e. $Y$ is equal to the closure of the interior of $Y$ and take
$c(x,y)$ to be the restriction of $d(x,y)^{2}/2$ to $X\times Y.$
Assume that $\nu$ is absolutely continuous wrt the Riemannian volume
form and has support $Y$ and that $\mu^{(k)}\rightarrow\mu$ and
$\nu^{(k)}\rightarrow\nu$ in $\mathcal{P}(X).$ Denote by $u^{(k)}$
the normalized fixed point of the scaled log Sinkhorn operator $S^{(k)}$
on $C(X).$ Then
\begin{itemize}
\item $v^{(k)}$ converges uniformly in $Y$ to a $c-$convex function $v,$
which is a potential for the unique optimal Borel map transporting
$\nu$ to $\mu,$ i.e. the map that can be expressed as 
\begin{equation}
y\mapsto x_{y}:=\text{exp}_{y}(\nabla v),\label{eq:y mapsto x as exp}
\end{equation}
(which means that $x_{x}$ is obtained by transporting $y$ along
a unit-length geodesic in the direction of $(\nabla v)(y)).$
\item $u^{(k)}$ converges uniformly on $X$ towards the $c-$convex function
$u$ given by the $c-$Legendre transform $v^{c}$ of $v.$ Moreover,
$u$ and $v$ satisfy
\[
u=(u+\chi_{X_{\nu}})^{cc},\,\,\,v=(u+\chi_{X_{\nu}})^{c}=u^{c}=v
\]
 
\item If $\mu$ is absolutely continuous wrt the Riemannian volume form,
then $x\mapsto\exp_{y}(\nabla u)$ defines the optimal transport of
$\mu$ to $\nu.$
\end{itemize}
\end{thm}

\begin{proof}
This will be shown to follow from the previous theorem combined with
results in \cite{v2}, generalizing Brenier's theorem in $\R^{n}$
\cite{br} and its Riemannian version in \cite{mc} (and, in particular,
\cite{c-e}, concerning the torus case). After passing to a subsequence,
as in the previous theorem, we may assume that $u^{(k)}\rightarrow u:=u^{(\infty)}$
and that $v^{(k)}\rightarrow v:=(\chi_{X_{\mu}}+u)^{c}.$ In particular,
$v$ is $c-$convex. Denote by $\gamma$ the corresponding optimal
transport plan, furnished by the previous theorem, which is supported
in the subset of $X\times Y$ where $u+v+c=0.$ Hence, it follows
from \cite[Thm 10.41]{v2} (and its proof) that the Borel map $y\mapsto\text{exp}_{y}(\nabla v)$
is the unique optimal transport (Borel) map, pushing forward $\nu$
to $\mu.$ Since $Y$ is assumed to be a topological domain it follows
that $v$ is uniquely determined on $Y,$ modulo additive constants
(see \cite[Remark 10.30]{v2}). Now, by formula \ref{eq:u in theorem as a Legendre transform twice}
we have $u=v^{c}$ and since $u(x_{0})=0$ it follows that $u$ is
uniquely determined. But, as shown in the proof of the previous theorem,
we have $v=(u+\chi_{X_{\nu}})^{c}$ and hence $v$ is also uniquely
determined (i.e. not only determined modulo an additive constant).
Next, since, by assumption, $Y=Y_{\nu}$ formula \ref{eq:u in theorem as a Legendre transform twice}
says that $(u+\chi_{X_{\nu}})^{cc}=u.$ In general, $w^{ccc}=w^{c}$
for any function $w$ and hence it follows that $(u+\chi_{X_{\nu}})^{c}=u^{c}$
which shows that $v=u^{c}.$ 
\end{proof}
The previous theorem applies more generally as soon as a unique Borel
optimal map exists (see for example \cite[Thm 10.38]{v2} for conditions
on $c$ ensuring that this is the case). 

\subsubsection{The torus case: proof of Theorem \ref{thm:conv in static torus setting intr} }

First assume only that the probability measure $\nu$ is absolutely
continuous wrt Lebesgue measure. By the previous theorem $u^{(k)}$
then converges uniformly towards a $c-$convex function $u$ such
that 
\[
(\nabla u^{c}+I)_{*}\nu=\mu
\]
If $\mu$ moreover and $\nu$ have densities $e^{-g}$ and $e^{-g}$
which are Hölder continuous, then it it is well-known that there exists
a unique optimal transport map and its potential (which is uniquely
determined up to a constant) is in $C^{2,\alpha}(T)$ for some $\alpha>0$
(see \cite{c-e} where this is deduced from the regularity results
of Caffarelli $\R^{n}$). It follows that $u^{c}$ and hence also
$u$ is $C^{2}-$smooth and strictly quasi-convex and solves the Monge-Ampère
equation \ref{eq:MA eq intro}. 

\subsection{\label{subsec:The-non-compact-setting}Application to the second
boundary value problem for the Monge-Ampère operator in $\R^{n}$}

Now consider the Euclidean case of Theorem \ref{thm:Static conv in Riem setting},
i.e. the case where $M=\R^{n}$ and $d(x,y)=|x-y|$ and assume that
$X$ and $Y$ are compact convex domains (and, in particular, topological
domains) and take $x_{0}=0.$ As before we also assume that the support
of $\nu$ is equal to $Y$ and that $\nu$ is absolutely continuous
wrt $dx$ and fix discretizations $\mu^{(k)}$ and $\nu^{(k)}$ satisfying
the density property \ref{eq:density property}. 

In order to conform to classical notation in $\R^{n}$ we set $\phi^{(k)}(x):=u^{(k)}(x)+|x|^{2}/2$
and $\psi^{(k)}=v^{(k)}(y)+|y|^{2}$ etc. This corresponds to replacing
the cost function $d^{2}/2$ with 
\[
c(x,y):=-x\cdot y.
\]
 We will use the classical notation $\phi^{*}$ for the corresponding
Legendre transform (formula \ref{eq:def of classical Leg tr}). Then,
by Theorem \ref{thm:Static conv in Riem setting} , $\phi^{(k)}$
converges uniformly to a convex function $\phi$ on $X$ and 
\begin{equation}
(\nabla\psi)_{*}\nu=\mu,\,\,\,\psi:=(\chi_{X}+\phi)^{*}\label{eq:Euclidean setting push for rel}
\end{equation}
We next observe that this means that $\phi$ satisfies the following
Monge-Ampère equation on $\Omega$
\begin{equation}
MA_{\nu}(\phi)=\mu,\label{eq:MA eq in Eucl setting}
\end{equation}
 where $MA_{\nu}(\phi)$ denote the Monge-Ampère measure of $\phi$
relative the target measure $\nu,$ in the sense of Alexandrov. We
recall that if $\nu$ is a given probability measure on $\R^{n}$
which is absolutely continuous wrt $dx$ and $\phi$ is a finite convex
function on a convex open set $\Omega$ in $\R^{n},$ then $MA_{\nu}(\phi)$
is the Borel measure on $\Omega$ defined by
\[
\int_{E}MA_{\nu}(\phi)=\int_{(\partial\phi)(E)}\nu,
\]
 where $E$ is a given Borel set in $\Omega$ and $(\partial\phi)(E)$
denotes the image of $E$ under the multivalued sub-gradient map $\partial\phi$
\cite{v1}. As is well-known, the assumption that $\nu$ is absolutely
continuous wrt $dx$ ensures that $MA_{\nu}(\phi),$ as defined above,
is indeed a measure on $\Omega,$ i.e. countably additive (see \cite[Section 3]{m-t-w}
for a more general setting involving a cost function $c).$ Moreover,
if $\phi$ is $C^{2}-$smooth and strictly convex, then making the
change of variables $y=\nabla\phi(x)$ reveals that
\begin{equation}
MA_{\nu}(\phi)=\rho_{\nu}(\nabla\phi)\det(\nabla^{2}\phi),\label{eq:MA measure in Eucl setting in terms of density}
\end{equation}
 where$\rho_{\nu}$ denotes the density of $\nu$ wrt $dx.$ We will
use the following general representation of the the Monge-Ampère measure
(see \cite[Section 2.2]{ber2} and references therein). 
\begin{lem}
Let $\phi$ be a bounded finite convex function in $\Omega.$ Then
\[
MA_{\nu}(\phi)=(\nabla\psi)_{*}\nu,\,\,\,\psi:=(\chi_{X}+\phi)^{*}
\]
Hence, by formula \ref{eq:Euclidean setting push for rel}, the limit
$\phi$ of $\phi^{(k)}$ indeed solves the Monge-Ampère equation \ref{eq:MA eq in Eucl setting}.
It should be stressed that, in general, there can be many different
convex solutions to the Monge-Ampère equation (which not only differ
by an additive constant). But the point is that the Sinkhorn iteration
singles out a particular solution $\phi.$ On the other hand, when
$Y$ is convex we have the following well-known uniqueness result:

Let $\Omega$ be a convex open set in $\R^{n}$ and $\mu$ and $\nu$
probability measures on $\R^{n}$ such that $\mu$ is supported in
$\Omega,$ $\nu$ is absolutely continuous wrt $dx$ and the support
$Y$ of $\nu$ is a convex body, i.e. a compact convex set with non-empty
interior. Then a solution $\phi$ to the \emph{second boundary value
problem}
\begin{equation}
MA_{\nu}(\phi)=\mu,\,\,\,(\partial\phi)(\Omega)\subset Y\label{eq:secon bound value pr}
\end{equation}
 is uniquely determined up to an additive constant. 
\end{lem}

\begin{proof}
A simple proof is given in \cite{ber2}, which we recall here since
it fits well with the spirit of the present paper (see also \cite[Thm 3.1]{m-t-w}
for a different proof in the setting of more general cost function).
The starting point is the observation that any convex function $\phi$
on an open convex set $\Omega$ with the property that $(\partial\phi)(\Omega)\subset Y,$
for a convex body $Y,$ may be expressed as $\phi=\left(\chi_{Y}+(\chi_{X}+\phi)^{*}\right).$
But, by the previous lemma, the equation $MA_{\nu}(\phi)=\mu$ implies
that the function $\psi:=(\chi_{X}+\phi)^{*}$ in $Y$ is uniquely
determined up to an additive constant. Hence, so is $\phi.$ 
\end{proof}
We thus arrive at the following result, which also applies to the
second boundary value problem in all of $\R^{n}:$
\begin{thm}
\label{thm:non-cpt static}Let $X$ and $Y$ be compact convex domains
in $\R^{n}$ endowed with probability measure $\mu$ and $\nu$ respectively.
Assume that the support of $\nu$ is equal to $Y$ and that $\nu$
is absolutely continuous wrt $dx.$ Let $\mu^{(k)}$ and $\nu^{(k)}$
be sequences of probability measures on $X$ and $Y$ converging weakly
towards $\mu$ and $\nu$ respectively and satisfying the density
property \ref{eq:density property}. Denote by $\phi^{(k)}$ the unique
normalized fixed point of the scaled log Sinkhorn operator on $C(X)$
corresponding to the cost function $c(x,y)=-x\cdot y.$ Then 
\begin{equation}
\phi^{(k)}\rightarrow\phi,\,\,\,k\rightarrow\infty\label{eq:conv in thm non cpt}
\end{equation}
uniformly on $X,$ where $\phi$ is the unique normalized convex solution
to the second boundary value problem \ref{eq:secon bound value pr}
in the interior of $X.$ More generally, the corresponding result
holds when $X$ is replaced by $\R^{n}$ under the assumption that
there exists a compact subset containing the support of $\mu^{(k)}$
for all $k.$ Then the corresponding convergence \ref{eq:conv in thm non cpt}
is uniform on compact subsets of $\R^{n}.$
\end{thm}

\begin{proof}
First assume that $X$ is compact. As explained above it then follows
from Theorem \ref{thm:Static conv in Riem setting} that $\phi^{(k)}$
converges uniformly to a normalized convex function $\phi$ on $X,$
satisfying $MA_{\nu}(\phi)=\mu.$ Next, note that, it follows directly
from the definition of the fixed point $\phi^{(k)}$ that $\phi^{(k)}$
may be expressed as
\[
\phi^{(k)}(x)=k^{-1}\log\int_{Y}e^{kx\cdot y}\nu'(y)
\]
 for a measure $\nu'$ supported on $Y.$ As a consequence, $\phi^{(k)}$
is smooth and $\nabla\phi^{(k)}$ is contained in the convex hull
of $Y,$ i.e. in $Y,$ since $Y$ is assumed to be convex. But then
it follows that $(\partial\phi)(\Omega)\subset Y,$ which concludes
the proof in the case when $X$ is compact. Note that, alternatively,
we could have used directly that, by Theorem \ref{thm:Static conv in Riem setting},
the limit $\phi$ satisfies $\phi=\left(\chi_{Y}+(\chi_{X}+\phi)^{*}\right),$
which is the unique normalized solution to the second boundary problem
in question (by the proof of the previous lemma). To prove the non-compact
case when $X$ is replaced by $\R^{n}$ denote by $X_{R}$ the ball
or radius $R$ centered at $0$ and assume that $R$ is sufficiently
large to ensure that $\mu$ and all $\mu^{(k)}$ are supported in
the interior of $X_{R}.$ Denote by $\phi_{R}^{(k)}(x)$ the normalized
fixed point of the corresponding iteration on $C(X_{R}).$ Since $\mu^{(k)}$
is supported in $X_{R}$ it follows, from the very definition of the
iteration, that $\phi_{R}^{(k)}(x)$ is, in fact, independent of $R.$
Accordingly, we may define a normalized convex function $\phi^{(k)}$
on $\R^{n}$ by setting $\phi^{(k)}:=\phi_{R}^{(k)}$ on $X_{R}$
for any $R$ sufficiently large. Then $\phi^{(k)}$ is the unique
normalized fixed point of the corresponding iteration on $C(\R^{n}).$
Moreover, since $X_{R}$ is compact $\phi^{(k)}$ converges uniformly
on $X_{R}$ to a convex function $\phi$ solving the second boundary
value problem \ref{eq:secon bound value pr} on $X_{R}.$ Since $R$
is arbitrary this shows that $\phi,$ in fact, solves the problem
on all of $\R^{n}.$
\end{proof}

\subsection{\label{subsec:Application-to-convolutional}Application to convolutional
Wasserstein distances}

Theorem \ref{Thm:weak conv of fixed points towards optimal transport plans}
holds more generally (with essentially the same proof) when the function
$c$ is replaced by a sequence $c_{k}$ such that 
\begin{equation}
\left\Vert c_{k}-c\right\Vert _{L^{\infty}(X\times Y)}\rightarrow0\label{eq:unif conv of ck}
\end{equation}
For example, in the Riemannian setting of Theorem \ref{thm:Static conv in Riem setting}.
denoting by $\mathcal{K}_{t}(x,y)$ the corresponding heat kernel
and setting $t:=2k^{-1},$ the sequence 
\begin{equation}
c_{k}:=-t^{-1}\log\mathcal{K}_{t}(x,y)\label{eq:c k as heat}
\end{equation}
satisfies \ref{eq:unif conv of ck}, by Varadhan's formula (which
holds more generally on Lipschitz Riemannian manifolds \cite{no}).
Replacing $c$ by $c_{k}$ in this setting thus has the effect of
replacing the matrix $A_{ij}:=e^{-kd^{2}(x_{i},x_{j})/2}$ appearing
in the corresponding Sinkhorn algorithm with the heat kernel matrix
$\mathcal{K}_{2k^{-1}}(x_{i},x_{j})$ which, as emphasized in \cite{s-d---},
has computational advantages. Following \cite{s-d---} we consider
the squared \emph{convolutional Wasserstein distance }between $\mu$
and $\nu:$
\[
\mathcal{W}_{(k)}^{2}(\mu,\nu):=k^{-1}\inf_{\gamma\in\Pi(\mu^{(k)},\nu^{(k})}\mathcal{I}(\gamma,\mathcal{K}_{2k^{-1}}\mu^{(k)}\otimes\nu^{(k)}),
\]
 defined wrt approximations $\mu^{(k)}$ and $\nu^{(k)},$ for example
given by weighted point clouds, as in Example \ref{exa:The-density-property}.
In \cite[Page 3]{s-d---}, the problem of developing conditions for
the convergence of $\mathcal{W}_{(k)}^{2}(\mu,\nu)$ was posed. The
following result provides an answer:
\begin{thm}
\label{thm:conv wasserst}Let $X$ be a compact Riemannian manifold
(possibly with boundary) and set $c(x,y):=d(x,y)^{2}/2,$ where $d$
is the Riemannian distance function. Suppose that $\mu^{(k)}\rightarrow\mu$
and $\nu^{(k)}\rightarrow\mu$ in $\mathcal{P}(X)$ and that $\mu^{(k)}$
and $\nu^{(k)}$ satisfy the density property \ref{eq:density property}.
Then 
\[
\lim_{k\rightarrow\infty}\mathcal{W}_{(k)}^{2}(\mu,\nu)=\mathcal{W}^{2}(\mu,\nu),
\]
 where $\mathcal{W}^{2}(\mu,\nu)$ denotes the squared $L^{2}-$Wasserstein
distance between $\mu$ and $\nu.$ 
\end{thm}

\begin{proof}
Repeating the argument in the proof of Theorem \ref{Thm:weak conv of fixed points towards optimal transport plans},
with $c$ replaced by $c_{k}$ as above, gives 
\[
\lim_{k\rightarrow\infty}\inf_{u\in C^{0}(X)}\mathcal{F}^{(k)}=\inf_{u\in C(X)}J(u)
\]
 According to formula \ref{eq:inf u equal to inf u} the infimum appearing
in the left hand side above is precisely $\mathcal{W}_{(k)}^{2}(\mu,\nu).$
Since the infimum in the right hand side above is equal to $\mathcal{W}^{2}(\mu,\nu),$
by Kantorovich duality (formula \ref{eq:Kant duality}), the result
follows.
\end{proof}

\section{\label{sec:Convergence-of-the iter torus}Convergence of the iteration
towards parabolic optimal transport equations on the torus}

\subsection{\label{subsec:Proof-of-Theorem}Proof of Theorem \ref{thm:conv in dynamic torus setting intro}}

We will denote by $\delta_{\Lambda_{k}}$ the uniform discrete probability
measure supported on the discrete torus $\Lambda_{k}$ with edge-length
$1/k:$ 
\[
\delta_{\Lambda_{k}}:=\frac{1}{N_{k}}\sum_{x_{i}\in\Lambda_{k}}\delta_{x_{i}}
\]
Given two probability measures $\mu=e^{-f}dx$ and $\nu=e^{-g}dy$
we can then define their discretizations as the probability measures
\[
\mu^{(k)}=\frac{1}{\int_{T^{n}}e^{-f}\delta_{\Lambda_{k}}}e^{-f}\delta_{\Lambda_{k}},\,\,\,\nu^{(k)}=\frac{1}{\int_{T^{n}}e^{-g}\delta_{\Lambda_{k}}}e^{-g}\delta_{\Lambda_{k}}
\]
Note that the normalization constants have the asymptotics 
\begin{equation}
\left|\int_{T^{n}}e^{-f}\delta_{\Lambda_{k}}-1\right|\leq C_{f}k^{-1},\,\,\,\,\left|\int_{T^{n}}e^{-g}\delta_{\Lambda_{k}}-1\right|\leq C_{g}k^{-1}\label{eq:asympt of norming constants}
\end{equation}
 where $C$ only depends on an upper bound on $|\nabla f|$ on $T^{n}$
(and similarly for $g$); see \cite[Page 2]{b-c-c-g-s-t}. 
\begin{rem}
As will be clear from the proof, Theorem \ref{thm:conv in dynamic torus setting intro}
also holds with the simpler discretizations $e^{-f}\delta_{\Lambda_{k}}$
and $e^{-g}\delta_{\Lambda_{k}}$ (which, in general, are not probability
measures).
\end{rem}

We start with the following discrete version of the classical Laplace
method of integration, proved in Appendix A: 
\begin{lem}
\label{lem:discrete stat phase}Let $\alpha$ be a lower-semicontinuous
(lsc) function on $T^{n}$ with a unique minimum at $x_{0}$ and assume
that $\alpha$ is $C^{4}-$smooth on a neighborhood of $x_{0}$ and
$\nabla^{2}\alpha(x_{0})>0.$ Then, if $h$ is $C^{2}-$smooth 
\[
k^{n/2}\int e^{-k\alpha}h\delta_{\Lambda_{k}}=(2\pi)^{n/2}e^{-k\alpha(x_{0})}\frac{h(x_{0})}{\sqrt{\det(\nabla^{2}\alpha(x_{0}))}}(1+\epsilon_{k}),\,\,\,|\epsilon_{k}|\leq Ck^{-1}
\]
where the constant $C$ only depends on an upper bounds on the $C^{4}-$norm
of $\alpha,$ the $C^{2}-$norm of $h$ and a strict lower bound on
the smallest eigenvalue of $\nabla^{2}\alpha$ on a neighborhood of
$x_{0}.$ 
\end{lem}

We next prove the key asymptotic result that will be used in the proof
of Theorem \ref{thm:conv in dynamic torus setting intro} giving the
asymptotics of the function $\rho_{ku}(x),$ defined by formula \ref{eq:explicit expres for pho}
(the result can be viewed as a refinement of Lemma \ref{lem:dens prop}).
\begin{prop}
\label{prop:asympt of pho}Let $u$ be a strictly quasi-convex function
in $C^{4}(T^{n}).$ Then the following asymptotics hold
\[
\rho_{ku}(x)=\det(I+\nabla^{2}u(x))e^{f(x)-g(x+\nabla u(x))}(1+\epsilon_{k}),\,\,\,|\epsilon_{k}|\leq Ck^{-1}
\]
where the constant $C$ only depends on upper bounds on the $C^{4}-$norm
of $u$ and the $C^{2}-$norms of $f$ and $g$ and a strict positive
lower bound on the eigenvalues of the matrix $\left(I+\nabla^{2}u(x)\right).$
\end{prop}

\begin{proof}
First recall that, by Lemma \ref{lem:quasi conv}, $u^{c}$ is also
strictly quasi-convex and in $C^{4}(T^{n}).$ Set $c(x,y)=d_{T^{n}}^{2}(x,y)/2.$
In the proof we will denote by $O(k^{-1})$ any sequence of functions
satisfying $|O(k^{-1})|\leq Ck^{-1},$ for a constant $C$ depending
on data as in the statement of the proposition. Fix $y\in Y.$ The
function $x\mapsto c(x,y)+u(x)$ is lsc on $T^{n}$ and has a unique
minimum $x_{y}$ and it is $C^{4}-$smooth close to $x_{y}$ (by Lemma
\ref{lem:smooth}). Applying the asymptotics \ref{eq:asympt of norming constants}
and Lemma \ref{lem:discrete stat phase} thus gives 
\[
k^{n/2}e^{kv^{(k)}[u](y)}:=k^{n/2}\int e^{-k\left(c(x,y)+u(x)\right)}\mu^{(k)}(x)=e^{ku^{c}(y)}\left(h(y)+O(1)k^{-1}\right),
\]
 where 
\[
h(y):=(2\pi)^{n/2}\frac{\exp(-f(x_{y}))}{\sqrt{\det(I+\nabla^{2}u(x_{y}))}}.
\]
As a consequence, the inverse of $k^{n/2}e^{kv^{(k)}[u](y)}$ is given
by $e^{-ku^{c}(y)}\left(h(y)^{-1}+O(1)k^{-1}\right)$ and hence
\[
e^{ku^{(k)}\left[v^{(k)}[u]\right](x)}:=k^{n/2}\int e^{-kc(x,y)}\left(\frac{1}{k^{n/2}e^{kv^{(k)}[u](y)}}\right)\nu^{(k)}(y)=
\]
\begin{equation}
=k^{n/2}\int e^{-k\left(c(x,y)+u(x)+u^{c}(y)\right)}h(y)^{-1}\nu^{(k)}(y)+R_{k}(x),\label{eq:rho in pf prop rhp}
\end{equation}
 where 
\[
R_{k}(x)\leq O(k^{-1})k^{n/2}\int e^{-k\left(c(x,y)+u(x)+u^{c}(y)\right)}\nu^{(k)}(y).
\]
By Lemma \ref{lem:discrete stat phase} (and the asymptotics \ref{eq:asympt of norming constants})
we have 
\[
k^{n/2}\int e^{-k\left(c(x,y)+u(x)+u^{c}(y)\right)}\nu^{(k)}(y)\leq C'
\]
and hence it follows that 
\[
R_{k}(y)=O(k^{-1}).
\]
Now, the same localization argument as above shows that the integral
over $\nu^{(k)}(y)$ in formula \ref{eq:rho in pf prop rhp} localizes
around a small neighborhood $V$ of $y=y_{x}.$ Hence, applying Lemma
\ref{lem:discrete stat phase} (and the asymptotics \ref{eq:asympt of norming constants})
again gives 
\[
k^{n/2}\int e^{-k(d(x,y)^{2}/2+u(x)+u^{c}(y))}h(y)^{-1}\nu^{(k)}(y)=e^{k\left((u^{c})^{c}\right)(x)}h^{-1}(y_{x})\frac{(2\pi)^{n/2}\exp(-g(y_{x}))}{\sqrt{\det(I+\nabla^{2}u^{c}(y_{x}))}}(1+O(k^{-1})).
\]
All in all, since $\left((u^{c})^{c}\right)=u,$ this shows that 
\[
e^{ku^{(k)}\left[v^{(k)}[u]\right](x)}=e^{ku(x)}h^{-1}(y_{x})\frac{(2\pi)^{n/2}\exp(-g(y_{x}))}{\sqrt{\det(I+\nabla^{2}u^{c}(y_{x}))}}(1+O(k^{-1})).
\]
The proof is thus concluded by invoking the inverse properties of
the Hessians in Lemma \ref{lem:quasi conv}. 
\end{proof}
\begin{lem}
\label{lem:general conv towards parab}Let $X$ be a compact topological
space and consider the following family of difference equations on
$C(X),$ parametrized by a positive number $k$ and a discrete time
$m:$ 
\begin{equation}
u_{m+1}^{(k)}-u_{m}^{(k)}=k^{-1}D^{(k)}(u_{m}^{(k)}),\label{eq:differ eq abs}
\end{equation}
where $D^{(k)}$ is an operator on $C(X),$ which descends to $C(X)/\R$
and with the property that $I+k^{-1}D^{(k)}$ is an increasing operator
(wrt the usual order relation on $C(X)).$ Assume that there exists
a subset $\mathcal{H}$ of $C(X)$ and an operator $D$ on $\mathcal{H}$
such that for any $u$ in the class 
\begin{equation}
\left|D^{(k)}(u)-D(u)\right|\leq C_{u}\epsilon_{k},\label{eq:conv of D k}
\end{equation}
 where $C_{u}$ is a positive constant depending on $u$ and $\epsilon_{k}$
is a sequence of positive numbers converging towards zero. Assume
that $u_{0}^{(k)}=u_{0}$ where $u_{0}$ is a fixed function in $\mathcal{H}$
and that there exists a family $u_{t}\in\mathcal{H},$ which is two
times differentiable wrt $t$ and solving 
\begin{equation}
\frac{\partial u_{t}}{\partial t}=D(u_{t}),\,\,\,\,(u_{t})_{|t=0}=u_{0},\label{eq:evol eq for u t abs}
\end{equation}
for $t\in[0,T].$ Then, for any $(k,m)$ such that $m/k\in[0,T],$
\[
\sup_{T^{n}}\left|u_{m}^{(k)}-u_{m/k}\right|\leq C_{T}\frac{m}{k}\cdot2\max\{\epsilon_{k},k^{-1}\},\,\,\,C_{T}:=\max\{\sup_{t\in[0,T]}C_{u_{t}},\sup_{X\times[0,T]}\frac{\partial^{2}u_{t}}{\partial^{2}t}\}.
\]
\end{lem}

\begin{proof}
We will write $\psi_{k,m}=u_{m/k}.$ 

\emph{Step 1: The following holds for all $(k,m)$ }
\[
\sup\left|\psi_{k,m+1}-\psi_{k,m}-k^{-1}D^{(k)}(\psi_{k,m})\right|\leq k^{-1}C_{T}\epsilon'_{k},\,\,\,\epsilon'_{k}:=2\max\{\epsilon_{k},k^{-1}\}.
\]

Indeed, using the mean value theorem we can write 
\[
\psi_{k,m+1}-\psi_{k,m}=\frac{1}{k}(\frac{u_{m/k+1/k}-u_{m/k}}{1/k})=\frac{1}{k}(\frac{\partial u_{t}}{\partial t}_{|t=m/k}+O(k^{-1})),
\]
 where the term $O(k^{-1})$ may be estimated as $|O(k^{-1})|\leq A_{T}k^{-1},$
where $A_{T}=\sup_{X\times[0,T]}|\frac{\partial^{2}u_{t}}{\partial^{2}t}|.$
Using the evolution equation for $u_{t}$ and applying formula \ref{eq:conv of D k}
thus proves Step 1.

\emph{Step 2: The discrete evolution on $C(X)$ defined by the difference
equation \ref{eq:differ eq abs} decreases the sup-norm}, i.e. 
\[
\sup_{X}|\phi_{m+1}-\psi_{m+1}|\leq\sup_{X}|\phi_{m}-\psi_{m}|
\]
 if $\phi_{m}$ and $\psi_{m}$ satisfy the difference equation \ref{eq:differ eq abs}
for a fixed $k.$ 

To see this set $C:=\sup|\phi_{m}-\psi_{m}|.$ Then, $\phi_{m}\leq\psi_{m}+C$
and hence, since $I+k^{-1}D^{(k)}$ is assumed to be increasing,
\[
\phi_{m+1}=\phi_{m}+k^{-1}D^{(k)}(\phi_{m})\leq\psi_{m}+C+k^{-1}D^{(k)}(\psi_{m}+C)=\psi_{m+1}+C
\]
 In particular, $\sup(\phi_{m+1}-\psi_{m+1})\leq C.$ Applying the
same argument with the roles of $\phi$ and $\psi$ interchanged concludes
the proof.

\emph{Step 3: Conclusion:
\begin{equation}
\sup_{X}|u_{m}^{(k)}-\psi_{k,m}|\leq C_{T}\frac{m}{k}\epsilon'_{k}.\label{eq:induction hyp}
\end{equation}
}We will prove this by induction over $m$ (for $k$ fixed) the statement
being trivially true for $m=0.$ We fix the integer $k$ and assume
as an induction hypothesis that \ref{eq:induction hyp} holds for
$m$ with $C$ the constant in the previous inequality. Applying first
Step 2 and then the induction hypothesis thus gives 

\[
\sup_{X}|\left(\psi_{k,m}+k^{-1}D^{(k)}(\psi_{k,m})\right)-\left(u_{m}^{(k)}+k^{-1}D^{(k)}(u_{m}^{(k)})\right)|\leq\sup_{X}|\psi_{k,m}-u_{m}^{(k)}|\leq C_{T}\frac{m}{k}\epsilon'_{k}.
\]
Now, by Step 1, 
\[
\sup_{X}|\psi_{k,m+1}-(\psi_{k,m}+k^{-1}D^{(k)}(\psi_{k,m})|\leq\frac{C_{T}}{k}\epsilon'_{k}
\]
for all $(m,k).$ Hence, 
\[
\sup_{X}|\psi_{k,m+1}-u_{m+1}^{(k)}|\leq C_{T}\frac{m}{k}\epsilon'_{k}+C_{T}\frac{1}{k}\epsilon'_{k}=C_{T}\frac{(m+1)}{k}\epsilon'_{k},
\]
proving the induction step and hence the final Step 3.
\end{proof}
In the present setting $\mathcal{H}$ will be taken as subspace of
$C^{4}(T^{n})$ consisting of all strictly quasi-convex functions
$u$ and 
\[
D(u)(x):=\log(\det(\nabla^{2}u(x)+I))-g(\nabla u(x)+x)+f(x).
\]
The following proposition essentially follows from the results in
\cite{s-s,ki,k-s-w} (for completeness a proof is provided in Appendix
B). The space $C^{k,\alpha}(T^{n})$ denotes, as usual, space of all
functions such that the $k$th order derivatives are Hölder continuous
with Hölder exponent $\alpha\in]0,1[.$ 
\begin{prop}
\label{prop:existence and conv parab torus}Let $f$ and $g$ be two
given functions in $C^{2,\alpha}(T^{n})$. Then, for initial data
$u_{0}\in C^{4,\alpha}(T^{n})$ which is strictly quasi-convex there
exists a solution $u(x,t)$ to the corresponding parabolic PDE \ref{eq:parabolic eq intro}
in $C^{2}([0,\infty[\times T^{n})$ such that for any $t>0$ $u_{t}\in C^{4}(T^{n})$
and $u_{t}$ is strictly quasi-convex. Moreover, 
\begin{itemize}
\item There exists a positive constant $C_{1}$ - only depending on upper
bounds on the $C^{2}-$norms of $u_{0},f$ and $g$ and a strict positive
lower bound on the eigenvalues of the matrix $(I+\nabla^{2}u_{0}(x))$
- such that 
\[
C_{1}^{-1}I\leq\left(\nabla^{2}u_{t}(x)+I\right)\leq C_{1}I
\]
 and a constant $C_{2}$ which, moreover, depends on the $C^{4,\alpha}-$norm
of $u_{0}$ and the $C^{2,\alpha}-$norms of $f$ and $g$ such that
\[
\left\Vert u_{t}\right\Vert _{C^{4}(T^{n})}\leq C_{2}
\]
\item There exist positive constants $A_{1}$ and $A_{0}$ such that 
\begin{equation}
\sup_{T^{n}}|u_{t}-u_{\infty}|\leq A_{1}e^{-t/A_{0}}\label{eq:exp conv}
\end{equation}
where $u_{\infty}$ is a potential solving the corresponding optimal
transport problem (i.e. a solution to the corresponding stationary
equation).
\item if $u_{0},f$ and $g$ are in $C^{\infty}(T^{n}),$ then so is $u_{t}$
and $\left\Vert u_{t}\right\Vert _{C^{s}(T^{n})}\leq C_{s}$  for
any positive integer $s.$
\end{itemize}
\end{prop}

Note that differentiating the parabolic PDE \ref{eq:parabolic eq intro}
gives
\[
\frac{\partial^{2}u_{t}}{\partial^{2}t}=L_{u_{t}}[D(u_{t})],
\]
 where $L_{u_{t}}$ is the linearization of $D$ at $t_{t}.$ By the
first point in the previous theorem $L_{u_{t}}$ is a uniformly elliptic
second order operator (see formula \ref{eq:form for L u} in Appendix
B) and hence 
\[
\sup_{X\times[0,\infty[}|\frac{\partial^{2}u_{t}}{\partial^{2}t}|\leq C_{3}
\]
 where $C_{3}$ only depends on the constants $C_{1}$ and $C_{2}$
in the previous proposition and the $C^{2}-$norms of $f$ and $g.$

\subsubsection{Conclusion of proof of Theorem \ref{thm:conv in dynamic torus setting intro}
and Corollary \ref{cor:conv of explicit apprl intro}}

Consider the log Sinkhorn iteration and define $D^{(k)}(u)$ to be
$\log(\rho_{ku}),$ as in equation \ref{eq:difference equation}.
The proof of Theorem \ref{thm:conv in dynamic torus setting intro}
follows directly from combining Lemma \ref{lem:general conv towards parab}
with Propositions \ref{prop:asympt of pho}, \ref{prop:existence and conv parab torus}
and also using that the corresponding operator $S^{(k)}=I+k^{-1}D^{(k)},$
defined by formula \ref{eq:scaled log Sinkhor op}, is clearly increasing
and invariant under the additive $\R-$action. Finally, Corollary
\ref{cor:conv of explicit apprl intro} follows by combining Theorem
\ref{thm:conv in dynamic torus setting intro} with the exponential
convergence in formula \ref{eq:exp conv}. Indeed, setting $m_{k}=kt_{k}$
with $t_{k}:=A\log k,$ where $A$ is the constant appearing formula
\ref{eq:exp conv}, gives 
\[
\left\Vert u_{m_{k}}^{(k)}-u_{\infty}\right\Vert _{C(X)}\leq\left\Vert u_{m_{k}}^{(k)}-u_{t_{k}}\right\Vert _{C(X)}+\left\Vert u_{t_{k}}-u_{\infty}\right\Vert _{C(X)}\leq ACk^{-1}\log k+A_{1}(k^{-1})^{A/A_{0}}
\]
 where $C$ is bounded uniformly from above, independently of $t$
as follows from \ref{thm:conv in dynamic torus setting intro} (using
the uniform bounds on $u_{t}$ in Prop \ref{prop:existence and conv parab torus},
which ensure that $C$ is bounded from above, independently of $t).$
Setting $u_{k}:=u_{m_{k}}^{(k)}$ this proves the estimate \ref{eq:estimate in cor torus intro}
when $A>A_{0}.$

Next note that the estimate \ref{eq:estimate in cor torus intro}
implies the estimate \ref{eq:estimate for gamma k cor torus intro}
for $\gamma_{k}.$ Indeed, by definition, $\gamma_{k}:=e^{-kd_{T^{n}}^{2}(x,y)^{2}/2}e^{-ku_{k}(x)}e^{-kv_{k}(y)}\mu^{(k)}\otimes\nu^{(k)},$
where $v_{k}:=v[u_{k}].$ By the estimate \ref{eq:estimate in cor torus intro}
and Lemma \ref{lem:local dens prop in model case} there exists a
positive number $C$ such that 
\[
u_{k}(x)+v_{k}(y)\geq u(x)+u^{c}(y)+Ck^{-1}\log k.
\]
 The proof is thus concluded by invoking the following elementary
inequality, which we claim holds for any strictly quasi-convex function
$u$ in $C^{2}(T^{n}):$ 
\begin{equation}
d_{T^{n}}^{2}(x,y)^{2}/2+u(x)+u^{c}(y)\geq\frac{\delta}{2}d_{T^{n}}(y,F_{u}(x))^{2},\,\,\,F_{u}(x)=x+(\nabla u)(x),\label{eq:ineq using strict conv}
\end{equation}
under the assumption that 
\[
0<\delta I\leq\left(\nabla^{2}u(x)+I\right)\leq\delta^{-1}I.
\]
To see this we identify $u$ and $u^{c}$ with $\Z^{n}-$periodic
functions on $\R^{n}$ and set $\phi(x):=u(x)+|x|^{2}/2.$ Then $\phi^{*}(y)=u^{c}(y)+|y|^{2}/2,$
where $\phi^{*}$ is the classical Legendre-transform on $\R^{n}$
(compare with the proof of Lemma \ref{lem:quasi conv}). Hence, it
will be enough to show that 
\begin{equation}
-x\cdot y+\phi(x)+\phi^{*}(y)\geq\frac{\delta}{2}|y-\nabla\phi(x)|^{2}.\label{eq:ineq using strict ineq in phi}
\end{equation}
 Indeed, the claimed inequality \ref{eq:ineq using strict conv} follows
from the latter one after replacing $x$ with $x+m$ and taking the
infimum over all $m\in\Z^{n}.$ Now, by assumption $\nabla^{2}\phi\leq\delta^{-1}I$
and hence $\nabla^{2}\phi^{*}\geq\delta I$ (by \ref{eq:matrix relation for u}).
As a consequence, $\phi^{*}(y)\geq\phi^{*}(y-t)+t\cdot\nabla\phi^{*}(y-t)+\delta|t|^{2}/2$
for any $t\in\R^{n}.$ Setting $t:=y-\nabla\phi(x)$ and using that
$\phi^{*}(\nabla\phi(x))=\nabla\phi(x)\cdot x-\phi(x)$ and $(\nabla\phi^{*})(\nabla\phi(x))=x$
this implies the desired inequality \ref{eq:ineq using strict ineq in phi}.

\subsubsection{Proof of Corollary \ref{cor:cons transport cost}}

By Prop \ref{prop:optim crit} $C(\mu,\nu)=-J(u),$ where $u$ is
an optimal transport potential and $J$ is the Kantorovich functional.
Now, since $u$ and $u^{c}$ are Lipschitz continuous (with Lipschitz
constant $\sqrt{n}$) we can approximate the integrals defining $J(u)$
to get 
\[
-C(\mu,\nu)=J(u):=\int u\mu+\int u^{c}\nu=\int u\mu^{(k)}+\int u^{c}\nu^{(k)}+O(k^{-1})
\]
(see \cite[Page 2]{b-c-c-g-s-t}). Next, by Corollary\ref{cor:conv of explicit apprl intro}
$u=u_{m_{k}}^{(k)}+O(k^{-1}\log k).$ Similarly, applying Corollary
\ref{cor:conv of explicit apprl intro} to the situation where the
roles of $\mu$ and $\nu$ have been reversed shows that $v=v_{m_{k}}^{(k)}+O(k^{-1}\log k)$
and hence 
\[
C(\mu,\nu)=-\sum_{x_{i}\in\Lambda_{k}}p_{i}u_{m_{k}}^{(k)}(x_{i})-\sum_{y_{i}\in\Lambda_{k}}q_{i}v_{m_{k}}^{(k)}(y_{i})+O(k^{-1}\log k)
\]
The proof is thus concluded by expressing $u_{m_{k}}^{(k)}$ and $v_{m_{k}}^{(k)}(y_{i})$
in terms of $a^{(k)}(m)$ and $b^{(k)}(m),$ respectively (formulas
\ref{eq:change of variables u a}, \ref{eq:chang of variables v intro})
and using that $k^{-1}\log p_{i}=k^{-1}\log\left(e^{-f}(x_{i})(1+O(k^{-1})\right)=O(k^{-1})$
and similarly for $q_{i}.$ This proves the first statement in Corollary
\ref{cor:cons transport cost}. The last statement then follows from
the observation that 
\[
\frac{1}{2}d(\mu^{(k)},\nu^{(k)})^{2}=\frac{1}{2}d(\mu,\nu)^{2}+O(k^{-1})
\]
This observation is without doubt standard but, for completeness,
we provide a proof. Denote by $X_{k}$ and $Y_{k}$ the support of
$\mu^{(k)}$ and $\nu^{(k)},$ respectively. By definition, $\frac{1}{2}d(\mu^{(k)},\nu^{(k)})^{2}$
is the sup of $\left\langle c,\gamma\right\rangle $ over all probability
measures on $X_{k}\times Y_{k}$ with marginals $\mu^{(k)}$ and $\nu^{(k)},$
respectively (with the usual identifications). Since any probability
measure $\gamma$ on $T^{n}\times T^{n}$ with marginals $\mu^{(k)}$
and $\nu^{(k)}$ is automatically supported on $X_{k}\times Y_{k}$
we may as well take the sup over all such $\gamma.$ Hence, $\frac{1}{2}d(\mu^{(k)},\nu^{(k)})^{2}=C(\mu^{(k)},\nu^{(k)})$
on $T^{n}\times T^{n}.$ Kantorovich duality then gives
\[
-\frac{1}{2}d(\mu^{(k)},\nu^{(k)})^{2}=\inf_{u:\,u^{cc}=u}\left(\int u\mu^{(k)}+\int u^{c}\nu^{(k)}\right)=\inf_{u:\,u^{cc}=u}\left(\int u\mu+\int u^{c}\nu\right)+O(k^{-1}),
\]
where the last equality follows from the argument in the beginning
of the proof of the corollary. Using Kantorovich duality again thus
concludes the proof.

\subsection{\label{subsec:Comparison-with-previous sinkhorn rate}Comparison
with previous convergence rates for the Sinkhorn algorithm and sharpness}

As pointed out in Remark \ref{rem:lin convergence rate} it is well-known
that, for $k$ fixed, $u_{m}^{(k)}$ converges exponentially in $C(X)/\R$
towards a fixed point $u_{\infty}^{(k)}:$ 
\[
\left\Vert u_{m}^{(k)}-u_{\infty}^{(k)}\right\Vert _{C(X)/\R}\leq C_{0}e^{-\delta_{k}m}
\]
(i.e. it converges at \emph{linear rate} in the terminology of numerical
analysis). The general estimate on $e^{-\delta_{k}}$ provided in
\cite{f-l} gives that $1-e^{-\delta_{k}}$ is comparable to $1-2e^{-k}$
if we assume, for simplicity, that $c$ is Lipschitz continuous with
Lipschitz constant $1$ and that the diameter of $X$ is equal one
(see \cite{via}). Hence, $\delta_{k}$ is comparable to $e^{-k},$
which means that, for $k$ large, on would needs to run $O(e^{k})$
iterations to get close to the fixed point $u_{\infty}^{(k)}.$ On
the order hand experimental findings reported in \cite[Remark 4.15]{m-p}
and \cite{fe} suggest that in the ``geometric settings'', such
as in present torus setting, $\delta_{k}$ is of the order $O(k^{-1}).$
Hence, only $O(k)$ iterations should be needed to get close to the
limiting fixed point (see also \cite[Prop 18]{sc}, where this is
shown for a modified asymmetric version of the Sinkhorn algorithm,
inspired by the auction algorithm). This is confirmed by Theorem \ref{thm:conv in dynamic torus setting intro}
(and the proof of Corollary \ref{cor:conv of explicit apprl intro}).
Indeed, setting $m_{k}=t_{k}k$ and applying Theorem \ref{thm:conv in dynamic torus setting intro}
gives
\[
\left\Vert u_{m_{k}}^{(k)}-u_{\infty}^{(k)}\right\Vert _{C(X)}\leq Ct_{k}k^{-1}+\left\Vert u_{t_{k}}-u_{\infty}\right\Vert _{C(X)}+\left\Vert u_{t_{k}}-u_{\infty}\right\Vert _{C(X)}
\]
Hence, by the exponential convergence in Prop \ref{prop:existence and conv parab torus}
and Theorem \ref{thm:conv in static torus setting intr}
\[
\left\Vert u_{m_{k}}^{(k)}-u_{\infty}^{(k)}\right\Vert _{C(X)}\leq Ct_{k}k^{-1}+Ae^{-t_{k}/A}+\epsilon_{k}.
\]
 In particular, by taking $m_{k}\sim Tk$ for $T$ sufficiently large
fixed number one can make the right hand side above arbitrarily small,
for $k$ large. But an important point of Corollary \ref{cor:conv of explicit apprl intro}
is that it provides a \emph{quantative} estimate on $\left\Vert u_{m_{k}}^{(k)}-u_{\infty}\right\Vert _{C(X)}$
when $m_{k}\sim Ak\log k,$ where $u_{\infty}$ is the optimal transport
potential, which arises as the limit of the corresponding parabolic
equation. It should also be pointed out that it can be shown that
the asymptotics in Prop \ref{prop:asympt of pho} are sharp in the
sense that error term can not be improved to $o(k^{-1}),$ in general.
Moreover, inspection of the proof of the Laplace method suggests that
$f$ and $g$ should have at least two derivatives to get the order
$O(k^{-1}).$ Hence, the rate of convergence in Theorem \ref{thm:conv in dynamic torus setting intro}
can be expected to be sharp.

\section{\label{sec:Generalizations-and-outlook}Convergence towards parabolic
optimal transport equations on compact manifolds}

Let $X$ and $Y$ be compact smooth manifolds (without boundary) and
$c$ a lsc function on $X\times Y,$ taking values in $]-\infty,\infty]$
which is smooth on the complement of a closed proper subset, denoted
by $\text{sing \ensuremath{(c)}.}$ 
\begin{rem}
Note that, in contrast to previous sections, we do not assume that
$c$ is continuous on all of $X\times Y.$ See for example Section
\ref{subsec:Application-to-the antenna} for a relevant example where
$c(x,y)=-\log|x-y|$ for $X$ and $Y$ given by the unit-sphere in
$\R^{n+1}$ and (hence $\text{sing \ensuremath{(c)}}$ is the diagonal
in $X\times Y).$ In the Riemannian case when $X=Y$ and $c=d^{2}/s$
(where $d$ is the Riemannian distance function) $c$ is, of course,
continuous on all of $X\times Y$ and $\text{sing \ensuremath{(c)} }$is
non-empty, due to the presence of cut-locus. 
\end{rem}

We will denote by $\partial_{x}$ the vector of partial derivatives
defined wrt a choice of local coordinates around a fixed point $x\in X.$
Given two normalized volume forms $\mu$ and $\nu$ in $\mathcal{P}(X)$
and $\mathcal{P}(Y)$ we locally express
\[
\mu=e^{-f}dx,\,\,\,\nu=e^{-g}dy
\]
in terms of the local volume forms $dx$ and $dy$ determined by a
choice of local coordinates. We assume that $f$ and $g$ are $C^{\infty}-$smooth.

Following standard practice we will assume that the cost function
satisfies the following assumptions 
\begin{itemize}
\item \emph{(A1) (``Twist condition'')} The map $y\mapsto\partial_{x}c(x,y)$
is injective for any $(x,y)\in X\times Y-\text{sing \ensuremath{(c)}}$
\item \emph{(A2) (``Non-degeneracy'')} det $(\partial_{x_{i}}\partial_{y_{j}}c)(x,y)\neq0$
for any $(x,y)\in X\times Y-\text{sing \ensuremath{(c)}}$
\end{itemize}
\begin{rem}
See \cite{v2} for an in depth discussion of various assumption on
cost functions. In \cite{v2} A1+A2 is called the strong twist condition
and as pointed out in \cite[Remark 12.23]{v2} it holds for the cost
function derived from any well-behaved Lagrangian, including the Riemannian
setting where $c=d^{2}/2).$ 
\end{rem}

\begin{defn}
\label{def:The-space H and F}The space $\mathcal{H}(X)$ (or $\mathcal{H}$
for short) of all \emph{smooth potentials }on $X$ is defined as the
subspace of $C^{\infty}(X)$ consisting of all $c-$convex (i.e. such
that $(u^{c})^{c}=u)$ smooth functions $u$ on $X$ with the property
that the subset $\Gamma_{u}$ of $X\times Y$ defined by formula \ref{eq:def of Gamma u}
is the graph of a diffeomorphism, denoted by $F_{u}$ and $c$ is
smooth in a neighborhood of $\Gamma_{u}.$

The definition has been made so that, if $u\in\mathcal{H}$ and $(F_{u})_{*}\mu=\nu,$
then $F_{u}$ is an optimal map (diffeomorphism) wrt the cost function
$c$ (by Prop \ref{prop:optim crit}). Accordingly, we will call $u$
the\emph{ potential} of the map $F_{u}.$ 
\end{defn}

\begin{lem}
\label{lem:twist}Assume that $c$ satisfies condition A1. If $u\in\mathcal{H},$
then 
\begin{equation}
y=F_{u}(x)\Longleftrightarrow(\partial_{x}c)(x,y)+(\partial_{x}u)(x)=0\label{eq:eq for F u}
\end{equation}
\end{lem}

\begin{proof}
By the very definition of $F_{u}$ we have that $y=F_{u}(x)$ iff
$u^{c}(y)=-c(x,y)-u(x).$ In turn, by the definition of $u^{c}$ this
equivalently means that $x$ maximizes the usc function $x'\mapsto\mathcal{C}(x'):=-c(x',y)-u(x')$
on $X.$ Now assume first that $y=F_{u}(x).$ Then, by assumption,
the function $\mathcal{C}$ is smooth close that $x$ and hence the
differential of $\mathcal{C}$ vanishes at $x,$ i.e. $(\partial_{x}c)(x,y)+(\partial_{x}u)(x)=0$
at $x.$ Conversely, if the latter equation holds then, by the twist
assumption A1, $y$ is uniquely determined by $x$ and since (as explained
above) $F_{u}(x)$ satisfies the equation in question it follows that
$y=F_{u}(x).$ 
\end{proof}
\begin{example}
Assume that $X=Y$ and that $c(x,y)\geq0$ with equality iff $x=y$
and that $c$ is smooth in a neighborhood of the diagonal. Then $u=0$
is in $\mathcal{H}$ (with $F_{0}$ given by the identity) and more
examples of potentials are obtained by using that $\mathcal{H}$ is
open in the $C^{\infty}-$topology, in general. In particular, this
applies in the ``Riemannian setting'', where $c=d^{2}/2$ on a compact
Riemannian manifold $X.$ 
\end{example}

\subsection{Parabolic optimal transport equations}

Consider now the following parabolic PDE, introduced in \cite{k-s-w}:
\begin{equation}
\frac{\partial u_{t}(x)}{\partial t}=\log\det\left(\partial_{x}F_{u_{t}}\right)-g(F_{u_{t}}(x))+f(x)\label{eq:parabolic pde general setting}
\end{equation}
expressed in terms of a choice of local coordinates, where $\det\left(\partial_{x}F_{u}\right)$
denotes the local Jacobian of the map $x\mapsto F_{u}(x).$ We note
that
\begin{itemize}
\item The right hand side in the equation \ref{eq:parabolic pde general setting}
is globally well-defined (i.e. independent of the choice of local
coordinates around $(x,F_{u}(x))$ in $X\times Y).$ Indeed, it is
equal to the logarithm of the quotient of $(F_{u})_{*}\mu$ and $\nu.$
Accordingly, $u$ is a stationary solution iff it is the potential
of an optimal transport map.
\item Differentiating the equation \ref{eq:eq for F u} reveals that (see
\cite[Section 12]{v2})
\end{itemize}
\begin{equation}
\det(\partial_{x}F_{u})=\frac{\det\left((\partial_{x}^{2}c)(x,F_{u}(x))+(\partial_{x}^{2}u)(x)\right)}{\det\left((\partial_{x}\partial_{y}c)(x,F_{u}(x))\right)}.\label{eq:form for det partial F}
\end{equation}

\subsection{Convergence of the Sinkhorn algorithm towards parabolic optimal transport}

We next introduce the following stronger local form of the ``global''
density property appearing in Lemma \ref{lem:dens prop}.
\begin{defn}
\label{def:local dens}A sequence (or family) of probability measures
$\mu^{(k)}$ on $X,$ converging weakly towards a measure $\mu,$
is said to have the \emph{local density property at order $s$ (and
length scale $k^{-1/2})$ }if there exists a positive integer $s\in[2,\infty[$
and constants $C_{1}$ and $C_{2}$ such for any fixed $x_{0}\in X$
there exist local coordinates $\xi:=(\xi_{1},...,\xi_{n})$ centered
at $x_{0}$ with the following property: for any sequence $h_{k}$
defined on the polydisc $2D_{k}$ of radius $2\log k,$ centered at
$0$ in $\R^{n}$ and satisfying $\left|\partial^{|\alpha|}h_{k}\right|(x)\leq C_{1}e^{-|x|^{2}/C_{1}}$
on $2D_{k}$ for all multiindices $\alpha$ satisfying $|\alpha|\leq s$
the bound
\[
k^{n/2}\int_{D_{k}}h_{k}(F_{x_{0}}^{(k)})_{*}(\mu^{(k)}-\mu)\leq C_{2}k^{-1},
\]
 holds, where $F_{x_{0}}^{(k)}$ is the scaled coordinate map from
a neighborhood of $x_{0}$ in $X$ into $\R^{n}$ defined by $F_{x_{0}}^{(k)}(x):=(k^{1/2}\xi(x)).$
If this is property holds for some $s\in[2,\infty[$ we will simply
say that $\mu^{(k)}$ has the local density property.
\end{defn}

We are now ready for the following generalization of Theorem \ref{thm:conv in dynamic torus setting intro}
stated in the introduction (as in the torus setting in Section \ref{subsec:Discrete-optimal-transport}
the entropic regularization parameter is expressed as $\epsilon=k^{-1},$
where $k$ is a positive number).
\begin{thm}
\label{thm:dynamic general}Let $c$ be a function satisfying the
assumptions A1 and A2 and $\mu^{(k)}$ and $\nu^{(k)}$ be two sequences
converging towards $\mu$ and $\nu$ in $\mathcal{P}(X)$ and $\mathcal{P}(Y),$
respectively, satisfying the local density property. Given $u_{0}\in\mathcal{H}$
assume that there exists a solution $u_{t}$ in $C^{2}(X\times[0,T])$
of the parabolic PDE \ref{eq:parabolic pde general setting} with
initial condition $u_{0}$ and such that $u_{t}\in\mathcal{H}$ for
any $t\in[0,T].$ Denote by $u_{m}^{(k)}$ the iteration \ref{eq:scaled iteration}
defined by the data $(\mu^{(k)},\nu^{(k)},c)$ and such that $u_{0}^{(k)}=u_{0}$
for any given $k.$ Then there exists a constant $C$ such that for
any $(m,k)$ satisfying $m/k\in[0,T]$ 
\[
\sup_{T^{n}}\left|u_{m}^{(k)}-u_{m/k}\right|\leq C\frac{m}{k}k^{-1},
\]
\end{thm}

\begin{proof}
The assumptions have been made precisely to ensure that the proof
of Theorem \ref{thm:conv in dynamic torus setting intro} can be generalized,
almost verbatim. Hence, we will be rather brief.

\emph{Step 1: }Let $\alpha$ be a lower-semicontinuous (lsc) function
on $X$ with a unique minimum at $x_{0}$ and assume that $\alpha$
is $C^{4}-$smooth on a neighborhood of $x_{0}$ and $\nabla^{2}\alpha(x_{0})>0$
(in local coordinates centered at $x_{0}).$\emph{ }Then, in local
coordinates centered at $x_{0},$
\begin{equation}
k^{n/2}\int_{X}e^{-k\alpha}h\mu^{(k)}=(2\pi)^{n/2}e^{-k\alpha(x_{0})}\frac{h(x_{0})e^{-f(x_{0})}}{\sqrt{\det(\partial_{x}^{2}\alpha(x_{0}))}}(1+O(k^{-1})).\label{eq:diskrete Laplace general}
\end{equation}
Using the local density assumption as a replacement of Lemma \ref{lem:local dens prop in model case},
this is shown essentially as before. The point is that the derivatives
$\partial^{\alpha}$ of any fixed order of $h_{k}(x):=\left((F_{x_{0}}^{(k)})^{-1}\right)^{*}\left(fe^{-k\alpha}\right)$
are bounded by $Ce^{-|x|^{2}/C}$ when $|x_{i}|\leq k^{1/2}$ (and
in particular, for $|x_{i}|\leq2\log k).$ This follows readily from
the chain rule, as before. In fact, the proof of formula \ref{eq:diskrete Laplace general}
is even a bit simpler than in the torus case, since the last step
3 is not needed when the local density assumption is made at any point
$x_{0}.$ 

\emph{Step 2:} If $u\in\mathcal{H},$ then the following asymptotics
holds 
\[
\rho_{ku}(x)=\det(\partial_{x}F_{u}(x))e^{f(x)-g(F_{u}(x))}(1+O(k^{-1}))
\]

To prove this first observe that if $u\in\mathcal{H}(X),$ then $u^{c}\in\mathcal{H}(Y).$
Indeed, by assumption there is a unique $x(=x_{y})$ such that $y=F_{u}(x).$
Moreover, by the very definition of $F_{u}$ (Definition\ref{def:The-space H and F})
we can express 
\[
u^{c}(y)=-c(x_{y},F_{u}(x_{y}))-u(x_{y}).
\]
 Since $c$ is assumed to be smooth in a neighborhood of $\Gamma_{u}$
the right hand side above defines a smooth function in $x$ and since
$F$ is a diffeomorphism it follows that $u^{c}(y)$ is smooth. Moreover,
by symmetry $\Gamma_{u^{c}}=\Gamma_{u},$ which can be identified
with the graph of the diffeomorphism $F_{u}^{-1}.$ This shows that
$u^{c}\in\mathcal{H}(Y)$ and 

\begin{equation}
F_{u^{c}}=(F_{u})^{-1}\label{eq:F u inverse}
\end{equation}
Setting $y_{x}:=F_{u}(x)$ and $x_{y}=F_{u^{*}}(y)$ we can now apply
the previous step, essentially as before, to get 
\[
\rho_{ku}(x)=\sqrt{\frac{\det\left((\partial_{x}^{2}c)(x_{y},x)+\partial_{x}^{2}u(x_{y})\right)}{\det\left((\partial_{y}^{2}c)(x,y_{x})+\partial_{y}^{2}u^{c}(y_{x})\right)}}e^{f(x)-g(F_{u}(x))}(1+O(k^{-1}))
\]
Finally, differentiating the relation \ref{eq:F u inverse} reveals
that 
\[
\det((\partial_{y}F_{u^{c}})(y_{x}))=\det((F_{u}(x))^{-1}
\]
and hence using equation \ref{eq:form for det partial F} and symmetry
(which ensures that the denominator appearing in equation \ref{eq:form for det partial F}
coincides with the one appearing obtained when $u$ is replaced by
$u^{c}$) concludes the proof of Step 2.

\emph{Step 3: }Conclusion of proof

The proof is concluded, as before, by invoking Lemma \ref{lem:general conv towards parab}. 
\end{proof}
As pointed out in \cite{k-s-w} it follows from standard short-time
existence results for parabolic PDEs that the existence of a solution
$u_{t}$ as in the previous theorem holds for some $T>0$ (see, for
example, \cite[Main Thm 1]{ba}). Moreover, by \cite{k-s-w} long-time
existence, i.e. $T=\infty,$ holds under the following further assumptions
on $c:$ 
\begin{itemize}
\item \emph{(A3) (``Stay-away property'') }For any $0<\lambda_{1},\lambda_{2}$
there exists $\epsilon>0$ only depending on $\lambda_{1},\lambda_{2}$
such that $\lambda_{1}\leq\text{|\ensuremath{\det}}\partial_{x}F_{u}|\leq\lambda_{2}\implies\text{dist \ensuremath{(\Gamma_{u},\text{\ensuremath{\text{sing \ensuremath{(c))\geq\epsilon}}}}}}$
for any $u\in\mathcal{H}$
\item (A4) (``Semi-concavity'') $c$ is locally semi-concave, i.e. the
sum of a concave and a smooth function on the domain where it is finite.
\item (A5) (``Strong MAW-condition'') The Ma-Wang-Trudinger tensor of
$c$ is bounded from below on $X\times Y-\text{sing \ensuremath{(c)}}$
by a uniform positive constant $\delta.$
\end{itemize}
\begin{thm}
\label{thm:(Kim-Streets-Warren-)-Assume}(Kim-Streets-Warren \cite{k-s-w})
Assume that $c$ satisfies the assumptions A1-A5. Then, for any given
$u_{0}\in\mathcal{H}$ there exists a solution $u(x,t)$ in $C^{\infty}(X\times[0,\infty[)$
of the parabolic PDE \ref{eq:parabolic pde general setting} with
initial condition $u_{0}$ and such that $u_{t}\in\mathcal{H}$ for
any $t>0.$ Moreover, $u_{t}$ converges exponentially in $C^{0}(X)),$
as $t\rightarrow\infty,$ to a potential $u\in\mathcal{H}$ of a diffeomorphism
$F_{u}$ transporting $\text{\ensuremath{\mu\ }to \ensuremath{\nu} , }$which
is optimal wrt the cost function $c.$
\end{thm}

\begin{proof}
Let us explain how to translate the result in \cite{k-s-w} to the
present setting. Following \cite{k-s-w} a function $u\in C^{2}(X)$
is said to be \emph{locally strictly $c-$convex, }if, in local coordinates,
the matrix $(\partial_{x}^{2}c)(x,F_{u}(x))+(\partial_{x}^{2}u)(x)$
is positive definite. This condition is independent of the choice
of local coordinates. Indeed, it equivalently means that any given
$x_{0}\in X$ is a non-degenerate local minimum for the function 
\begin{equation}
x\mapsto c(x,F(x_{0}))+u(x)\,\,\,\text{on\,}X.\label{eq:function that x not min}
\end{equation}
It follows that for any such $u$ the corresponding map $F_{u}$ is
a local diffeomorphism. The main result in \cite{k-s-w} says that,
under the assumptions on $c$ in the statement above, for any initial
datum $u_{0}\in C^{\infty}(X)$ which is\emph{ }locally strictly $c-$convex,\emph{
}there exists a solution $u(x,t)$ in $C^{\infty}(X\times]0,T])$
which is also locally strictly $c-$convex. To make the connection
to the present setting first note that if $u\in\mathcal{H}$ then
$u_{0}$ is even an absolute minimum for the function \ref{eq:function that x not min},
which is non-degenerate (since $F_{u}$ is a diffeomorphism) and hence
$u$ is locally strictly $c-$convex. Conversely, if $u$ is locally
strictly $c-$convex then \cite[Cor 7.1]{k-s-w} says that $u\in C^{2}(X)$
is $c-$convex (i.e. $(u^{c})^{c}=u)$ and the proof given in \cite[Cor 7.1]{k-s-w}
moreover shows that $F_{u}$ is a global $C^{1}-$diffeomorphism.
It then follows from the Assumption A3 that $c$ is smooth in a neighborhood
of $\Gamma_{u}.$ Hence, $u\in C^{\infty}(X)$ is locally strictly
$c-$convex iff $u\in\mathcal{H},$ which concludes the proof of the
theorem. 
\end{proof}
\begin{rem}
Under the assumptions in the previous theorem it follows, in particular,
that the optimal transport map is smooth. Conversely, the assumptions
are ``almost necessary'' for regularity of the optimal transport
map (see \cite[Chapter 12]{v2} and reference therein). Also note
that the semi-concavity assumption is always satisfied in the case
when $X=Y$ is a compact Riemannian manifold and $c=d^{2}/2$ \cite[(b), Page 278]{v2}. 
\end{rem}

Combining the exponential large-time convergence of $u_{t},$ in the
previous theorem, with Theorem \ref{thm:dynamic general} gives, just
as in the torus setting, the following 
\begin{cor}
\label{cor:constr of approx in quite general setting }Assume that
the cost function $c$ satisfies the assumptions in Theorem \ref{thm:dynamic general}
(or that $c=d_{T^{n}}^{2}/2$ in the case of the torus $T^{n}$).
Assume, moreover, that $\mu^{(k)}$ and $\nu^{(k)}$ satisfy the local
density property. Then, for any given $u_{0}\in\mathcal{H},$ there
exists a positive constant$A_{0}$ such that for any $A>A_{0}$ the
following holds: $u_{k}(x):=u_{m_{k}}^{(k)}(x_{i_{k}}),$ with $m_{k}:=\left\lfloor Ak\log k\right\rfloor ,$
converges uniformly to the optimal transport potential $u(x)$. More
precisely, there exists a constant $C$ (depending on $A)$ such that
\[
\sup_{T^{n}}\left|u_{k}-u\right|\leq Ck^{-1}\log k
\]
\end{cor}

In turn, this corollary implies, just as before, that the analog of
Corollary \ref{cor:cons transport cost} holds (under the same assumptions
as in the previous corollary).
\begin{example}
\label{exa:The-assumptions-in}The assumptions on $c$ in the previous
corollary are satisfied when $X=Y$ is the $n-$sphere and $c(x,y)=d^{2}(x,y)/2$
for the standard round metric or $c(x,y)=-\log|x-y|,$ where $|x-y|$
denotes the chordal distance (see \cite{k-s-w} and references therein).
The latter case appears in the reflector antenna problem, as explained
in Section \ref{subsec:Application-to-the antenna}. 
\end{example}

\subsection{Constructing discretizations using Quasi-Monte Carlo systems}

In order to construct discretizations satisfying the local density
property (Definition \ref{def:local dens}) on general compact manifolds
we will employ Quasi-Monte Carlo systems, familiar from the theory
of numerical integration. These can be viewed as generalizations of
the standard grids with $N$ points on the torus.

Let $(X,g)$ be an $n-$dimensional compact Riemannian manifold and
denote by $dV$ the corresponding normalized volume form on $X.$
Following \cite{b-e-g} (and \cite{b-s-s-w} in the case of a sphere)
the \emph{worst case error }of integration of points $\{x_{i}\}_{i=1}^{N}\Subset X$
and weights $\{w_{i}\}_{i=1}^{N}\Subset\R_{+}$ (assumed to sum to
one) with respect to some Banach space $W$ of continuous functions
on $X,$ is defined as 
\[
\text{wce \ensuremath{(\left\{ (x_{i},w_{i})\right\} _{i=1}^{N})}:=\ensuremath{\sup}}\left\{ \left|\int fdV-\sum_{i=1}^{N}f(x_{i})w_{i}\right|:\,\,f\in W,\,\,\left\Vert f\right\Vert \leq1\right\} 
\]
We will use the shorthand $X_{N}$ for the \emph{weighted point cloud}
$\left\{ (x_{i},w_{i})\right\} _{i=1}^{N}$ and and call the corresponding
discrete probability measure 
\begin{equation}
\lambda_{X_{N}}:=\sum_{i=1}^{N}w_{i}\delta_{x_{i}}\label{eq:def of sampling measure}
\end{equation}
for the \emph{sampling measure} associated to $X_{N}.$ Let now $W:=W_{p}^{s}(X)$
be the Sobolev space of all functions $f$ on $X$ such that all (fractional)
distributional derivatives of order $s$ are in $L^{p}(X)$ and assume
that $s>n/p$ and $p\in[1,\infty]$ (which ensures that $W_{p}^{s}(X)\subset C^{0}(X)).$
Then a sequence of weighted point clouds\emph{ }$X_{N}:=\left\{ (x_{i},w_{i})\right\} _{i=1}^{N}$
is said to be a\emph{ quasi-Monte Carlo system }for $W_{p}^{s}(X)$
if 
\[
\text{wce \ensuremath{(X_{N})\leq\frac{C}{\left(N^{1/n}\right)^{s}}} }
\]
for some uniform constant $C$ (the corresponding lower bound holds
for any sequence $X_{N}).$ In view of the applications to the present
setting we introduce the following definition (modeled on the corresponding
definition for $p=2$ in \cite{b-s-s-w}):
\begin{defn}
Given $p\in[1,\infty],$ a sequence of weighted point clouds $X_{N}:=\left\{ (x_{i},w_{i})\right\} _{i=1}^{N}$
is said to be a\emph{ generic Quasi-Monte Carlo $p-$system }if $X_{N}$
is a quasi-Monte Carlo system for $W_{p}^{s}(X)$ for any positive
integer $s>n/p.$ Moreover, $X_{N}$ is a \emph{completely generic
QMC system} if it is a\emph{ }generic Quasi-Monte Carlo $p-$system
for any $p\in[1,\infty].$ 
\end{defn}

\begin{example}
\label{exa:grid is qmc}When $M=T^{n}$ the standard grid with $N$
points on $T^{n}$ (i.e. with ``edge length'' $1/N^{1/n}))$ and
with all weights taken to be equal to $1/N$ defines a completely
generic Quasi-Monte Carlo system $X_{N}$ \cite{b-c-c-g-s-t}. Moreover,
in the case of the standard $n-$dimensional sphere $S^{n}$ it follows
from \cite{b-r-v,b-s-s-w} that there exists a sequence of completely\emph{
}generic quasi-Monte Carlo system with all weights equal to $1/N.$
The corresponding point sets are taken as \emph{spherical $t-$designs}
with $t\sim N^{1/n}$ (these are defined as quadrature points, discussed
in Section \ref{subsec:Constructing-completely-generic} below, with
all weights equal). Such points have been generated for large values
of $N$ \cite{wo}. But allowing different weights has the advantage
that \emph{explicit} completely generic QMC systems can be constructed,
as discussed below.
\end{example}

\begin{rem}
\label{rem:cover rad}Recall that the covering radius $\rho_{N}$
of a point cloud $\{x_{1},...,x_{N}\}$ on $(X,g)$ is defined by
$\sup_{x\in X}\min_{i\leq N}d(x,x_{i}).$ By \cite[Prop 4.1, 4.3]{b-e-g},
if $X_{N}$ is a weighted generic QMC $1-$system then the corresponding
covering radius $\rho_{N}$ is comparable to $N^{-1/n}.$
\end{rem}

In order to apply this setup to the present setting of the Sinkhorn
iteration we will need to adapt the number $N$ of points to the value
of the parameter $k$ (defined as the invers of the entropic regularization
parameter). This is made precise by the following lemma, quantifying
how large $N(=N_{k})$ needs to be:
\begin{lem}
Let $s$ be a positive integer. Assume that $X_{N_{k}}:=\left\{ (x_{i}^{(k)},w_{i}^{(k)})\right\} _{i=1}^{N_{k}}$
is a sequence of quasi-Monte Carlo systems for $W_{p}^{s}(X)$ (where
$s>n/p)$ indexed by a parameter $k.$ Then the corresponding \emph{sampling
measures} $\lambda_{k}(:=\lambda_{N_{k}})$ have the local density
property at order $s$ at length scales $k^{-1/2}$ (Definition \ref{def:local dens})
if the ``number condition'' 
\begin{equation}
N_{k}^{-1/n}\leq\frac{C}{k^{1/2}}\frac{1}{k^{\left(1+\frac{n}{2}(1-1/p)\right)/s}}\label{eq:number cond}
\end{equation}
holds for some positive constant $C.$ 
\end{lem}

\begin{proof}
Given a sequence $h_{k}$ as in Definition \ref{def:local dens} set
$f_{k}:=(F_{x_{0}}^{(k)})^{*}(\chi_{k}h_{k}),$ where $\chi_{k}(x):=\chi(x/\log k)$
for a given smooth function $\chi$ in $\R^{n}$ equal to one on the
polydisc $D_{1}$ centered at $0$ and with support in $2D_{1}.$
By the assumed quasi-Monte Carlo property there exists a constant
$C$ such that 
\[
\left|\int_{X}f_{k}(\lambda_{k}-dV)\right|\leq C\frac{1}{N_{k}^{s/n}}\sum_{|\alpha|\leq s}\left(\int|\partial^{\alpha}f_{k}|^{p}dV\right)^{1/p}.
\]
Multiplying both sides with $k^{n/2}$ and using the chain rule thus
gives 
\[
k^{n/2}\left|\int_{D_{k}}h_{k}(F_{x_{0}}^{(k)})_{*}(\lambda_{k}-dV)\right|\leq k^{n/2}\left|\int_{2D_{k}}\chi_{k}h_{k}(F_{x_{0}}^{(k)})_{*}(\lambda_{k}-dV)\right|=
\]

\[
\frac{Ck^{n/2}}{N_{k}^{s/n}}\left(\sum_{|\alpha|\leq s}k^{|\alpha|/2}\int_{D_{2k}}|\partial^{\alpha}(\chi_{k}h_{k})|^{p}(F_{x_{0}}^{(k)})_{*}dV\right)^{1/p}\leq C'\frac{k^{n/2}k^{-n/2p}}{N_{k}^{s/n}}k^{s/2}\sum_{|\alpha|\leq s}\left(\int_{D_{2k}}|\partial^{\alpha}(\chi_{k}h_{k})|^{p}dx\right)^{1/p}.
\]
Now, by assumption, all $|\partial^{\alpha}h_{k}|(x)\leq Ce^{-|x|^{2}/C}$
on $D_{2k}.$ Moreover, $|\partial^{\alpha}\chi_{k}|$ is uniformly
bounded (since $\partial^{\alpha}\chi_{k}=\partial\chi/(\log k)^{|\alpha|}$)
and hence the sum in the right hand side above is uniformly bounded.
Since the number condition ensures that the factor in front of the
sum is bounded by a constant times $k^{-1}$ this concludes the proof. 
\end{proof}
\begin{rem}
When $X=T^{1}$ and $s=2$ and $p=1$ (which satisfies $s>n/p)$ the
previous lemma is closely related to the case $n=1$ of Lemma \ref{lem:local dens prop in model case}
(recall that the case $X=T^{n}$ can be reduced to the case when $n=1).$ 
\end{rem}

\subsubsection{\label{subsec:Constructing-completely-generic}Constructing completely
generic QMC systems}

Finally, we recall how to construct completely generic QMC systems
on any $n-$dimensional compact Riemannian manifold $(X,g),$ generalizing
the standard grid on the torus, following \cite{b-c-c-g-s-t}. Given
a positive number $W$ (the ``bandwidth'') we denote by $H_{\leq W}(X,g)$
the finite dimensional subspace of $C^{\infty}(X)$ consisting of
all eigenfunctions of the Laplacian with eigenvalue bounded from above
by $W^{2}$. A weighted point cloud $X_{N}$ is said to consist of
\emph{weighted quadrature points} for $H_{\leq W}(X,g)$) if the corresponding
numerical integration error vanishes for any $f\in H_{\leq W}(X).$
For any sufficently large constant $C_{X}$ there exists a sequence
of weighted quadrature points $X_{N}$ for $H_{\leq C_{X}N^{1/n}}(X)$
(as follows from \cite[Cor 2.11]{b-c-c-g-s-t} and the Weyl asymptotics,
saying that the dimension of $H_{\leq W}(X)$ grows like a constant
times $W^{n}).$ Moreover, by \cite[Cor 2.13]{b-c-c-g-s-t} such a
sequence $X_{N}$ defines a completely generic QMC system. 
\begin{example}
\label{exa:the two sphere}When $(X,g)$ is the round two-sphere $S^{2}$
it is customary to rewrite $W^{2}=l(l+1).$ Letting $l$ range over
all non-negative integers then enumerates all eigenvalues of the Laplacian
and the corresponding space $H_{\leq W}(X)$ then coincides with the
space of all spherical polynomials of degree $l$ (see Section \ref{subsec:Optimal-transport-on two sph}).
In this case the most commonly used explicit weighted quadrature points
are obtained using various longitude-latitude rules; see \cite[Section 4.1]{h-s-w}
and also \cite[Thm 3]{d-h} (applied to $(l,m)=(0,0)$) for the ``equi-angular''
case. 
\end{example}

\subsection{\label{subsec:Exploiting-higher-regularity}Reducing the numbers
$N_{k}$ of discretization points by exploiting higher regularity
of the data}

Coming back to the setup in the beginning of Section \ref{sec:Generalizations-and-outlook}
we fix two sequences $X_{N}$ and $Y_{N}$ of $N$ weighted points
on $X$ and $Y,$ respectively. We will assume that $X_{N}$ and $Y_{N}$
are completely generic QMC systems defined with respect to some Riemannian
volume forms $dV_{X}$ and $dV_{Y}$ on $X$ a $Y,$ respectively.
Fix $\delta\in]0,1/2]$ and index the number of points $N$ by the
parameter $k$ so that 
\begin{equation}
N_{k}^{1/n}\geq C_{\delta}k^{1/2+\delta},\,\,\,\delta>0,\,\,C_{\delta}>0\label{eq:lower bound on number in section higher}
\end{equation}
(by Remark \ref{rem:cover rad} this equivalently means that the covering
radius of the corresponding point clouds is of the order $O(1/k^{1/2+\delta}$).
We then define a family of ``discretizations'' $\mu^{(k)}$ and
$\nu^{(k)}$ of $\mu$ and $\nu,$ by proceeding as in the torus case,
but replacing the sampling measure $\delta_{\Lambda_{k}}$ with the
sampling measures corresponding to $X_{N}$ and $Y_{N}.$ In other
words, setting $\rho_{X}:=\mu/dV_{X}$ the discrete probability measure
$\mu^{(k)}$ is defined as 
\[
\mu^{(k)}:=\rho_{X}\lambda_{X_{N_{k}}}/Z_{N_{k}},\,\,\,
\]
where $Z_{N_{k}}$ is the normalizing constant (and $\nu$ is defined
in a similar fashion). 
\begin{lem}
The sequences $\mu^{(k)}$ and $\nu^{(k)}$ have the local density
property.
\end{lem}

\begin{proof}
First observe that if follows directly from the assumptions on $X_{N}$
and $Y_{N}$ that the corresponding norming constants are equal to
$1+O(k^{-1})$ (in fact, the error terms are of the order $O(k^{-\infty})$
since the densities are assumed to be smooth). Hence, it will be enough
to show that local density property holds for the sequences $\rho_{X}\lambda_{X_{N_{k}}}$
and $\rho_{Y}\lambda_{Y_{N_{k}}}.$ But then the result is reduced
to Lemma \ref{lem:dens prop} by simply replacing $h_{k}$ with $\tilde{h}_{k}:=h_{k}(F_{x_{0}}^{(k)^{-1}})^{*}\rho,$
where $\rho$ is the density of $\mu$ or $\nu.$ Indeed, by Leibniz
rule and the chain rule $\tilde{h}_{k}$ also satisfies the estimates
in the lemma. Thus taking $s=1/\delta$ and $p=1$ and invoking Lemma
\ref{lem:dens prop} concludes the proof.
\end{proof}
This means that Theorem \ref{thm:dynamic general} and Corollary \ref{cor:constr of approx in quite general setting }
apply to the discretization scheme above, as long as the number $N_{k}$
of points satisfies the bound \ref{eq:lower bound on number in section higher}
for some $\delta>0.$ 
\begin{rem}
The previous argument shows that if the densities of $\mu$ and $\nu$
are in $C^{s}(X)$ for some integer $s\in[2,\infty[$ and $s>n/p$
for $p\in[1,\infty[$ then the local density condition (at order $s)$
holds if $N_{k}$ satisfies the condition \ref{eq:number cond}. In
particular, if $s>n$ on can take $p=1$ and then the condition is
that $N_{k}^{1/n}\geq k^{1/2+1/s}.$ In the particular case of a uniform
grid on the torus the condition is thus that the edge-length of the
grid is bounded from below by $Ck^{-(1/2+1/s)}$ (the assumption $s>n$
is not needed in this case, since one can reduce to the case when
$n=1).$ 
\end{rem}

\section{\label{sec:Nearly-linear-complexity}Nearly linear complexity on
the torus and the sphere}

In this section we start by showing that the convergence results in
Section \ref{sec:Generalizations-and-outlook} hold in a more general
setting where the kernel $\mathcal{K}^{(k)}(x,y):=e^{-kc(x,y)}$ is
replaced with an appropriate approximate kernel. This extra flexibility
is then applied in the setting of optimal transport on the two-sphere,
using ``band-limited'' heat-kernels, where it leads to a nearly
linear algorithmic cost for the corresponding Sinkhorn iterations.

\subsection{Sequences $c_{k}$ and approximate kernels $K_{k}$}

Just as in the generalization of the (static) Theorem \ref{Thm:weak conv of fixed points towards optimal transport plans},
considered in Section \ref{thm:conv wasserst}, the (dynamic) Theorem
\ref{thm:dynamic general} can be generalized by replacing the cost
function $c$ with a suitable sequence $c_{k}.$ But then the uniform
convergence of $c_{k}$ towards $c$ (formula \ref{eq:unif conv of ck})
has to be supplemented with further asymptotic properties on the complement
of the singularity locus of $c.$ For example, the proof of Theorem
\ref{thm:dynamic general} goes through, almost word for word, if
the upper bound corresponding to \ref{eq:unif conv of ck} holds globally,
i.e.: 
\begin{equation}
e^{-kc_{k}(x,y)}\leq O(e^{\epsilon k})e^{-kc(x,y)}\label{eq:upper bound on minus c k}
\end{equation}
(where $O(e^{\epsilon k})$ denotes a sequence of sub-exponential
growth) and $c_{k}$ has the following further property: on any given
compact subset in the complement of $\text{sing \ensuremath{(c)}}$
there exists a strictly positive smooth function $h_{0}(x,y)$ and
a uniformly bounded sequence $r_{k}(x,y)$ of functions such that
\begin{equation}
\mathcal{K}^{(k)}(x,y):=e^{-kc_{k}(x,y)}=e^{-kc(x,y)}(h_{0}(x,y)+k^{-1}r_{k}(x,y))\label{eq:expansion c k}
\end{equation}
This implies, in particular, that if Theorem \ref{Thm:weak conv of fixed points towards optimal transport plans}
holds for a given kernel $\mathcal{K}^{(k)},$ then it also holds
for any other kernel $\tilde{\mathcal{K}}^{(k)}$ which has error
$O(k^{-1})$ as an approximation relative to $\mathcal{K}^{(k)},$
i.e. such that 
\begin{equation}
|\mathcal{K}^{(k)}-\tilde{\mathcal{K}}^{(k)}|\leq Ck^{-1}\mathcal{K}^{(k)}\label{eq:relative error of kernels}
\end{equation}
or such that $\tilde{\mathcal{K}}^{(k)}$ has absolute error $e^{-Ck},$
for $C$ sufficiently large, i.e. 
\begin{equation}
|\mathcal{K}^{(k)}-\tilde{\mathcal{K}}^{(k)}|\leq e^{-Ck},\,\,\,C>\inf_{X\times Y}c\label{eq:abs relative error}
\end{equation}

\subsection{\label{subsec:Heat-kernel-regularization}Heat kernel approximations
in the Riemannian setting}

Consider now the Riemannian setting where $X$ is a compact Riemannian
manifold (without boundary) and $c=d^{2}/2$ and $c_{k}$ is defined
in terms of heat kernel (formula \ref{eq:c k as heat}): 
\begin{thm}
\label{thm:heat kernel para}Let $X$ be a compact Riemannian manifold
(without boundary) and set $c(x,y):=d(x,y)^{2}/2,$ where $d$ is
the Riemannian distance function. Then the results in Theorem \ref{thm:dynamic general}
and Corollary \ref{cor:constr of approx in quite general setting }
still hold when the matrix kernel $e^{-kd^{2}(x,y)/2}$ is replaced
with the heat kernel $\mathcal{K}_{2k^{-1}}(x,x)$ (at time $t=2k^{-1})$
\end{thm}

\begin{proof}
As discussed above this follows from the following heat kernel asymptotics
(which are a special case of \cite[Thm 3.1]{ben} and more generally
hold for the heat kernel associated to a suitable hypoelliptic operator).
Assume that $x$ and $y$ are contained in a compact subset of the
complement of the cut-locus. Then 
\[
\mathcal{K}_{t}(x,y)=t^{-n/2}e^{-t^{-1}d^{2}(x,y)/4}\left(h_{0}(x,y)+tr_{1}(t,x,y)\right),
\]
where $h_{0}$ is smooth and $h_{0}>0$ and $r_{1}$ is smooth and
uniformly bounded on $]0,t_{0}]\times X.$ This is not exactly of
the form \ref{eq:expansion c k} due to the presence of the factor
$t^{-n/2}:=A_{k}.$ But it is, in fact, enough to know that \ref{eq:expansion c k}
holds when the right hand side is multiplied with a sequence $A_{k},$
only depending on $k.$ Indeed, the iteration $u_{m}^{(k)}$ is unaltered
when the cost function $c_{k}(x,y)$ is replaced by $c_{k}(x,y)+C_{k}$
for some constant $C_{k}$ (which is consistent, as it must, with
the fact that the parabolic equation \ref{eq:parabolic pde general setting}
is unaltered when a constant is added to $c).$
\end{proof}
The use of the heat kernel in the Sinkhorn algorithm for optimal transport
on Riemannian manifolds was advocated in \cite{s-d---}, where it
was found numerically that discretized heat kernels provide substantial
speedups, when compared to other methods. The previous theorem offers
a theoretical basis for the experimental findings in \cite{s-d---},
as long as the discretized heat kernels $\tilde{\mathcal{K}}^{(k)}$
satisfy one of the the approximation properties \ref{eq:relative error of kernels}
and \ref{eq:abs relative error} (when compared with the corresponding
bona fide heat kernel). However, the author is not aware of any general
such approximation results in the discretized setting (but see \cite{c-w-w}
and references therein for various numerical approaches to discretizations
of heat kernels). We will instead follow a different route, based
on ``band-limited'' heat kernels and fast Fourier type transforms,
applied to the case when $X$ is the two-sphere.

\subsection{Near linear complexity using fast transforms}

Each iteration in the Sinkhorn algorithm amounts to computing two
vector-matrix products of the form
\begin{equation}
a_{i}=\sum_{j=1}^{N}\mathcal{K}(x_{i},y_{j})b_{j},\,\,\,\,i=1,...,N,\label{eq:sum matrix vector}
\end{equation}
 for a given function $\mathcal{K}$ on $X\times Y$ (followed by
$N$ inversions), where $b$ and $a$ denote generic ``input vector''
and ``output vectors'', respectively. In general, this requires
$O(N^{2})$ arithmetic operations. But, as we will next exploit, in
the presence of suitably symmetry fast summations techniques can be
used to lower the complexity to nearly linear, i.e. to at most $CN(\log N)^{p}$
operations (for some positive constants $C$ and $p).$ Alternatively,
separability properties of the kernels in question can often be exploited
to directly decrease the complexity (as in \cite[Remark 4.17]{m-p}).
For example, consider the case when $X=X_{1}\times X_{2}$ and $Y=Y_{1}\times Y_{2}$
and denote by $\pi_{1}$ and $\pi_{2}$ the projections on the first
and second factors, respectively. If $\mathcal{K}(x,y)$ is \emph{separable},
i.e. factors as 
\[
\mathcal{K}(x,y)=\mathcal{K}_{1}\left(\pi_{1}(x),\pi_{1}(y)\right)\mathcal{K}_{1}\left(\pi_{2}(x),\pi_{2}(y)\right)
\]
one can take the point clouds $(x_{i})_{i\leq1}$ and $(y_{i})_{i\leq N}$
as the ``grid''s induced by given point clouds on each factor with
(consisting of $N_{1}$ and $N_{2}$ points, respectively). Then by
first summing over the first factor the complexity of the computation
is reduced to $O\left(N(N^{1}+N^{2})\right)$ operations. More generally,
if the separability holds with $r$ factors, then, by induction, the
complexity becomes $O\left(N(\max_{i\leq r}N_{i})\right).$ 

\subsubsection{\label{subsec:Optimal-transport-on linear}Optimal transport on the
flat torus }

Let us first come back to the case of the flat torus $T^{n}$ discretized
by the discrete torus $\Lambda_{k},$ considered in Section \ref{subsec:Main-results-in-the torus}.
Since $\mathcal{K}(x,y):=e^{-kd^{2}(x,y)}$ is invariant under the
diagonal action of the torus $T^{n}$ it is follows from standard
arguments that the sums \ref{eq:sum matrix vector} can be computed
in $O(N)(\log N)$ arithmetic operations. Indeed, using the group
structure on $T^{n}$ we can write $\mathcal{K}(x,y)=h(x-y),$ for
some function $h$ on $\Lambda_{k}.$ Then the classical convolution
theorem of Fourier Analysis, on the discrete torus $\Lambda_{k}$
(viewed as an abelian finite group), gives (with $m_{1},...m_{N}$
the points of the dual discrete torus, identified with $\Lambda_{k}):$
\[
a_{i}=\sum_{j=1}^{N}h(x_{i}-y_{j})b_{j}=\sum_{j=1}^{N}\hat{h}(m_{j})\hat{b}(m_{j})e^{2\pi im_{j}\cdot x_{i}},\,\,\,\,\hat{f}(m_{j}):=\sum_{i=1}^{N}f_{i}e^{-2\pi ix_{i}\cdot m_{j}}
\]
This requires evaluating two Discrete Fourier Transforms (DFT) at
the $N=k^{n}$ points $m_{1},...m_{N},$ Using the Fast Fourier Transform
(FFT) this can be done in $O(N)(\log N)$ arithmetic operations. Note
that, since the heat kernel is also torus invariant, the same argument
can also be used for the kernel appearing in Theorem \ref{thm:heat kernel para},
in the torus case. Alternatively, using that $\mathcal{K}(x,y)$ is
separable on $T^{n}$ (since, in general, the squared Riemannian distance
function on a Riemannian product is the sum of the squared distances
of the factors), the summing can directly be achieved with complexity
$O(N^{1+1/n})$ for any fixed point cloud on $S^{1}.$ 

\subsubsection{\label{subsec:Optimal-transport-on two sph}Optimal transport on
the round two-sphere}

Consider the round two-sphere $S^{2}$ embedded as the unit-sphere
in $\R^{3}.$ Removing the north and south pole on $S^{2}$ we have
the standard spherical (longitude-colatitude) coordinates $(\varphi,\theta)\in[0,2\pi[\times]-\pi,\pi[.$
A complete set of (non-normalized) eigenfunctions for the Laplacian
on $L^{2}(S^{2})$ is given by the spherical harmonics

\[
Y_{l}^{m}(\varphi,\theta):=e^{im\varphi}P_{l}^{m}(\cos\theta),\,\,\,\,|m|\leq l,
\]
which has eigenvalue $\lambda_{l,m}^{2}:=l(l+1).$ Here $P_{l}^{m}(z)$
denotes, för $z\in[-1,1],$ as usual, the Legendre function of degree
$l$ and order $m$ (aka the associated Legendre polynomial); see,
for example, \cite{d-h,h-s-w}.

Given a positive number $W$ (the ``band-width'') we consider the\emph{
band-limited heat kernel }on the two-sphere:
\begin{equation}
\mathcal{K}_{t}(x,y)_{W}:=\sum_{|m|\leq l\leq W}c_{m,l}Y_{l}^{m}(x)\overline{Y_{l}^{m}(y)},\,\,\,c_{m,l}:=e^{-tl(l+1)}\left\Vert Y_{l}^{m}\right\Vert _{L^{2}}^{-2}\label{eq:band limited heat kernel on sphere}
\end{equation}
(By the spectral theorem this means that $\mathcal{K}_{t}(x,y)_{W}$
is the integral kernel of $e^{-t\Delta}\Pi_{W}$ where $\Pi_{W}$
is the orthogonal projection onto the space of all band-limited functions).
\begin{thm}
\label{thm:sphere band limit heat}Consider the two-sphere, discretized
by a given good Quasi-Monte Carlo (QMC) system and take $R$ such
that $R>1.$ Then the analog of all the results in Section \ref{subsec:Main-results-in-the torus}
are valid when the matrix kernel $e^{-kd^{2}(x,y)/2}$ is replaced
by the band-limited heat kernel $\mathcal{K}_{2k^{-1}}(x,y)_{Rk}.$
Moreover, the arithmetic complexity of each Sinkhorn iteration is
$O(N^{3/2}).$ 
\end{thm}

\begin{proof}
As recalled in Example \ref{exa:The-assumptions-in} the cost function
$d(x,y)^{2}$ on the sphere satisfies the assumptions in Theorem \ref{thm:dynamic general}
(with $t=\infty)$ and Corollary \ref{cor:constr of approx in quite general setting }.

\emph{Step 1: The asymptotics \ref{eq:upper bound on minus c k} and
\ref{eq:expansion c k} are satisfied.}

By Theorem \ref{thm:heat kernel para} it is enough to observe that
the following basic estimate holds if $t=2k^{-1}$ and $W=Rk:$ 
\[
\left|\mathcal{K}_{t}(x,y)-\mathcal{K}_{t}(x,y)_{W}\right|\leq C_{\delta}e^{-2R^{2}k(1-\delta)}
\]
for any given $\delta\in]0,1[.$ To prove the estimate note that 
\[
\left|\mathcal{K}_{t}(x,y)-\mathcal{K}_{t}(x,y)_{W}\right|\leq\sum_{l>W}e^{-2k^{-1}l(l+1)}\frac{\left\Vert Y_{l}^{m}\right\Vert _{L^{\infty}}^{2}}{\left\Vert Y_{l}^{m}\right\Vert _{L^{2}}^{2}}\leq2Ck^{3}\sum_{l/k>R}e^{-2k(\frac{l}{k})^{2}}\frac{(l+1)^{2}}{k^{2}}\frac{1}{k},
\]
 using that the quotient involving $Y_{l}^{m}$ is dominated by $Cl$
(and that a given $l$ corresponds to $2l+1$ $m$s). Indeed, this
is a special case of the the universal $L^{2}-$estimates for an eigenfunction
$\Psi_{\lambda}$ of the Laplacian (with eigenvalue $\lambda^{2})$
on a general $n-$dimensional Riemannian manifold \cite{sog}, which
gives the growth factor $C\lambda^{n-1}.$ Finally, dominating the
Riemann Gaussian sum above with the integral of the function $e^{-ks^{2}}s^{2}$
over $[R,\infty[$ concludes the proof. 

\emph{Step 3: Complexity analysis}

Using formula \ref{eq:band limited heat kernel on sphere} gives
\[
a_{i}=\sum_{|m|\leq l\leq W}c_{m,l}\hat{b}_{l,m}Y_{l}^{m}(x_{i}),\,\,\,c_{m,l}=e^{-tl(l+1)}\left\Vert Y_{l}^{m}\right\Vert _{L^{2}}^{-1}\hat{b}_{l,m},\,\,\,\hat{b}_{l,m}:=\sum_{j=1}^{N}b_{j}\overline{Y_{l}^{m}(x_{j})})
\]
 $\hat{b}_{l,m}$ is the ``forward discrete spherical Fourier transform''
evaluated at $(l,m).$ Once it has been computed for all $(l,m)$
$a_{i}$ becomes an ``inverse discrete spherical Fourier transform''
(with coefficients $c_{m,l}\hat{b}_{l,m}$ ). By separation of variables,
both these transforms can be computed using a total of $O(k^{3})(=O(N^{3/2}))$
arithmetic operations (see the discussion after formula 1.9 in \cite{r-t}). 
\end{proof}
In the special case when the good Quasi-Monte Carlo system is defined
by an ``equi-angular'' colatitude-longitude rule of $N$ weighted
points on $S^{2}$ the arithmetic complexity of each iteration can
be reduced to $O(N)(\log N)^{2}$ operations, using a fast discrete
spherical Fourier transform \cite{h-r-k-m}. 

\subsubsection{\label{subsec:Application-to-the antenna}Application to the reflector
antenna problem}

The extensively studied \emph{far field reflector antenna problem}
appears when $X=Y=S^{n}$ is the $n-$dimensional sphere $S^{n},$
embedded as the unit-sphere in $\R^{n+1}$ and the cost function is
taken as $c(x,y):=-\log|x-y|$ \cite{waII,g-o}. Briefly, the problem
is to design a perfectly reflecting surface $\Sigma$ in $\R^{n+1}$
with the following property: when $\Sigma$ is illuminated with light
emitted from the origin with intensity $\mu\in\mathcal{P}(S^{n})$
the output reflected intensity becomes $\nu\in\mathcal{P}(S^{n})$
(of course, $n=2$ in the usual applications). Representing $\Sigma$
as a radial graph over $S^{n}:$ 
\[
\Sigma:=\{h(x)x\},\,\,x\in S^{n},
\]
 for a positive function $h$ on $S^{n}$ it follows from the reflection
law and conservation of energy that $h$ satisfies the following Monge-Ampère
type equation, expressed in terms of the covariant derivatives $\nabla_{i}$
in local orthonormal coordinates:

\begin{equation}
\frac{\det(-\nabla_{i}\nabla_{j}h+2h^{-1}\nabla_{i}h\nabla_{j}h+(h-\eta)\delta_{ij})}{((|\nabla h|^{2}+h^{2})/2h)^{n}}=e^{g(F_{h}(x))-f(x)},\label{eq:ant eq intro}
\end{equation}
 where $F_{h}(x)$ denotes the reflected direction of the ray emitted
in the direction $x$ (and $\mu$ and $\nu$ are represented as in
\ref{eq:mu and nu in terms of f and g intro}). The equation is also
supplemented with the ``second boundary value condition'' that $F_{h}$
maps the support of $\mu$ onto the support of $\nu.$ Assuming that
$f$ and $g$ are smooth there exists a smooth solution $h,$ which
is unique up to scaling (see \cite{c-g-h} and references therein).
\begin{thm}
Consider the two-sphere, discretized by a given good Quasi-Monte Carlo
(QMC) system. Let $K^{(k)}$ be the $N_{k}\times N_{k}$ matrix defined
by 
\[
K_{ij}^{(k)}=|x_{i}^{(k)}-x_{j}^{(k)}|^{k}
\]
Consider the Sinkhorn algorithm associated to $(p^{(k)},q^{(k)},K^{(k)}).$
There exists a positive constant $A_{0}$ such that for any $A>A_{0}$
the following holds: after $m_{k}=\left\lfloor Ak\log k\right\rfloor $
Sinkhorn iterations the function $h_{k}$ on $S^{n}$ defined by the
$k$ th root of $a_{k}:=a^{(k)}(m_{k})$ converges uniformly, as $k\rightarrow\infty,$
towards a solution $h$ of the antenna equation \ref{eq:ant eq intro}
satisfying the corresponding second boundary value condition. More
precisely, there exists a constant $C$ (depending on $A)$ such that
\[
\sup_{S_{N}}|h_{k}-h|\leq Ck^{-1}\log k,
\]
Moreover, the arithmetic complexity of each iteration is $O(N^{3/2})$
in general and $O(N)(\log N)^{2}$ in the case of an equi-angular
grid. 
\end{thm}

\begin{proof}
As recalled in Example \ref{exa:The-assumptions-in} the cost function
$-\log|x-y|$ on the sphere satisfies the assumptions in Corollary
\ref{cor:constr of approx in quite general setting }. For the complexity
analysis we first recall the general fact that any kernel $K^{(k)}(x,y)$
which is radial, i.e. only depends on $|x-y|,$ may be expressed as
\begin{equation}
K^{(k)}(x,y)=\sum_{l=1}^{\infty}C_{m,l}Y_{l}^{m}(x)\overline{Y_{l}^{m}(y)}\label{eq:form for K radial basis}
\end{equation}
for some positive constants $C_{m,l}$ (a proof will be given below).
By the argument in the proof of Step 2 in Theorem \ref{thm:sphere band limit heat},
it will be enough to show that $C_{m,l}=0$ when $l>k,$ i.e. that
$K^{(k)}(x,y)$ is already band-limited with $W=k.$ To this end we
follow the general approach in \cite{k-k-p}. First observe that when
$x$ and $y$ are in $S^{2}$ we can write $|x-y|^{2}=2(1-x\cdot y).$
Hence, $K_{k}(x,y)=2^{k}f^{(k)}(x\cdot y),$ where $f^{(k)}(s)=(1-s)^{k}$
for $s\in[-1,1].$ The Legendre polynomials $p_{l}(=p_{l}^{0}$ form
a base in the space of all polynomials of degree at most $k$ (which
is orthogonal wrt Lebesgue measure on $[0,1])$ and hence we can decompose
\[
2^{k}f^{(k)}=\sum_{l=1}^{k}c_{l}^{(k)}p_{l}.
\]
Formula \ref{eq:form for K radial basis} now follows from the classical
Spherical Harmonic (Legendre) addition theorem: 
\[
p_{l}(x\cdot y)=\frac{4\pi}{2l+1}\sum_{|m|\leq l}Y_{l}^{m}(x)\overline{Y_{l}^{m}(y)}.
\]
\end{proof}

\section{Outlook\label{sec:Outlook}}

\subsection{Generalized parabolic optimal transport and singularity formation }

Consider the setting in Section \ref{sec:Generalizations-and-outlook}
with a cost function $c$ satisfying the assumptions A1 and A2, but
assume for simplicity that $c$ is globally continuous (for example,
$c=d^{2}/2$ in the Riemannian setting). Recall that, given initial
data $u_{0}\in\mathcal{H}$ and volume forms $\mu$ and $\nu,$ the
parabolic equation \ref{eq:parabolic pde general setting} admits
a smooth solution $u_{t}$ on some maximal time-interval $[0,T[$
and the corresponding maps $F_{u_{t}}$ give an evolution of diffeomorphisms
of $X.$ It does not seem to be known whether $T=\infty,$ in general,
i.e. it could be that there are no solutions in $C^{\infty}(X\times]0,\infty[),$
in general. Still, using the corresponding iteration $u_{k}^{(m)}$
(say, defined with respect to $\mu_{k}=\mu$ and $\nu_{k}=\nu$) a
generalized notion of solution can be defined:
\begin{prop}
\label{prop:generalized parabolic}Given a $c-$convex function $u_{0},$
define the following curve $u_{t}$ of functions on $X,$ emanating
from $u_{0}:$ 
\[
u_{t}:=\sup\left\{ u_{k}^{(m)}:\,\,\,(m,k):\,\,m/k\rightarrow t,\,\,k\rightarrow\infty\right\} 
\]
 Then $u_{t}$ is $c-$convex for any fixed $t$ (and, in particular,
continuous) and there exists a constant $C$ such that $\sup_{X\times[0,\infty[}|u_{t}(x)|\leq C.$ 
\end{prop}

\begin{proof}
\emph{Step 1: there exists a constant such that $|u_{k}^{(m)}|\leq C.$ }

By the argument in Step 3 in the proof of Theorem \ref{thm:conv of u m in general setting}
we have 
\[
\mathcal{F}^{(k)}(u^{(k)})+\mathcal{L}^{(k)}(u_{0})\leq I_{\mu}(u_{m}^{(k)})\leq I_{\mu}(u_{0})
\]
By Lemma \ref{lem:dens prop} $\mathcal{L}^{(k)}(u_{0})\rightarrow-\int u_{0}^{c}\nu$
and by Theorem \ref{Thm:weak conv of fixed points towards optimal transport plans}
$\mathcal{F}^{(k)}(u_{m}^{(k)})\rightarrow\inf_{C^{0}(X)}J$ and hence
the lhs above is uniformly bounded in $k.$ Thus, there exists a constant
$C$ such that $-C\leq I_{\mu}(u_{m}^{(k)})\leq C.$ The proof of
Step 1 is now concluded by observing that there exist constants $A_{1}$
and $A_{2}$ such that, for any $c-$convex function $u,$ 
\[
\sup_{X}u\leq I_{\mu}(u)+A_{1},\,\,\,\inf_{X}u\geq I_{\mu}(u)-A_{2}.
\]
 Indeed, both functionals $f_{1}(u):=\sup_{X}u-I_{\mu}(u)$ and $f_{2}(u):=\inf_{X}u-I_{\mu}(u)$
are continuous on $C(X)$ and descend to $C(X)/\R.$ But the space
of $c-$convex functions is compact in $C(X)/\R$ (as is shown precisely
as in Lemma \ref{lem:compactness of function spaces}) and hence any
continuous functional on the space is uniformly bounded, which implies
the two inequalities above. 

\emph{Step 2: If $\{u_{\alpha}\}_{\alpha\in A}$ is a finite family
of $c-$convex functions, then $u:=\max\{u_{\alpha}\}_{\alpha\in A}$
is $c-$convex.}

It is enough to find a function $v\in C(X)$ such that $u=v^{c}.$
We will show that $v:=\min\{u_{\alpha}^{c}\}_{\alpha\in A}$ does
the job. To this end first observe that $u\mapsto u^{c}$ is order
preserving. Hence, $u_{\alpha}\leq u$ implies that $u_{\alpha}^{c}\geq u^{c},$
giving $v\geq u^{c}.$ Applying the $c-$Legendre transform again
thus gives $v^{c}\leq u^{cc}=u.$ To prove the reversed inequality
first observe that, by definition, $u_{\alpha}^{c}\geq v$ and hence
$u_{\alpha}=(u_{\alpha}^{c})^{c}\leq v^{c}.$ Finally, taking the
sup over all $\alpha$ proves the desired reversed inequality.

\emph{Step 3: Conclusion}

Denote by $\mathcal{K}_{t}$ the closure in $C(X)$ of the set $S_{t}$
of all $u_{m}^{(k)}$ such that $m/k\rightarrow t$ and $k\rightarrow\infty.$
By Step 1 and Lemma \ref{lem:compactness of function with k} $\mathcal{K}_{t}$
is compact. Let $u_{1},...,u_{m}$ be the limit points of $S_{t}.$
By the argument towards the end of Step 1 in the proof of Theorem
\ref{Thm:weak conv of fixed points towards optimal transport plans},
$u_{i}$ is $c-$convex. Hence, by Step 2, so is $u:=\max\{u_{i}\}.$ 
\end{proof}
The curve $u_{t}$ \ref{eq:parabolic pde general setting} is well-defined
for any probability measure $\mu$ and $\nu$ on compact topological
spaces $X$ and $Y$ and for any continuous cost function $c.$ Moreover,
if $\mu$ and $\nu$ are normalized volume forms on compact manifolds,
assumptions A1 and A2 hold and $u_{0}\in\mathcal{H},$ then, by Theorem
\ref{thm:conv of u m in general setting}, $u_{t}$ coincides with
the classical solution of the parabolic equation \ref{eq:parabolic pde general setting},
as long as such such a solution exists in $\mathcal{H},$ i.e. as
long as $F_{u_{t}}$ is a well-defined diffeomorphism. This makes
the curve $u_{t}$ a candidate for a solution to the problem posed
in \cite[Problem 9]{d-f} of defining some kind of weak solution to
the parabolic equation \ref{eq:parabolic pde general setting}, without
making assumptions on the MTW-tensor etc (as in Theorem \ref{thm:(Kim-Streets-Warren-)-Assume}).
The connection to the Sinkhorn algorithm also opens the possibility
of numerically exploring singularity formation of classical solutions
$u_{t}$ to the parabolic equation \ref{eq:parabolic pde general setting}
as $t\rightarrow T$ (the maximal existence time). As indicated in
\cite[Problem 9]{d-f} one could expect that the first derivatives
of a classical solution $u_{t}$ blow up along a subset $S$ of $X$
of measures zero as $t\rightarrow T$ (moreover, in the light of the
discussion in \cite[Problem 8]{d-f}, the subset $S$ might be expected
to be rectifiable and of Hausdorff codimension at least one). 

Finally, it may be illuminating to point out that, even if the construction
of the generalized solution $u_{t}$ may appear to be rather non-standard
from a PDE point of view it bears some similarities to the method
of ``vanishing viscosity'' for constructing solutions to PDEs by
adding small regularizing terms. This is reinforced by the interpretation
of the inverse of $k$ as an ``entropic regularization parameter''
discussed in the introduction of the paper (also note that the approximations
$u_{m_{k}}^{(k)}$ are smooth when the heat kernel is used, as in
Theorem \ref{thm:heat kernel para}). One is thus lead to ask whether,
under suitable regularity assumptions on $(\mu,\nu,c)$ the curve
$u_{t}$ is a viscosity solution of the parabolic PDE \cite{c-i-l}? 

\section{\label{sec:Appendix:-proof-of}Appendix A: proof of Lemma \ref{lem:discrete stat phase}
(discrete Laplace method)}

We start with the following elementary lemma:
\begin{lem}
\label{lem:local dens prop in model case}Let $h_{k}$ be a sequence
of continuous convex functions on the polydisc $D_{k}$ in $\R^{n}$
of radius $\log k$ centered at $0$ such that there exists a constant
$C$ such such that $|h_{k}(x)|+|\nabla_{e_{1}}\nabla_{e_{2}}h_{k}|(x)\leq Ce^{-|x|^{2}/C},$
where $e_{1}$ and $e_{2}$ are any two unit-vectors. Then there exists
a constant $C'$ (only depending on $C)$ such that
\[
\left|k^{-n/2}\sum_{x_{i}^{(k)}\in D_{k}\cap(k^{-1/2}\Z)^{n}}h_{k}(x_{i}^{(k)})-\int_{D_{k}}h_{k}dx\right|\leq C'/k
\]
\end{lem}

\begin{proof}
By restricting the integration to one variable at a time it is enough
to consider the case when $n=1.$ Fix $x_{i}^{(k)},$ which, by symmetry,
may be assumed non-negative. For any fixed $x$ in the interval $I_{k}(x_{i}^{(k)})$
centered at $x_{i}^{(k)},$ of length $k^{-1/2},$ Taylor expanding
$h_{k}$ gives 
\[
|h_{k}(x_{i}^{(k)})-h_{k}(x)-(x_{i}^{(k)}-x)h_{k}'(x_{i}^{(k)})|\leq k^{-1}Ce^{-(x_{i}^{(k)})^{2}/C}\leq Ck^{-1}e^{-(x-1/2k^{1/2})^{2}/C},
\]
 using that $e^{-x^{2}/C}$ is decreasing in the last step. By symmetry,
the integral over $I_{k}(x_{i}^{(k)})$ of the linear term $(x_{i}^{(k)}-x)h_{k}'(x_{i}^{(k)})$
vanishes, giving

\[
k^{-1/2}h_{k}(x_{i}^{(k)})=\int_{I_{k}(x_{i}^{(k)})}h_{k}(x_{i}^{(k)})dx=\int_{I_{k}(x_{i}^{(k)})}h_{k}(x)dx+\epsilon_{i}^{(k)},
\]
 where 
\[
|\epsilon_{i}^{(k)}|\leq Ck^{-1}\int_{I_{k}(x_{i}^{(k)})}e^{-(x-1/2k^{1/2})^{2}/C}dx
\]
Hence, summing over all points $x_{i}^{(k)}\in D_{k}\cap k^{-1/2}\Z$
except the end points and using that $|h_{k}(x_{i}^{(k)})|\leq Ce^{-(\log k)^{2}/C}\leq C'/k^{-1}$
at the end points gives

\[
\left|k^{-1/2}\sum_{x_{i}^{(k)}\in D_{k}\cap k^{-1/2}\Z}h_{k}(x_{i}^{(k)})-\int_{D_{k}}h_{k}dx\right|\leq C'k^{-1}+k^{-1}C''\int_{0\leq s\leq\log k}e^{-(s-1/2k^{1/2})^{2}/C}ds,
\]
 which concludes the proof. 
\end{proof}
In the sequel we will denote by $+$ the ordinary group structure
on $T^{n}$ and by $0$ the zero with respect to the group structure.
Without loss of generality we may as well assume that $\alpha(x_{0})=0$
and $h(x_{0})=1.$ We will denote by $O(k^{-1})$ any sequence of
functions satisfying $|O(k^{-1})|\leq Ck^{-1},$ for a constant $C$
depending on data as in the statement of the proposition to be proved.

\subsection*{Step 1: Localization to a polydisc $U_{k}$ of radius $k^{-1/2}\log k$
centered at $x_{0}$}

First fix a neighborhood $U$ of $x_{0}.$ Since $\alpha$ is assumed
lower-semicontinuous we have $\alpha(x)\geq\delta>0$ on $U$ and
hence 
\[
k^{n/2}\int_{T^{n}-U}e^{-k\alpha}h\delta_{\Lambda_{k}}\leq k^{n/2}\sup_{T^{n}}|h|e^{-k\delta}=k^{n/2}O(1)k^{-1},
\]
 using that $\delta_{\Lambda_{k}}$ is a probability measure. Next,
assume that $U$ is a small coordinate neighborhood of $x_{0}$ and
denote by $U_{k}$ the polydisc of radius $k^{-1}\log k$ centered
at $x_{0}.$ By assumption we may assume that $\alpha(x)\geq|x-x_{0}|^{2}/C$
on $U.$ Hence, 
\[
k^{n/2}\int_{U-U_{k}}e^{-k\alpha}h\delta_{\Lambda_{k}}\leq k^{n/2}\sup_{T^{n}}|h|e^{-(\log k)^{2}}=O(1)k^{-1}
\]

\subsection*{Step 2: \emph{The case when $x_{0}=0$ in $T^{n}$ }}

Introducing the notation $\alpha^{(k)}(x):=k\alpha(k^{-1/2}x)$ and
$h_{k}(x):=h(k^{-1/2}x)\exp(-\alpha^{(k)}(x))$ we can write 
\[
I_{k}:=k^{n/2}\int_{U_{k}}e^{-k\alpha}h\delta_{\Lambda_{k}}=k^{-n/2}\sum_{x_{i}^{(k)}\in D_{k}\cap(k^{-1/2}\Z)^{n}}h_{k}(x_{i}^{(k)})
\]
 Now, Taylor expanding $\alpha$ and denoting by $p^{(3)}$ the third
order term (i.e. defining a polynomial of homogeneous degree three)
gives, when $|x|\leq k^{1/2}/C$ (and, in particular, when $|x|\leq\log k)$
\[
\alpha^{(k)}(x)=Ax\cdot x/2+k^{-1/2}p^{(3)}+k^{-1}O(|x|^{4})
\]
Thus $h_{k}(x)$ may be Taylor expanded as follows
\[
h_{k}(x)=e^{-Ax\cdot x/2}\left(1+k^{-1/2}(q^{(1)}+p^{(3)})+k^{-1}O(|x|^{4})\right),\,\,\,\,q^{(1)}(x):=(\nabla h)(0)\cdot x
\]
Next note that $h_{k}(x)$ satisfies the assumptions of the previous
lemma. Indeed, by the chain rule $(\nabla_{e_{1}}\nabla_{e_{2}}\alpha^{(k)})(x)$
is, when $|x_{i}|\leq k^{1/2}/C,$ equal to the corresponding second
derivatives of $\alpha$ on $U,$ which are uniformly bounded on $U,$
by assumption. Hence, applying the lemma in question gives
\[
I_{k}=\int_{D_{k}}h_{k}dx+O(k^{-1})
\]
This shows that in the present discrete setting we get the same result,
up to the negligible error term $O(k^{-1}),$ as the ordinary Laplace
method of integration, which can hence be invoked to conclude. For
completeness we provide an alternative direct argument. By the expansion
above we have

\[
I_{k}=\int_{D_{k}}e^{-Ax\cdot x/2}\left(1+k^{-1/2}(q^{(1)}+p^{(3)})+k^{-1}O(|x|^{4})\right)
\]
Using the exponential decay of $e^{-Ax\cdot x/2}$ the integral may
be taken over all of $\R^{n},$ up to introducing an error term $O(k^{-\infty}).$
Hence computing the Gaussian integral concludes the proof, once one
has verified that the integrals over $q^{(1)}$ and $p^{(3)}$ vanish.
In the case when $A$ is the identity matrix the vanishing follows
directly from the fact that $q^{(1)}$ and $p^{(3)}$ are odd. In
the general case one first observes that the space of polynomials
of homogeneous degree $d$ on $\R^{n}$ is invariant under the action
of the space of invertible linear maps. Hence the problem reduces,
by a linear change of variables, to the previous case of an identity
matrix.

\subsection*{\emph{Step 3: The case of a general $x_{0}$ }}

Set $\tilde{\alpha}(x):=\alpha(x+x_{0})$ and $\tilde{f}(x):=f(x+x_{0})$
and decompose $x_{0}=m_{k}+r_{k}$ where $m_{k}\in\Lambda_{k}$ and
$|r_{k}|\leq1/k$ (where we have identified a small neighborhood of
$0$ in $T^{n},$ containing $r_{k}$ with $\R^{n}).$ Then we can
write 
\[
\int e^{-k\alpha}f\delta_{\Lambda_{k}}=\int e^{-k\tilde{\alpha}}\tilde{f}\delta_{(\Lambda_{k}-r_{k})}
\]
Indeed, for any function $h$ on $T^{n}$ we have, since $m_{k}\in\Lambda_{k},$
\[
\sum_{x_{i}\in\Lambda_{k}}h(x_{i}+m_{k}+r_{k})=\sum_{x_{i}}h(x_{i}+r_{k})
\]
Now, we note that the conclusion in the previous lemma remains true
when $\Lambda_{k}$ is replaced by the shifted set $\Lambda_{k}-r_{k}$
(with essentially the same proof) and hence we can conclude as before.

\section{Appendix B: Proof of Prop \ref{prop:existence and conv parab torus}
(the parabolic PDE on the torus)}

The main difference between Proposition \ref{prop:existence and conv parab torus},
concerning the $n-$dimensional torus $T^{n}$ and the corresponding
result for cost functions $c$ on general compact manifolds in \cite{k-s-w}
(see Theorem \ref{thm:(Kim-Streets-Warren-)-Assume}) is that in \cite{k-s-w}
it is assumed that the strong M-W-T condition holds, i.e. the Ma-Wang-Trudinger
tensor of $c$ is bounded from below by a \emph{strictly} positive
constant, while the Ma-Wang-Trudinger tensor vanishes identically
for the standard cost function $c$ on $T^{n}.$ The only place where
the strong M-W-T condition enters in \cite{k-s-w} is in the proof
of the $C^{2}-$estimate of $u_{t}.$ Here we will provide a proof
of the $C^{2}-$estimate in question on $T^{n},$ building on estimates
in \cite{ki}. To make the connection to the setting in \cite{ki}
we identify a solution $u_{t}$ of the parabolic flow appearing in
Prop \ref{prop:existence and conv parab torus} with a periodic function
in $\R^{n},$ i.e. $u_{t}$ is a $\Z^{n}-$invariant function in $\R^{n}.$
Using the notation in \cite{ki} the parabolic equation in question
may be expressed as 

\begin{equation}
\frac{\partial u_{t}(x)}{\partial t}=D[u_{t}],\,\,\,\,D[u_{t}]:=\log\det\left(\nabla^{2}u_{t}(x)-A\right)-\log B\left(x,\nabla u_{t}(x)\right)\label{eq:parab eq in app}
\end{equation}
with 
\[
A=-I,\,\,\,\log B\left(x,\nabla u(x)\right)=g\left(x+\nabla u(x)\right)-f(x).
\]
Recall that the initial data $u_{0}$ is assumed to be a periodic
function in $C^{4,\alpha}(\R^{n})$, which is strictly quasi-convex,
i.e. the matrix $(\nabla^{2}u_{0})(x)+I$ is positive definite for
all $x.$ This is precisely the parabolic equation appearing in \cite{ki}
in the case of the smooth cost function $c(x,y)=|x-y|^{2}/2$ in $\R^{n}\times\R^{n}.$
However, in \cite{ki} the equation is considered on a compact domain
$\Omega$ in $\R^{n}$ with additional boundary conditions, while
here we assume that $u_{t}$ is a periodic function on $\R^{n}.$
Denote by $L_{u}$ the linearization of the operator $D$ at $u.$
A direct calculation reveals that
\begin{equation}
L_{u}=\Delta_{w}-(\nabla g)\cdot\nabla,\label{eq:form for L u}
\end{equation}
 where $\Delta_{w}$ denotes the Laplacian operator defined wrt the
Hessian metric 
\begin{equation}
(w_{ij})=(\nabla^{2}u_{t})(x)+I,\label{eq:def of wij}
\end{equation}
i.e., denoting by $(w^{ij})$ the inverse matrix,
\[
\Delta_{w}=\sum_{i,j}w^{ij}\nabla_{i}\nabla_{j}
\]
and the gradient $\nabla g$ is evaluated at the point $x+\nabla u(x)$
and $(\nabla g)\cdot\nabla$ denotes the corresponding first order
differential operator. Note that $u$ is strictly quasi-convex iff
$L_{u}$ is an elliptic operator.

\subsection{The $C^{2}-$estimate}

First observe that 
\begin{equation}
|\nabla u_{t}|\leq\sqrt{n}.\label{eq:bound on gradient in append}
\end{equation}
This follows directly from the fact that $u_{t}$ is $c-$convex,
when viewed as a function on the torus $T^{n}$ (see the beginning
of Section \ref{subsec:The-torus-setting}).
\begin{lem}
\label{lem:C2}Let $u_{t}(x)$ be a solution to the parabolic equation
\ref{eq:parab eq in app} on $T^{n}\times[0,t_{max}[$ which is $C^{2}-$smooth
in $x$ and jointly $C^{1}-$smooth in $x$ and $t.$ There exists
a constant $C_{0},$ independent of $t,$ such that
\[
|\nabla^{2}u_{t}|\leq C_{0}.
\]
 More precisely, $C_{0}$ only depends on upper bounds on the $C^{2}-$norms
of $u_{0},f$ and $g$ and a strict positive lower bound on the eigenvalues
of the matrix $(I+\nabla^{2}u_{0}(x)).$ 
\end{lem}

\begin{proof}
We will adapt the proof of the interior $C^{2}-$estimate in \cite[Thm 10.1]{ki}
to the present periodic setting (the estimate in \cite[Thm 10.1]{ki}
is a parabolic version of the corresponding estimate in the elliptic
setting considered in \cite{t-w}). Following the convention used
in \cite{ki} we will omit the sum sign $\sum$ from the formulas
below. Moreover, $C$ will denote a constant only depending on an
upper bound on $|\nabla u_{t}|$ and the first and second order derivatives
of $\log B$ (i.e. of $g$ and $f),$ whose precise value may change
from line to line. By the a priori estimate \ref{eq:bound on gradient in append}
these bounds are thus independent of $t.$ Given a unit-vector $\xi\in S^{n-1}$
and $a>0$ we consider the function 
\[
v(x,t)=w_{\xi\xi}(x)+a|\nabla u_{t}(x)|^{2},\,\,\,w_{\xi\xi}(x):=w_{ij}(x)\xi_{i}\xi_{j},
\]
which coincides with the function defined in the beginning of the
proof of \cite[Thm 10.1]{ki}. Fix a point $x_{0}\in\R^{n}$ and take
$\xi$ to be a unit-vector $\xi_{0}$ maximizing $w_{ij}(x_{0})\xi_{i}\xi_{j}.$ 

\emph{Step 1:} The following inequality holds at the fixed point $x_{0}$
if $w_{ij}(x_{0})\xi_{i}\xi_{j}\geq1$
\[
(-\frac{\partial}{\partial t}+L_{u_{t}})[v]\geq(2a-C)w_{ii}-Cw^{ii}-Ca,
\]

This inequality is shown in the course of the proof of \cite[Thm 10.1]{ki},
only using that the MTW-tensor is non-negative (in the present case
where the matrix $A$ is constant the MTW-tensor vanishes identically).
Indeed, first it is shown, by differentiating the parabolic equation
for $u_{t},$ that 
\[
(-\frac{\partial}{\partial t}+L_{u_{t}})[v]\geq\frac{w^{il}w^{jk}\nabla_{\xi}w_{ij}\nabla_{\xi}w_{lk}}{(w_{\xi\xi})^{2}}-\frac{w^{ij}\nabla_{i}w_{\xi\xi}\nabla_{j}w_{\xi\xi}}{w_{\xi\xi}}+(2a-C)w_{ii}-Cw^{ii}-Ca.
\]
 Then it is shown that the sum of the first two terms is bounded from
below by a constant times $-w^{ii}.$ But, in fact, in the present
case, where the matrix $A$ is constant, the sum in question is actually
non-negative (see the proof of \cite[Thm 3.3]{l-t-u} Theorem 3.3
for a different proof of this fact). 

\emph{Step 2: }The following inequality holds at any point $x$
\[
L_{u}[-u]\geq w^{ii}-C|\nabla u|.
\]

To see this first observe that the following identity holds: 
\[
-\Delta_{w}u+n=w^{ii},
\]
 This is seen by multiplying the matrix identity \ref{eq:def of wij}
with the inverse $(w^{ij})$ of $(w_{ij})$ and taking the trace.
Hence, 
\[
L[-u]\geq w^{ii}+(\nabla g)\cdot\nabla u
\]
which concludes the proof of Step 2.

\emph{Step 3: }Conclusion using the parabolic maximum principle

Now fix $b>0$ and consider the function 
\[
v_{a,\kappa}(x,t)=w_{ij}\xi_{i}\xi_{j}+a|\nabla u_{t}|^{2}-b(u_{t}-u_{t}(0)).
\]

Combining the previous steps we get, at the fixed point $x_{0}:$
\[
\mathcal{L}_{a,\kappa}[v]\geq(2a-C)w_{ii}+(b-C)w^{ii}-C(a+b).
\]

Take the parameters $a$ and $b$ sufficiently large to ensure that
$2a-C>0$ and $b-C>0$ (this condition is independent of the choice
of fixed point $x_{0}).$ Now fix $T<t_{max}$ and consider the smooth
function $v_{a,\kappa}$ on $\R^{n}\times[0,T]\times S^{n-1}.$ Since
$v_{a,\kappa}$ is periodic in the $x-$variable its sup is attained
at some $(x_{0},t_{0},\xi_{0})$ in $[0,1]^{n}\times[0,T]\times S^{n-1}.$
Note that $\xi_{0}$ maximizes $w_{ij}(x_{0})\xi_{i}\xi_{j}.$ We
may assume that $w_{ij}(x_{0})\xi_{i}\xi_{j}\geq1$ (otherwise we
are already done). Now, if $t_{0}>0$ then 
\[
0+0\geq(-\frac{\partial}{\partial t}+L)[v_{a,\kappa}]
\]
 at $(x_{0},t_{0}).$ But then the inequality in Step 1 implies an
upper bound on $w_{ii}$ at $(x_{0},t_{0})$ only depending on $C.$
Finally, since $a|\nabla u_{t}|^{2}-b(u_{t}-u_{t}(0))$ is uniformly
bounded by a constant this shows that the sup of $w_{ii}$ over $\R^{n}\times[0,T]$
is also uniformly bounded by a constant. Finally, if $t_{0}=0$ then
we get an upper bound on $w_{ii}$ in terms of the Hessian of $u_{0}.$ 
\end{proof}

\subsection{Conclusion of the proof of Prop \ref{prop:existence and conv parab torus}}

The rest of the argument proceeds as in \cite{k-s-w}, but for the
convenience of the reader we provide some details, highlighting the
dependence on the regularity assumptions on the data $u_{0},f$ and
$g.$ 

\subsubsection{Short-time existence.}

First assume that $u_{0}\in C^{2,\alpha}(T^{n})$ and that $u$ is
strictly quasi-convex. Then the linearization $L_{u_{0}}$ is uniformly
elliptic and hence standard short-time existence results for parabolic
equations imply that there exists a maximal $t_{max}>0$ and a unique
solution $u(x,t)$ to the equation \ref{eq:parab eq in app} on $T^{n}\times[0,t_{max}[$
with the property that $u_{t}$ is in $C^{2,\alpha}(X)$ and the corresponding
Hölder derivatives of $u_{t}$ are continuous as $t\rightarrow0.$
In particular, the norms $\left\Vert u_{t}\right\Vert _{C^{2,\alpha}(T^{n})}$
are uniformly bounded for $t\in[0,T]$ for any given $T<t_{max}.$
This follows, for example, from \cite[Main Thm 1]{ba} or \cite[Thm 1.2]{hua}.
Next given two unit-vectors $e_{1}$ and $e_{2}$ denote by $\nabla_{1}$
and $\nabla_{2}$ the corresponding derivations. Applying $\nabla_{1}$
to the equation \ref{eq:parab eq in app} gives
\begin{equation}
\frac{\partial(\nabla_{1}u_{t})}{\partial t}=L_{u_{t}}[\nabla_{1}u_{t}]-\nabla_{1}f.\label{eq:eq for nabla u}
\end{equation}
Next assume that $u_{0}\in C^{4,\alpha}(T^{n}).$ Since $L_{u_{t}}$
is uniformly elliptic when $t\leq T<t_{max}$ and the coefficients
of $L_{u_{t}}$ are Hölder continuous (with exponent $\alpha)$ global
linear parabolic Shauder estimates \cite[Thm 2.3]{hua} yield, since
$e_{1}$ was arbitrary, that $\nabla u_{t}\in C^{2,\alpha}(T^{n})$
and that $\left\Vert \nabla u_{t}\right\Vert _{C^{2,\alpha}(T^{n})}$
is uniformly bounded for $t\in[0,T].$ Now, applying $\nabla_{2}$
to the equation \ref{eq:eq for nabla u} gives
\[
\frac{\partial(\nabla_{2}\nabla_{1}u_{t})}{\partial t}=L_{u_{t}}[\nabla_{2}\nabla_{1}u_{t}]-F_{t},\,\,\,F_{t}:=\sum_{i,j}(\nabla_{2}w^{ij})\nabla_{i}\nabla_{j}(\nabla_{1}u_{t})-\nabla_{2}(\nabla g)\cdot\nabla(\nabla_{1}u_{t})-\nabla_{2}\nabla_{1}f.
\]
By assumption, $\nabla^{2}g$ and $\nabla^{2}f$ are Hölder continuous
and hence (also using that $\nabla u_{t}\in C^{2,\alpha}(T^{n}))$
the function $F_{t}$ is Hölder continuous on $T^{n}$ for some positive
Hölder exponent $\alpha'(\leq\alpha).$ We can thus apply the global
linear parabolic Shauder estimates again (now to the linear operator
$L_{t}[\cdot]-F_{t}$) and deduce that the norms $\left\Vert \nabla^{2}u_{t}\right\Vert _{C^{2,\alpha'}(T^{n})}$
are uniformly bounded for $t\in[0,T].$ 

\subsubsection{$t_{0}-$independent a priori estimates}

Next note that there exist constant $C$ (independent of $x$ and
$t$) such that
\begin{equation}
C^{-1}I\leq(w_{ij}(x,t))\leq CI.\label{eq:bound on wij in concl}
\end{equation}
Indeed, the upper bound is the content of Lemma \ref{lem:C2} and
to prove the lower bound first note that the time derivative $\dot{u}_{t}$
of $u_{t}$ solves the linear parabolic equation
\begin{equation}
\frac{\partial U_{t}}{\partial t}=L_{u_{t}}[U_{t}],\label{eq:linear PDE}
\end{equation}
Hence, by the parabolic maximum principle
\[
\sup_{T^{n}\times[0,t_{max}[}|\dot{u}_{t}|\leq\sup_{T^{n}}|\dot{u}_{0}|.
\]
Plugging this bound into the equation \ref{eq:parab eq in app} this
means that the determinant of $w_{ij}(x,t)$ is uniformly bounded
from below (by a constant only depending on the sup norm of $g$ and
the Hessian of $u_{0}).$ Combined with the upper bound in \ref{eq:bound on wij in concl}
this implies the lower bound in \ref{eq:bound on wij in concl}. Next
note that the bound \ref{eq:bound on wij in concl} says that the
linearized operator $L_{u_{t}}$ is uniformly elliptic (with constants
independent of $t).$ Hence, the Krylov-Safanov theory for fully non-linear
parabolic equations yield uniform interior $C^{2,\alpha}(T^{n})-$estimates
for $u_{t}$ for some $\alpha>0$ (i.e. they hold for $t\geq\epsilon>0$
with constants only depending on $\epsilon,$ the $C^{2}-$norms of
$u_{t}$ and $f$ and $g,$ i.e. on the previous uniform constant
$C).$ It follows that $t_{max}=\infty$ (otherwise we could restart
the flow again). Then, by the argument in the previous section, using
parabolic Shauder estimates, we get a bound 
\[
\left\Vert \nabla^{2}u_{t}\right\Vert _{C^{2,\alpha'}(T^{n})}\leq C'
\]
 for a constant independent of $t\in[0,\infty[.$ In particular, such
a bound holds for $\left\Vert u_{t}\right\Vert _{C^{4}(T^{n})},$
which concludes the proof of the first point in Prop \ref{prop:existence and conv parab torus}.

\subsubsection{Exponential convergence}

Finally, the exponential convergence follows from the generalization
of the Li-Yau Harnack inequality in \cite[Thm 5.2, Cor 5.3]{h-s-w}.
Briefly, by \cite[Thm 5.2, Cor 5.3]{h-s-w} the a priori bounds in
the first point of Prop \ref{prop:existence and conv parab torus}
imply that there exists a constant $C'$ such that any positive solution
$U_{t}$ to the linear parabolic equation \ref{eq:linear PDE} satisfies
\[
\sup_{T^{n}}U_{t+1/2}\leq C'\inf_{T^{n}}U_{t}
\]
when $n\geq3.$ Using a standard induction argument this implies an
exponential decay of the sup-norm of the oscillation of $\dot{u},$
which in turn implies the exponential convergence in Prop \ref{prop:existence and conv parab torus}.
The assumption that $n\geq3$ in \cite[Cor 5.3]{h-s-w} is then bypassed
by taking a product of $T^{n}$ with $T^{2},$ just as in \cite[Section 7.1.2]{k-s-w}.

\end{document}